\documentclass[11pt]{amsart}
\newcommand{\preprint}[1]{}

\newcommand{\hide}[1]{}

\usepackage{amssymb}
\usepackage{amsbsy}
\usepackage{amscd}
\usepackage{amsmath}
\usepackage{amsthm}
\usepackage{enumerate}
\input xy
\xyoption{all}

\numberwithin{equation}{section}

\theoremstyle{plain}
\newtheorem{thm}{Theorem}[section]
\newtheorem{prop}[thm]{Proposition}

\newtheorem{conj}[thm]{Conjecture}
\newtheorem{cor}[thm]{Corollary}
\newtheorem{lem}[thm]{Lemma}

\theoremstyle{definition}
\newtheorem{defi}[thm]{Definition}

\theoremstyle{remark}
\newtheorem{example}[thm]{Example}
\newtheorem{question}[thm]{Question}
\newtheorem{rem}[thm]{Remark}

\newtheorem{condition}[thm]{Condition}

\newcommand{\Def}{{\mathcal D}ef}
\newcommand{\A}{{\mathcal A}}

\newcommand{\CC}{{\mathbb C}}
\newcommand{\D}{{\mathcal D}}
\newcommand{\gd}{{\Delta}}
\newcommand{\E}{{\mathcal E}}
\newcommand{\F}{{\mathcal F}}
\newcommand{\G}{{\mathcal G}}
\renewcommand{\H}{{\mathcal H}}
\newcommand{\HH}{{\mathbb H}}

\newcommand{\gL}{{\Lambda}}
\newcommand{\M}{{\mathcal M}}
\newcommand{\fm}{{\mathfrak m}}
\newcommand{\fM}{{\mathfrak M}}
\newcommand{\tfM}{\widetilde{{\mathfrak M}}}
\newcommand{\hfM}{\overline{\fM}}

\newcommand{\ko}{{\mathcal O}}
\renewcommand{\P}{{\mathcal P}}
\newcommand{\PP}{{\mathbb P}}

\newcommand{\RR}{{\mathbb R}}

\newcommand{\U}{{\mathcal U}}
\newcommand{\V}{{\mathcal V}}
\newcommand{\W}{{\mathcal W}}
\newcommand{\X}{{\mathcal X}}
\newcommand{\Y}{{\mathcal Y}}
\newcommand{\Z}{{\mathcal Z}}
\newcommand{\RealNumbers}{{\mathbb R}}
\newcommand{\Integers}{{\mathbb Z}}
\newcommand{\ComplexNumbers}{{\mathbb C}}
\newcommand{\RationalNumbers}{{\mathbb Q}}
\newcommand{\LieAlg}[1]{{\mathfrak #1}}

\newcommand{\IsomRightArrow}{\stackrel{\cong}{\rightarrow}}

\newcommand{\RightArrowOf}[1]{\stackrel{#1}{\rightarrow}}

\newcommand{\LongRightArrowOf}[1]{\stackrel{#1}{\longrightarrow}}

\newcommand{\StructureSheaf}[1]{{\mathcal O}_{#1}}
\newcommand{\EndProof}{\hfill  $\Box$}
\newcommand{\restricted}[2]{#1_{\mid_{#2}}}

\newcommand{\rank}{{\rm rank}}
\newcommand{\coker}{{\rm coker}}

\newcommand{\Pic}{{\rm Pic}}
\newcommand{\Sym}{{\rm Sym}}
\newcommand{\Ext}{{\rm Ext}}

\newcommand{\Hom}{{\rm Hom}}
\newcommand{\Aut}{{\rm Aut}}
\newcommand{\End}{{\rm End}}

\newcommand{\SheafHom}{{\mathcal H}om}
\newcommand{\SheafEnd}{{\mathcal E}nd}

\newcommand{\SheafExt}{{\mathcal E}xt}
\newcommand{\SheafTor}{{\mathcal T}or}

\newcommand{\wt}[1]{\widetilde{#1}}
\newcommand{\Wedge}[1]{\stackrel{#1}{\wedge}}

\newcommand{\depth}{{\rm depth}}

\begin{document}
\title{A global Torelli theorem for rigid hyperholomorphic sheaves}
\author{Eyal Markman}
\address{Department of Mathematics and Statistics, 
University of Massachusetts, Amherst, MA 01003, USA}
\email{markman@math.umass.edu}
\author{Sukhendu Mehrotra}
\address{Chennai Mathematical Institute, 
H1 SIPCOT IT Park, Siruseri, Kelambakkam 603103, INDIA}
\email{mehrotra@cmi.ac.in}

\date{\today}

\begin{abstract}
Let $X$ be an irreducible holomorphic symplectic manifold of
$K3^{[n]}$ deformation type, $n\geq 2$. There  exists over $X\times X$ a rank $2n-2$
rigid and stable reflexive sheaf $A$ of Azumaya algebras, constructed in 
\cite{markman-hodge}, such that the pair $(X\times X,A)$ is deformation equivalent 
to the following pair $(M\times M,A_M)$.
The manifold $M$ is a smooth and compact moduli space of stable sheaves
over a $K3$ surface $S$, and $A_M$ is the reflexive sheaf, whose 
fiber, over a pair
$(F_1,F_2)\in M\times M$ of non-isomorphic stable sheaves $F_1$ and $F_2$, is $\End\left(\Ext^1_S(F_1,F_2)\right)$. 
We prove in this paper the following uniqueness result.

Let $A_1$ and $A_2$ be two rigid Azumaya algebras over $X\times X$, which are $(\omega,\omega)$-slope-stable
with respect to the same K\"{a}hler class $\omega$ on $X$, and
such that $(X\times X,A_i)$ is deformation equivalent to $(M\times M,A_M)$,
for $i=1,2$. Assume that the singularities of $A_1$ and $A_2$ along the diagonal are of the same infinitesimal type prescribed by $A_M$.
Then $A_1$ is isomorphic to $A_2$ or $A_2^*$. Furthermore, if $Pic(X)$ is trivial,
or  cyclic generated by a class of non-negative Beauville-Bogomolov-Fujiki degree, then 
there exists a {\em unique} pair $\{A,A^*\}$, such that $(X\times X,A)$ is 
deformation equivalent to $(M\times M,A_M)$, and $A$ is $(\omega,\omega)$-slope-stable
with respect to every K\"{a}hler class $\omega$ on $X$. 

The above is the main example of a more general 
global Torelli theorem proven. 
The result is used in the authors forthcoming work on 
generalized deformations of $K3$ surfaces.
\end{abstract}

\maketitle
\tableofcontents

%
\section{Introduction}
\label{sec-introduction}

An {\em irreducible holomorphic symplectic manifold} is a simply connected compact K\"{a}hler manifold $X$,
such that $H^0(X,\Wedge{2}T^*X)$ is spanned by an everywhere non-degenerate holomorphic $2$-form.
The second cohomology $H^2(X,\Integers)$ admits a symmetric non-degenerate integral bi-linear
pairing of signature $(3,b_2(X)-3)$, called the {\em Beauville-Bogomolov-Fujiki} pairing
\cite{beauville-varieties-with-zero-c-1}. The pairing is positive on the K\"{a}hler cone, and 
normalized, so that $\gcd\{(\alpha,\beta) \ : \ \alpha,\beta\in H^2(X,\Integers)\}=1.$ 
The pairing is invariant under deformations of the complex structure of $X$, and is thus monodromy invariant.
Given a K\"{a}hler $K3$ surface $S$, the Hilbert scheme (or Douady space) $S^{[n]}$ of
length $n$ subschemes of $S$ is an example of an irreducible holomorphic symplectic manifold. 
An irreducible holomorphic symplectic manifold is said to be of {\em $K3^{[n]}$-type},
if it is deformation equivalent to $S^{[n]}$, for some (hence any) K\"{a}hler $K3$ surface $S$.

Let $X_0$ be an irreducible holomorphic symplectic manifold.
Fix a lattice $\Lambda$ isometric to $H^2(X_0,\Integers)$, endowed with its 
Beauville-Bogomolov-Fujiki pairing. Let $\fM_\Lambda$ be the moduli space of marked 
irreducible holomorphic symplectic manifolds with integral second cohomology isometric to $\Lambda$.
A point of $\fM_\Lambda$ parametrizes  an isomorphism class
of a pair $(X,\eta)$, where $\eta:H^2(X,\Integers)\rightarrow \Lambda$ is an isometry \cite{huybrects-basic-results}.

Fix a positive integer $d$. 
In section \ref{sec-existence-of-coarse-moduli-of-triples}
we construct  a coarse moduli space $\tfM_\Lambda$,
consisting of isomorphism classes of triples $(X,\eta,E)$, 
where $(X,\eta)$ is a marked pair as above and $E$ is a  reflexive sheaf 
of Azumaya algebras over the $d$-th Cartesian product $X^d$, which is
stable with respect to some K\"{a}hler class, 
$c_2(E)$ is invariant under the diagonal action of 
a finite index subgroup of the monodromy group of $X$, and $E$
satisfies the technical open Condition \ref{cond-open-and-closed}.
We introduce next the necessary background needed for the statement of 
Condition \ref{cond-open-and-closed}. The main results are then stated using that condition.

\begin{defi}
\label{def-reflexive-sheaf-of-Azumaya-algebras}
A {\em reflexive sheaf of Azumaya\footnote{Caution: 
The standard definition of a sheaf of
Azumaya $\StructureSheaf{X}$-algebras assumes that $E$ is a locally free
$\StructureSheaf{X}$-module, while we assume only that it is reflexive.} 
$\StructureSheaf{X}$-algebras of rank $r$} over a 
K\"{a}hler manifold $X$
is a sheaf $E$ of reflexive coherent $\StructureSheaf{X}$-modules, with a 
global section $1_E$, and an associative multiplication 
$m:E\otimes E\rightarrow E$ with identity $1_E$, admitting an open covering
$\{U_\alpha\}$ of $X$, and an isomorphism 
$\eta_\alpha:\restricted{E}{U_\alpha}\rightarrow \SheafEnd(F_\alpha)$
of unital associative algebras,
for some reflexive sheaf $F_\alpha$ of rank $r$ over each $U_\alpha$.
\end{defi}

From now on the term a {\em sheaf of Azumaya algebras} will mean
a reflexive sheaf of Azumaya $\StructureSheaf{X}$-algebras.
A subsheaf $P$ of a sheaf $E$ of Azumaya algebras is a 
{\em sheaf of maximal parabolic subalgebras}
if, in the notation of Definition \ref{def-reflexive-sheaf-of-Azumaya-algebras},
$\eta_\alpha(\restricted{P}{U_\alpha})$ is the sheaf of subalgebras of $\SheafEnd(F_\alpha)$ 
leaving invariant a non-zero subsheaf $F'_\alpha\subset F_\alpha$ of lower rank, for all $\alpha$.  
Let $\pi_i:X^k\rightarrow X$ be the projection onto the $i$-th factor.  

\begin{defi}
\label{def-slope-stability}
Given a sheaf $E$ of rank $r$ Azumaya algebras over $X^d$ and 
a K\"{a}hler class $\omega$ on $X$,
we say that $E$ is {\em $\omega$-slope-stable}, 
if every non-trivial subsheaf of maximal parabolic subalgebras of $E$ has negative
$(\sum_{i=1}^d\pi_i^*\omega)$-slope, 
and the rank $r^2$ sheaf $E$ is 
$(\sum_{i=1}^d\pi_i^*\omega)$-slope-polystable.
\end{defi}

Let $M$ be a complex manifold and $E$ a rank $r$ sheaf of reflexive Azumaya algebras over $M$.
Let $U\subset M$ be the open subset where $E$ is locally free.
$E$ corresponds to a class in
$H^1(U,PGL(r))$.
The connecting homomorphism of the short exact sequence 
$1\rightarrow \mu_r\rightarrow SL(r)\rightarrow PGL(r)\rightarrow 1$ associates to $E$ a characteristic class in
$H^2(U,\mu_r)$, which is isomorphic to $H^2(M,\mu_r)$, since the singular locus has co-dimension
$\geq 3$. We will refer to this class as the {\em characteristic class of $E$ in $H^2(M,\mu_r)$}.
The image of this class in $H^*(M,\StructureSheaf{M}^*)$ is called the {\em Brauer class}
of $E$.

\begin{prop}
\label{prop-maximally-twisted-sheaf-is-slope-stable}
\cite[Prop. 7.8]{markman-hodge}
Let $E$ be a sheaf of Azumaya algebras over a K\"{a}hler manifold $M$.
If the order of the Brauer class of $E$ in $H^2(M,\StructureSheaf{M}^*)$ is equal to the rank of $E$, 
then $E$ is $\omega$-slope-stable with respect to every K\"{a}hler class $\omega$ on $M$.
\end{prop}

Let $r$ be a positive integer and 
$\tilde{\theta}$ an element of order $r$ in $\Lambda^{\oplus d}/r\Lambda^{\oplus d}$.
Denote by $\mu_r\subset \ComplexNumbers^*$ the group of $r$-th roots of unity.
Let $\iota:H^2(X^d,\mu_r)\rightarrow H^2(X^d,\StructureSheaf{X^d}^*)$ be the natural homomorphism.
Given a marked pair $(X,\eta)$, let 
$\bar{\eta}:H^2(X^d,\Integers/r\Integers)\rightarrow \Lambda^{\oplus d}/r\Lambda^{\oplus d}$ be the isomorphism
induced by the marking $\eta$. 
Denote by 
\begin{equation}
\label{eq-theta-X-eta}
\theta_{(X,\eta)}
\end{equation}
the class $\iota\left[\exp\left(\frac{-2\pi i}{r}\bar{\eta}^{-1}(\tilde{\theta})\right)\right]$ in 
$H^2(X^d,\StructureSheaf{X^d}^*)$. 

\begin{rem}
\label{rem-trivial-picard-implies-Brauer-class-has-order-r}
Note that $\theta_{(X,\eta)}$ has order $r$, if $\Pic(X)$ is trivial. In that case any rank $r$ Azumaya algebra with Brauer 
class $\theta(X,\eta)$ is slope-stable with respect to every K\"{a}hler class on $X$,
by Proposition \ref{prop-maximally-twisted-sheaf-is-slope-stable}.
\end{rem}



\begin{defi}
\label{def-monodromy}
{\rm
The {\em monodromy group} $Mon(X)$ of  $X$
is the subgroup, of the automorphism group of the cohomology ring $H^*(X,\Integers)$,
generated by monodromy operators $g$ of families $\X \rightarrow B$ 
(which may depend on $g$) of irreducible holomorphic symplectic manifolds
deforming $X$. 
}
\end{defi}

Let $E$ be a sheaf of Azumaya algebras over a complex manifold $M$ and $F$ a 
(possibly twisted) reflexive sheaf, such that
$E$ is isomorphic to $\SheafEnd(F)$. Let
$\Ext^1_0(F,F)$ be the kernel of the trace homomorphism
$\Ext^1(F,F)\rightarrow H^1(M,\StructureSheaf{M})$.
Then $\Ext^1_0(F,F)$ parametrizes infinitesimal deformations of $E$ as a sheaf of Azumaya algebras over $M$.
Set $T^1_M(E):=\Ext^1_0(F,F)$.

\begin{condition}
\label{cond-open-and-closed}
Let $X$ be an irreducible holomorphic symplectic manifold and 
$E$ a sheaf of rank $r$ Azumaya algebras over $X^d$.
\begin{enumerate}
\item
\label{cond-item-characteristic-class}
The Chern class $c_2(E)\in H^{2,2}(X^d,\Integers)$ is invariant under 
the diagonal action of a finite index subgroup of $Mon(X)$. 
\item
\label{cond-item-Brauer-class-is-theta-X-eta}
The characteristic class of $E$ in $H^2(X^d,\mu_r)$ is 
$\exp\left(\frac{-2\pi i}{r}\bar{\eta}^{-1}(\tilde{\theta})\right)$, for some marking $\eta$.
In particular, the Brauer class of $E$ is $\theta_{(X,\eta)}$, given in Equation (\ref{eq-theta-X-eta}). 
\item
\label{cond-item-infinitesimally-rigid}
$T^1_{X^d}(E)=0$, 
so that the Azumaya algebra $E$ is {\em infinitesimally rigid} over the fixed $X^d$.
\item
\label{condition-stability}
$E$ is $\omega$-slope-stable with respect to some K\"{a}hler class $\omega$ on $X$.
\item
\label{cond-item-same-homogeneous-bundle}
If $d=1$, then $E$ is locally free. 
If $d>1$ then $E$ is locally free away from the diagonal image of $X$ in $X^d$. 
If $E$ is not locally free along the diagonal, then its singularity along the diagonal
satisfies the infinitesimal constraints listed in Section 
\ref{sec-infinitesimal-constraints-on-singularities-of-a-reflexive-sheaf}.
\end{enumerate}
\end{condition}

\hide{
Let $E$ be a sheaf of Azumaya algebras over $X$. 
Denote by $T^1_{X,lt}(E)$ the kernel of the trace homomorphism
$H^1(X,E)\rightarrow H^1(X,\StructureSheaf{X})$. 
Then $T^1_{X,lt}(E)$ parametrizes infinitesimal locally trivial deformations of $E$.
Indeed, the infinitesimal deformations of $E$ are equal to those of a reflexive twisted sheaf 
$F$ such that $E=\SheafEnd(F)$, modulo the $\Pic^0(X)$-action.
The locally trivial infinitesimal  deformations of $F$ are parametrized by 
the kernel $H^1(\SheafEnd(F))$ 
of the homomorphism $\Ext^1(F,F)\rightarrow H^0(\SheafExt^1(F,F))$.
Hence, the locally trivial infinitesimal  deformations of $E$ are parametrized by the 
kernel $T^1_{X,lt}(E)$ of the homomorphism 
$
\Ext^1(F,F)\rightarrow H^0(\SheafExt^1(F,F))\oplus H^1(X,\StructureSheaf{X}).
$
If $H^{0,1}(X)=0$, then $T^1_{X,lt}(E)=H^1(X,E).$
}

Let $E$ be a $\omega$-slope-stable 
sheaf of Azumaya algebras over $X^d$ satisfying Condition
\ref{cond-open-and-closed}. Parts
(\ref{cond-item-characteristic-class}) and
(\ref{condition-stability}) of 
Condition \ref{cond-open-and-closed} imply that the sheaf $E$
is  {\em  $\omega$-stable hyperholomorphic} in the sense of  
Verbitsky \cite{verbitsky-coherent-sheaves}. Associated to the K\"{a}hler class $\omega$ in
part (\ref{condition-stability}) above is 
a twistor family $\pi:\X\rightarrow \PP^1_\omega$ over $\PP^1$ and a 
reflexive sheaf $\E$ of Azumaya algebras over the $d$-th fiber product
$\X^d_\pi$
\cite[Lemma 6.14]{verbitsky-coherent-sheaves,markman-hodge},
which is  flat over $\PP^1_\omega$, by Proposition \ref{prop-flatness-of-the-twistor-deformation}. 
Part (\ref{cond-item-infinitesimally-rigid}) of 
Condition \ref{cond-open-and-closed} implies that $H^1(X^d,E)=0$.
Denote by $\E_t$, $t\in \PP^1_\omega$,
the restriction of $\E$ to the fiber $\X_t^d$ of $\X^d_\pi$ over $t$.

\begin{conj}
\label{conj-Ext-1-is-constant}
$H^1(\X_t^d,\E_t)=(0),$
for all $t\in \PP^1_\omega$.
\end{conj}

Verbitsky proved the above conjecture in case $E$ is locally free
\cite[Cor. 8.1]{verbitsky-hyperholomorphic}.
In the reflexive sheaf case it is known that the moduli space of $\E_t$
is zero dimensional, i.e., if non-zero first order infinitesimal deformations of $\E_t$ exist, then they must be obstructed
\cite[Theorem 3.27]{kaledin-verbitski-book}.

The main results of this note are the following two theorems. 

\begin{thm} 
\label{thm-existence-introduction}
(Theorem \ref{thm-existence}) Assume Conjecture \ref{conj-Ext-1-is-constant}.
There exists a coarse moduli space $\tfM_{\Lambda,\tilde{\theta}}$, which is a complex non-Hausdorff manifold, 
parametrizing isomorphism classes of triples $(X,\eta,E)$ satisfying Condition
\ref{cond-open-and-closed}. 
\end{thm}

We denote by $\tfM_{\Lambda}$ the union of $\tfM_{\Lambda,\tilde{\theta}}$, as $\tilde{\theta}$
varies over all elements of order $r$ in $\Lambda^{\oplus d}/r\Lambda^{\oplus d}$. 
Let $\Omega$ be the period domain 
\begin{equation}
\label{eq-period-domain}
\Omega := \{x\in \PP[\Lambda\otimes_\Integers\ComplexNumbers] \ : \ (x,x)=0, \ \mbox{and} \ (x,\bar{x})>0\}
\end{equation}
associated to the lattice $\Lambda$.
We get the period map 
\[
P \ : \ \fM_\Lambda \ \ \ \rightarrow \ \ \ \Omega,
\]
$P(X,\eta)=\eta(H^{2,0}(X))$,
which is a local homeomorphism \cite{beauville-varieties-with-zero-c-1}. 
Given a period $x\in \Omega$,
denote by $\Lambda^{1,1}(x)$ the sub-lattice of $\Lambda$ consisting of classes orthogonal to the line $x$.
Let $\phi:\tfM_\Lambda\rightarrow \fM_\Lambda$ be the morphism sending
the isomorphism class of a triple $(X,\eta,E)$ to that of the marked pair $(X,\eta)$.
Fix a connected component $\tfM_\Lambda^0$ of $\tfM_\Lambda$ and let $\fM_\Lambda^0$
be the connected component of $\fM_\Lambda$ containing $\phi(\tfM_\Lambda^0)$.

\begin{thm} 
\label{thm-introduction-Global-Torelli-for-triples}
(Theorem \ref{main-thm})
Assume that Conjecture \ref{conj-Ext-1-is-constant} holds.
\begin{enumerate}
\item
\label{thm-item-inseparable}
The morphism $\widetilde{P}:=P\circ \phi:\tfM_\Lambda^0\rightarrow \Omega$ 
is surjective and a local homeomorphism. Any two points in the same fiber of $\widetilde{P}$ are inseparable
in $\tfM_\Lambda^0$.
\item 
If $\Lambda^{1,1}(x)=0$,  or if $\Lambda^{1,1}(x)$ is cyclic generated by a class of non-negative self intersection,
then the fiber $\widetilde{P}^{-1}(x)$ consists of a single separable point of $\tfM_\Lambda^0$.
\end{enumerate}
\end{thm}

Note that the above theorem is unconditional when $d=1$, as Conjecture \ref{conj-Ext-1-is-constant} 
is known when $E$ is locally free. 
When $X$ is a $K3$ surface and $d=1$, Theorem 
\ref{thm-introduction-Global-Torelli-for-triples} generalizes the Torelli theorem for 
$K3$ surfaces \cite{PS,burns-rapoport}
and it may be viewed as a global analogue of Mukai's result that 
the moduli space of $H$-stable vector bundles with Mukai vector
$v$ of self intersection $-2$ on a fixed polarized $K3$ surface $(X,H)$ is 
non-empty, smooth, connected, and zero dimensional (a point)
\cite{mukai-hodge}. We consider in our global analogue  only rank $r$ Azumaya algebras (or equivalently 
$\PP^{r-1}$ bundles) with characteristic class of order $r$ in $H^2(X,\mu_r)$, as assumed in 
Condition \ref{cond-open-and-closed} (\ref{cond-item-Brauer-class-is-theta-X-eta}). 
The locally free case of Theorem \ref{thm-introduction-Global-Torelli-for-triples} 
follows {\em easily} from Verbitsky's proof of his global Torelli theorem \cite{verbitsky-torelli}, 
and Verbitsky's work on hyperholomorphic vector bundles \cite{verbitsky-hyperholomorphic}. 
Most of the effort in this paper goes toward the reflexive but non-locally-free cases. 
When the dimension of $X$ is larger than $2$, the invariance of $c_2(E)$ assumed in
Condition \ref{cond-open-and-closed} (\ref{cond-item-characteristic-class}) is highly restrictive,
and the only interesting example known to the authors is the non-locally free example, which we now describe.

\begin{defi}
Let $(X_1,E_1)$ and $(X_2,E_2)$ be two pairs each consisting of an irreducible
holomorphic symplectic manifold $X_i$
and a reflexive sheaf $E_i$ of Azumaya algebras over $X_i^d$. 
We say that the two pairs are 
{\em deformation equivalent},
if there exist markings $\eta_i$, such that the two triples 
$(X_i,\eta_i,E_i)$, $i=1,2$, belong to the same connected component of the moduli space
$\tfM_\Lambda$.
\end{defi}

Following is our main application. In this case we are able to state a version of Theorem
\ref{thm-introduction-Global-Torelli-for-triples} for isomorphism classes of pairs $(X,E)$, forgetting the marking.
Let $S$ be a K\"{a}hler $K3$ surface and $S^{[n]}$ the Hilbert scheme of length $n$ 
subschemes of $S$. Denote by $\U$ the ideal sheaf of the universal subscheme in $S\times S^{[n]}$.
Let $\pi_{ij}$ be the projection from $S^{[n]}\times S\times S^{[n]}$ onto the product of the
$i$-th and $j$-th factors. Let
\begin{equation}
\label{eq-sheaf-F}
F \ : = \ \SheafExt^1_{\pi_{13}}(\pi_{12}^*\U,\pi_{23}^*\U)
\end{equation}
be the relative extension sheaf over $S^{[n]}\times S^{[n]}$. 
$F$ is a reflexive sheaf of rank $2n-2$ \cite[Prop. 4.5]{markman-hodge}. 
Let $\tau$ be the involution of $S^{[n]}\times S^{[n]}$
interchanging the two factors. Then $\tau^*F\cong F^*$, by Grothendieck-Verdier duality 
and the triviality of the canoical line bundle $\omega_{\pi_{13}}$.
Let $E$ be the sheaf $\SheafEnd(F)$ and set $E^*:=\SheafEnd(F^*)$. The two sheaves 
$E$ and $E^*$ are isomorphic as sheaves, but if $n>2$ they are not isomorphic
as sheaves of Azumaya algebras.

\begin{thm}
\label{thm-main-example}
\begin{enumerate}
\item
\label{thm-item-E-satisfies-open-closed-condition}
$E$ is a  sheaf  of Azumaya algebras over $S^{[n]}\times S^{[n]}$,
which satisfies\footnote{In the set-up of the theorem, where the marking
is forgotten, the characteristic class $\tilde{\theta}$ in 
Condition \ref{cond-open-and-closed} (\ref{cond-item-Brauer-class-is-theta-X-eta})
is replaced by a monodromy orbit of characteristic classes
in $\Lambda^{\oplus 2}/(2n-2)\Lambda^{\oplus 2}$, where $\Lambda=H^2(S^{[n]},\Integers)$.
The orbit is determined in Equation (\ref{eq-monodromy-orbit-of-characteristic-class-tilde-theta}).
} 
Condition \ref{cond-open-and-closed}  and is slope-stable with respect to every K\"{a}hler class on $S^{[n]}$, 
if $\Pic(S)$ is trivial. 
\item
\label{thm-item-existence-of-E-over-every-X}
Let $X$ be of $K3^{[n]}$-type. There exists a sheaf $E'$ of Azumaya algebras over $X\times X$, which satisfies 
Condition \ref{cond-open-and-closed}, and such that $(X,E')$ is 
deformation equivalent to $(S^{[n]},E)$.
\item
\label{thm-item-Torelli-without-marking}
Let $X$ be of $K3^{[n]}$-type and 
$E_1$ and $E_2$ be two  sheaves of Azumaya algebras over $X\times X$, such that 
$(X,E_i)$ is deformation equivalent to $(S^{[n]},E)$ and satisfies
Condition \ref{cond-open-and-closed}, $i=1,2$. 
Assume that one of the following conditions hold.
\begin{enumerate}
\item
$Pic(X)$ is trivial, or cyclic generated by a line-bundle of non-negative Beauville-Bogomolov-Fujiki degree.
\item
$E_1$ and $E_2$ are $\omega$-slope-stable with respect to the same K\"{a}hler class $\omega$ on $X$.
\end{enumerate}

\medskip
\noindent
Then $E_1$ is isomorphic to $E_2$ or $E_2^*$, as a sheaf of Azumaya algebras. 
\end{enumerate}
\end{thm}

The theorem is proven in section \ref{sec-Example}. 
Parts \ref{thm-item-existence-of-E-over-every-X} and \ref{thm-item-Torelli-without-marking} 
of the theorem are conditional on either Conjecture \ref{conj-Ext-1-is-constant} or 
the following conjecture involving only isolated singularities.

\begin{conj}
\label{conj-Ext-1-is-constant-isolated-singularity-case}
Let $X$ be an irreducible holomorphic symplectic manifold, $\omega$ a K\"{a}hler class on $X$, and $E$ an 
$\omega$-stable hyperholomorphic reflexive sheaf on $X$, which is locally free away from a single point $x_0\in X$.
Assume that $H^1(X,E)=0$. Then $H^1(\X_t,\E_t)=0$, for every point $t$ in the base $\PP^1_\omega$ of the
hyperholomorphic deformation $(\X_t,\E_t)$ of $(X,E)$.
\end{conj}

We will see that Condition \ref{cond-open-and-closed} is open (Lemma \ref{lemma-condition-involving-W-is-open}), and so 
the condition holds more generally for the sheaf $E$ in Theorem \ref{thm-main-example}
(\ref{thm-item-E-satisfies-open-closed-condition}) over $S^{[n]}\times S^{[n]}$ 
for $K3$ surfaces $S$ in a dense open subset of moduli and, in particular, for a generic projective $K3$ surface. 
Let $M$ be a more general smooth and projective connected 
component of the moduli space of sheaves on a projective $K3$ surface $S$
and $\U$ a (possibly twisted) universal sheaf over $S\times M$. Set $n:=\dim(M)/2$ and assume that $n\geq 2$.
We get the rank $2n-2$ reflexive sheaf $F$ over $M\times M$, as above,
and the proof of the above Theorem shows that $E_M:=\SheafEnd(F)$  satisfies Condition
\ref{cond-open-and-closed}, 
except possibly for the stability Condition \ref{cond-open-and-closed}  (\ref{condition-stability}). 
When the Brauer class of $\U$ has order $2n-2$, then $E_M$ satisfies Condition
\ref{cond-open-and-closed} by Proposition \ref{prop-maximally-twisted-sheaf-is-slope-stable}.
Explicit examples are provided in \cite[Theorem 1.5]{markman-hodge}. 

We expect $E_M$ to always satisfy Condition \ref{cond-open-and-closed}, regardless of the order of the Brauer class of 
$\U$. 
Given any smooth and projective component $M$ of the moduli space of sheaves over a $K3$ surface $S_1$, 
the pairs $(M,E_M)$ and
$(S^{[n]},E)$ 
are known to be deformation equivalent, where $S$ and
$E$ are as in Theorem \ref{thm-main-example} above
\cite{yoshioka-abelian-surface}. 
Consequently, 
if $E_M$  satisfies Condition
\ref{cond-open-and-closed} (\ref{condition-stability}), then for every marking $\eta_1$ of $M$,
there exists a marking $\eta_2$ of $S^{[n]}$, such that the triples
$(M,\eta_1,E_M)$ and $(S^{[n]},\eta_2,E)$ belong to the same connected component 
of the moduli space $\tfM_\Lambda$. 

{\bf Acknowledgements:}
We thank D. Toledo for the reference \cite{ramis-ruget-verdier}, and P. Sastry for
pointing us to the commutative algebra result in \cite{matsumura} 
required for the proof of Lemma \ref{lemma-vanishing-of-sheaf-ext-i-Q-O-X}. 
The work of Eyal Markman was partially supported by a grant  from the Simons Foundation (\#245840), and by NSA grant H98230-13-1-0239. This article was begun while 
Sukhendu Mehrotra was visiting the University of Wisconsin; its
hospitality is gratefully acknowledged. 

%
\subsection{Infinitesimal constraints on the singularities of a reflexive sheaf}
\label{sec-infinitesimal-constraints-on-singularities-of-a-reflexive-sheaf}

Let $X$ be a complex manifold, $\Delta\subset X$ a smooth subvariety,  and $\beta:Y\rightarrow X$
the blow-up of $X$ centered along $\Delta$. Let $G$ be a locally free sheaf over $Y$ and set $F:=\beta_*G$.
Recall that the {\em Atiyah class} $At_F$ is the extension class in $\Ext^1(F,F\otimes T^*X)$ of the extension 
\[
0 \rightarrow F\otimes T^*X \rightarrow J^1(F)\rightarrow F \rightarrow 0,
\]
where $J^1(F)$ is the sheaf of first order jets of $F$.
The local extension sheaf $\SheafExt^1(F,F\otimes T^*X)$ is isomorphic to 
$\SheafHom(TX,\SheafExt^1(F,F))$. The image of $At_F$ in $H^0(X,\SheafExt^1(F,F\otimes T^*X)$
may thus be interpreted as a sheaf homomorphism
\begin{equation}
\label{eq-a-F}
a_F \ : \ TX \ \ \ \rightarrow \ \ \ \SheafExt^1(F,F). 
\end{equation}
Given an open Stein subset $U$ of $X$ and a section $\xi$ of $TU$ we get the infinitesimal action morphism 
$\xi:U\times \rm{Spec}(\ComplexNumbers[\epsilon]/(\epsilon^2))\rightarrow U$,
as well as the projection $\pi_1:U\times \rm{Spec}(\ComplexNumbers[\epsilon]/(\epsilon^2))\rightarrow U$. 
The sheaf $\pi_{1_*}\xi^*F$ 
is an extension of $F$ by $F$, whose extension class is $a_F(\xi)$.

Let $\widetilde{\D}^1(G):=\SheafHom(J^1(G),G)$ be the sheaf of  differential operators of order $\leq 1$.
We have the symbol map 
$\sigma_G:\widetilde{\D}^1(G)\rightarrow TY\otimes \SheafEnd(G)$.
Let $\D^1(G)\subset\widetilde{\D}^1(G)$ 
be the subsheaf with scalar symbol. Define $\widetilde{\D}^1(F)$ and $\D^1(F)$ similarly.
There is a natural homomorphism
\begin{equation}
\label{eq-homomorphism-of-sheaves-of-differential-operators}
\beta_* \ : \ \beta_*\D^1(G) \ \ \longrightarrow \ \ \D^1(\beta_*G).
\end{equation}


\begin{defi} 
\label{def-tightness}
We say the the sheaf $F$ is {\em $\beta$-tight}, if 
the homomorphism $a_F$, given in (\ref{eq-a-F}),  is surjective and 
the natural homomorphism $\beta_*:\beta_*\D^1(G)  \longrightarrow  \D^1(F)$
is an isomorphism.
\end{defi}

\hide{
Let $I_\Delta$ be the ideal sheaf of the image $\Delta$ of the diagonal embedding $\iota:X\rightarrow X^d$.
Consider the decreasing filtraten $F^iTX^d:=TX^d\cdot I_\Delta^i$, $i\geq 0$. 
Set $Gr^iTX^d:=F^iTX^d/F^{i+1}TX^d$.
Let $F^i\SheafExt^1(F,F)$ be the image of $F^iTX^d$ via the homomorphism $a_F$. 
Define the graded summands $Gr^i\SheafExt^1(F,F)$ similarly. 
Denote by 
\[
a_F^i \ : \ Gr^iTX^d \ \ \rightarrow \ \ Gr^i\SheafExt^1(F,F)
\]
the homomorphism induced by $a_F$.
Note the isomorphisms
\[
Gr^iTX^d \ \ \cong \ \ \iota^*(TX^d)\otimes \Sym^i(N^*_{\Delta/X^d})
\ \ \cong \ \ T\Delta^{\oplus d}\otimes \Sym^i(T\Delta^{\oplus (d-1)}).
\]

We require the sheaf $F$ to satisfy the following conditions:
\begin{condition} 
\label{cond-atiyah-class}
\begin{enumerate}
\item 
\label{cond-item-Atiyah-homomorphism-is-surjective}
The homomorphism $a_F$ is surjective.
\item
\label{cond-item-kernel-of-Atiyah-homomorphism-is-a-subrepresentation}
The kernel of the homomorphism $a^i_F$ has a trivial determinant line-bundle (over $\Delta$).
\item
\label{cond-item-graded-summands-of-sheaf-Ext-1-F-F-are-locally-free}
The graded summands $Gr^i\SheafExt^1(F,F)$ are locally free $\StructureSheaf{\Delta}$-modules.
\item 
\label{cond-item-global-sections-of-graded-summands-vanish}
$H^0(\Delta,Gr^i\SheafExt^1(F,F))$ vanishes for all $i\geq 0$.
\end{enumerate} 
\end{condition}
}

\begin{defi}\label{inf-rig}
\begin{enumerate}
\item
Let $W$ be a  coherent sheaf over $\PP^n$. 
The pair $(\PP^n,W)$ is said to be {\em infinitesimally rigid}, 
if the differential $\LieAlg{sl}_{n+1}\rightarrow \Ext^1(W,W)$, of the pullback action
of $\Aut(\PP^n)$,
is surjective.
\item
Let $W$ be a coherent sheaf over $\PP(V\otimes \ComplexNumbers^m)$, where $V$ is a symplectic vector space.
The symplectic group $Sp(V)$ acts on the first factor $V$ and the
group $SL(m)$ acts on the second factor
$\ComplexNumbers^m$ yielding an action of $Sp(V)\times SL(m)$ on $V\otimes \ComplexNumbers^m$. 
We say that $W$ is {\em $\LieAlg{sl}(m)$-invariant}, if the kernel of the differential
$\LieAlg{sl}(V\otimes \ComplexNumbers^m)\rightarrow \Ext^1(W,W)$ 
contains $\LieAlg{sl}(m)$. We say that 
$W$ is {\em $\LieAlg{sp}(V)$-equivariant}, if the above kernel is an
$\LieAlg{sp}(V)$-subrepresentation. 
\end{enumerate}
\end{defi}

Let $\pi:\X\rightarrow B$ be a smooth morphism of
complex analytic spaces with conected fibers.
Let $T_\pi$ be the vertical tangent bundle.
Set $\D:=\PP(T_\pi^{\oplus(d-1)})$ and let $p:\D\rightarrow \X$ be the natural morphism.

\begin{defi}
\label{def-vector-bundle-over-D-is-locally-trivial-over-X}
Given a vector bundle $\W$ over $\D$, we consider it as a family over $B$ of pairs 
$(\restricted{\W}{(\pi\circ p)^{-1}(b)},\X_b)$,
where the first term is the restriction of $\W$ to the fiber $\PP((T\X_b)^{\oplus d-1})$ of $\pi\circ p$ over $b\in B$.
We say that the family $\W$ is {\em locally trivial in the topology of} $\X$, if for each point $x\in \X$
there exist open neighborhoods $U$ of $x$ in $\X$ 
and $\overline{U}$ of $\pi(x)$ in $B$, as well as 
an isomorphism $f:U\rightarrow (U\cap \X_{\pi(x)})\times \overline{U}$, 
such that $\pi\circ f^{-1}$ is the projection to $\overline{U}$,
and an isomorphism 
\begin{equation}
\label{eq-trivializing-isomorphism-of-vector-bundles}
\tilde{f}:\pi_1^*\bar{\iota}^*\W\rightarrow \iota^*\W,
\end{equation}
where $\iota:p^{-1}(U)\rightarrow \D$ and $\bar{\iota}:p^{-1}(U\cap \X_{\pi(x)})\rightarrow \D$ are the inclusions,
and the projection $\pi_1:p^{-1}(U)\rightarrow p^{-1}(U\cap \X_{\pi(x)})$, from the projectivised relative tangent bundle
onto the projectivized tangent bundle of $U\cap \X_{\pi(x)}$, is defined using the 
isomorphism $p^{-1}(U)\cong p^{-1}(U\cap \X_{\pi(x)})\times \overline{U}$ induced by $f$.
\end{defi}

Keep the notation of Condition \ref{cond-open-and-closed}. 
Following are the constraints on the singularities along the diagonal 
of the sheaf $E$ of Azumaya algebras over $X^d$ in 
Condition \ref{cond-open-and-closed}
(\ref{cond-item-same-homogeneous-bundle}).
Let $\beta: Y \rightarrow X^d$ be the blow-up of the diagonal 
and let $D\subset Y$ be the exceptional divisor. 
Let $F$ be a reflexive twisted sheaf, such that $E\cong\SheafEnd(F)$.

\noindent
{\bf Condition \ref{cond-open-and-closed} (\ref{cond-item-same-homogeneous-bundle}):}
{\em
The quotient $G:=(\beta^*F)/{\rm torsion}$ is a locally free sheaf and $F$ is $\beta$-tight.
$G$ restricts to each fiber $D_x$ of 
$D\rightarrow X$, which is the projective space
$\PP(T_xX\otimes\ComplexNumbers^{d-1})$, 
as the same\footnote{The condition is needed for flatness of the deformation $\E_t$, $t\in \PP^1_\omega$, 
of the sheaf $E$ over the twistor line $\PP^1_\omega$ 
(Proposition \ref{prop-flatness-of-the-twistor-deformation}) as well as for the condition to be open.
}  
stable unobstructed vector bundle $W$, modulo the action of the automorphism group of the projective space.
The $\Aut(D_x)$-orbit $[W]$ of the isomorphism class of the vector bundle $W$ has the following properties:
\begin{enumerate}[(a)]
\item
\label{condition-item-push-forward-of-V-same-as-that-of-V-minus-mD}
There exist nonnegative integers $m$ and $m'$, such that 
\begin{eqnarray}
\label{eq-push-forward-of-V-same-as-that-of-V-minus-mD}
H^0(W\otimes \StructureSheaf{D_x}(j))&=& 0, \hspace{4ex}  \mbox{for} \ j<m, 
\\
\label{eq-push-forward-of-V-minus-kD-is-flat}
H^i(W\otimes \StructureSheaf{D_x}(j))&=& 0, \hspace{4ex}  \mbox{for} \ i>0, \ 
\mbox{and for} \ j\geq m,
\end{eqnarray}
the analogues of (\ref{eq-push-forward-of-V-same-as-that-of-V-minus-mD})
and (\ref{eq-push-forward-of-V-minus-kD-is-flat}) holds for $W^*$ and $m'$,
and $W\otimes \StructureSheaf{D_x}(m)$ is generated by its global sections.
\item
The pair $(D_x,W)$ is infinitesimally rigid 
and $W$ is $\LieAlg{sl}(d\!\!-\!\!1)$-invariant
 as well as $\LieAlg{sp}(T_xX)$-equivariant.
Finally, there exists a positive integer $k$, such that the traceless endomorphism bundle
$\SheafEnd_0(W)$ satisfies 
\begin{eqnarray}
\label{eq-push-forward-of-End-V-same-as-that-of-End-V-minus-kD}
H^0(\SheafEnd_0(W)\otimes \StructureSheaf{D_x}(j))&=& 0, \hspace{4ex} \mbox{for} \ j <k, 
\\
\label{eq-push-forward-of-End-V-minus-kD-is-flat}
H^i(\SheafEnd_0(W)\otimes \StructureSheaf{D_x}(j))&=& 0, \hspace{4ex}  \mbox{for} \ i>0, \ 
\mbox{and for} \ j\geq 1,
\end{eqnarray}
and $\SheafEnd_0(W)\otimes \StructureSheaf{D_x}(k)$ is generated by its global sections.\footnote{\label{ft}
Caution: There is a subtlety involved in the pullback of sheaves of Azumaya algebras.
The condition
implies that $(\beta^*E)/\rm{torsion}$ is isomorphic to 
$[\SheafEnd_0(G)\otimes \StructureSheaf{Y}(-kD)]\oplus \StructureSheaf{Y}$, where 
$G:=(\beta^*F)/\rm{torsion}$ (see 
Lemma \ref{lemma-sufficient-conditions-for-reflexivity-over-pi} 
(\ref{lemma-item-pullback-of-F-modulo-torsion-is-locally-free})). 
In particular, $(\beta^*E)/\rm{torsion}$ is not a sheaf of Azumaya algebras, if $E$ is not locally free,
as $k>0$.
}
\item
\label{condition-item-local-triviality}
Let $\pi:\X\rightarrow B$ be a smooth and proper family of irreducible holomorphic symplectic 
manifolds over an  analytic space $B$. If a vector bundle $\W$ over $\D:=\PP(T_\pi^{d-1})$ restricts as 
a vector bundle in the orbit $[W]$
to each fiber of $p:\D\rightarrow\X$,
then $\W$ is locally trivializable in the topology of $\X$ in the sense of Definition
(\ref{def-vector-bundle-over-D-is-locally-trivial-over-X}).
\end{enumerate}
}

A simple example of a vector bundle $W$ over a projective space, having the above properties, 
is provided in Equation (\ref{eq-W-is-a-subquotient}). In Condition
\ref{cond-open-and-closed} (\ref{cond-item-same-homogeneous-bundle}\ref{condition-item-local-triviality})
we view the pair $(D_x,W)$ as an infinitesimal structure on the tangent space  $T_xX$ of $X$.
In the case of the vector bundle in Equation (\ref{eq-W-is-a-subquotient}) this structure is equivalent to 
a symplectic structure, up to a scalar, as shown in section \ref{sec-Example}.

The $\beta$-tightness condition  
is needed in Proposition \ref{prop-vanishing-of-global-sections-of-extension-sheaf}.
There we show that the vanishing of $H^1(\X_t^d,\E_t)$ in Conjecture
\ref{conj-Ext-1-is-constant} implies that $\E_t$ is infinitesimally rigid.

%
\section{Families of reflexive sheaves of Azumaya algebras}
We define and study in this section a class of families of reflexive sheaves with good base change properties.
Throughout this section 
$S$ will be an analytic space (not necessarily reduced) and 
$\pi:\X\rightarrow S$ will be a smooth morphism with connected fibers.
In the applications $\pi$ will be proper, but 
we do not assume properness in this section.

%
\subsection{Families of reflexive sheaves}
\label{sec-families-of-reflexive-sheaves}
Given a coherent sheaf $\F$ on $\X$ and a morphism $T\rightarrow S$,
set $\X_T:=T\times_S\X$ and let $\F_T$ be the pullback of $\F$ by
the natural morphism from $\X_T$ to $\X$. 

\begin{defi}
A {\em family $\F$ of reflexive sheaves over} $\pi$ is a coherent  sheaf $\F$ over $\X$
satisfying the following conditions.
\begin{enumerate}
\item
Both $\F$ and its dual $\F^*:=\SheafHom(\F,\StructureSheaf{\X})$ are flat over $S$.
\item 
\label{def-item-reflexivity-of-pullback}
The natural homomorphisms
$(\F^*)_T\rightarrow (\F_T)^*$ 
and 
$\F_T\rightarrow (\F_T)^{**}$
are isomorphisms, for every morphism $T\rightarrow S$.
\end{enumerate}
If $\F$ is a twisted coherent sheaf over $\X$ satisfying the conditions above, we say that $\F$ is a 
{\em family of reflexive twisted sheaves over} $\pi$.
\end{defi}


In the rest of this subsection we provide a construction of families of reflexive sheaves.
Let $\Z\subset \X$ be a subscheme, smooth and proper over $S$, of relative co-dimension $c\geq 2$.
Let $\beta:\Y\rightarrow \X$ be the blow-up of $\X$ with center $\Z$. Denote by
$\D\subset \Y$ the 
exceptional divisor, $e:\D\hookrightarrow \Y$ the closed immersion, 
and let $p:\D\rightarrow \Z$ be the natural morphism.
Note that $p$ is a $\PP^{c-1}$ bundle.
Let $V$ be a locally free sheaf over $\Y$. 
$V$ may be $(\beta^*\theta)$-twisted by the pullback $\beta^*\theta$
of a \v{C}ech $2$-co-cycle $\theta$ for the sheaf $\StructureSheaf{\X}^*$.
In the latter case the higher direct images 
$R^i\beta_*V$ are $\theta$-twisted coherent sheaves, for all $i\geq 0$.

\begin{lem}
\label{lemma-push-forward-is-flat}
Assume that 
${\displaystyle R^ip_*e^*(V(-j\D))=0}$, for $i>0$ and for $j\geq 0$.
Then $R^i\beta_*V=0$, for $i>0$, and 
$\beta_*V$ is flat over $S$.
\end{lem}

\begin{proof}
We prove first that $R^i\beta_*V=0$, for $i>0$.
Let $I:=\StructureSheaf{\Y}(-\D)$ be the ideal sheaf of $\D$.
Let $V_n:=V\otimes (\StructureSheaf{\Y}/I^{n+1})$ be the restriction of $V$ to the $n$-th order 
infinitesimal neighborhood of $\D$.
By Grothendieck's comparison theorem, it suffices to prove that 
$R^i\beta_*(V_n)$ vanishes, for all $n\geq 0$. 
Tensor by $V$ the short exact sequence
\[
0\rightarrow I^n/I^{n+1}\rightarrow \StructureSheaf{\Y}/I^{n+1}\rightarrow \StructureSheaf{\Y}/I^n\rightarrow 0
\] 
to get the short exact sequence
\[
0\rightarrow e_*e^*(V(-n\D))\rightarrow V_n\rightarrow V_{n-1}\rightarrow 0.
\]
Now $R^i\beta_*(e_*e^*(V(-n\D)))=R^ip_*e^*(V(-n\D))$ vanishes,
by assumption, and we get the isomorphism 
$R^i\beta_*(V_n)\cong R^i\beta_*(V_{n-1})$, for all $i>0$ and for all $n\geq 1$.
Consequently, $R^i\beta_*(V_n)\cong R^i\beta_*(V_0)=R^i\beta_*(e_*e^*V)=R^ip_*(e^*V)=0,$
for all $i>0$.

It remains to prove that $\beta_*V$ is flat over $S$. 
Given a morphism $g:T\rightarrow S$ we get the cartesian diagram
\begin{equation}
\label{diagram-T-S}
\xymatrix{
\Y_T\ar[r]^h \ar[d]_{\beta_T}& 
\Y\ar[d]^\beta
\\
\X_T\ar[d] \ar[r]^f & 
\X\ar[d]
\\
T\ar[r]^g & S.
}
\end{equation}
The vanishing of $R^i\beta_*V$, for $i>0$, implies that $\beta_{T_*}(h^*V)\cong Lf^*\beta_*V$, by
\cite[Prop. 6.3]{cusp}.
Therefore, $Lf^*(\beta_*V)$ is concentrated in degree $0$.
Considering  a closed point $T$ of $S$ we get that $\SheafTor_i^S(\ComplexNumbers_T,\beta_*V)=0$, for
$i>0$. Flatness of $\beta_*V$ over $S$ now follows from the local criterion for flatness
\cite[Ch. 8 Theorem 49]{matsumura}.
\end{proof}

\begin{lem}
\label{lemma-sufficient-conditions-for-reflexivity-over-pi}
\begin{enumerate}
\item
\label{lemma-item-push-forward-F-is-a-family-of-reflexive-sheaves}
Assume that there exist integers $k\geq 0$ and $\ell\geq 1-c$, such that the following conditions hold.
\begin{enumerate}
\item
\label{assumption-item-push-forward-of-V-and-V-kD-are-equal}
${\displaystyle p_*e^*(V(-j\D))=0,}$ for $j<k$.
\item
\label{assumption-item-flatness-of-F}
${\displaystyle R^ip_*e^*(V(-j\D))=0}$, for $i>0$ and for $j\geq k$.
\item
\label{assumption-item-push-forward-of-V-dual-times-dualizing-sheaf}
${\displaystyle p_*e^*(V^*(-j\D))=0,}$ for $j<\ell$.
\item
\label{assumption-item-f;atness-of-F-dual}
${\displaystyle R^ip_*e^*(V^*(-j\D))=0}$, for $i>0$ and for $j\geq \ell$.
\end{enumerate}
Then $\F:=\beta_*V$ is a family of reflexive sheaves over $\pi$.
\item
\label{lemma-item-pullback-of-F-modulo-torsion-is-locally-free}
Assume (\ref{assumption-item-push-forward-of-V-and-V-kD-are-equal}) and 
(\ref{assumption-item-flatness-of-F}) above and that 
 the counit $p^*p_*\rightarrow \mbox{id}$ for the adjunction $p^*\nolinebreak \vdash \nolinebreak p_*$ induces a surjective
 homomorphism $p^*p_*e^*V(-k\D)\rightarrow e^*V(-k\D)$. 
Then $\beta_*V=\beta_*V(-k\D)$ and
the analogous natural homomorphism
\begin{equation}
\label{eq-pullback-of-F-modulo-torsion-is-locally-free}
\beta^*\F=\beta^*\beta_*V(-k\D) \rightarrow V(-k\D)
\end{equation}
is surjective and its kernel is supported on $\D$.
\end{enumerate}
\end{lem}

\begin{proof}
\ref{lemma-item-push-forward-F-is-a-family-of-reflexive-sheaves}) 
We prove first the flatness of $\F$ and $\F^*$.
Assumption (\ref{assumption-item-push-forward-of-V-and-V-kD-are-equal})
implies that $\F$ is isomorphic to $\beta_*(V(-k\D))$. The sheaf
$\beta_*V(-k\D)$ is flat over $S$ and $R\beta_*V(-k\D)$ is isomorphic to $\beta_*V(-k\D)$, by assumption
(\ref{assumption-item-flatness-of-F}) and Lemma
\ref{lemma-push-forward-is-flat}.
We have the isomorphisms
\begin{eqnarray*}
\F^*&\cong& 
\H^0\left[
R\SheafHom(\beta_*V(-k\D),\StructureSheaf{\X})
\right]\cong 
\H^0\left[
R\SheafHom(R\beta_*V(-k\D),\StructureSheaf{\X})
\right]
\\
&\cong&
\H^0\left[
R\beta_*(V^*(k\D)\otimes\omega_\beta)
\right]\cong
\beta_*(V^*((k+c-1)\D))\cong\beta_*(V^*(-\ell\D)),
\end{eqnarray*}
where the first is clear, 
the second follows from the isomorphism $\beta_*V(-k\D)\cong R\beta_*V(-k\D)$
established above, the third is Grothendieck-Verdier Duality \cite{ramis-ruget-verdier},
the fourth follows from the isomorphism $\omega_\beta\cong \StructureSheaf{\X}((c-1)\D)$, 
and the last follows from 
assumption (\ref{assumption-item-push-forward-of-V-dual-times-dualizing-sheaf}).
Now $\beta_*(V^*(-\ell\D))$ is flat over $S$, by 
assumption (\ref{assumption-item-f;atness-of-F-dual}) and Lemma
\ref{lemma-push-forward-is-flat}.

We prove next the isomorphism $f^*(\F^*)\cong(f^*\F)^*$.
We may assume that $k=0$, possibly after replacing $V$ by $V(-k\D)$ and replacing $\ell$ by $\ell+k$. 
Given a morphism $g:T\rightarrow S$ we get the cartesian diagram
(\ref{diagram-T-S}).
We have the isomorphisms
\begin{equation}
\label{eq-pullback-of-F-is-push-forward-of-pullback-of-V}
f^*\beta_*V\cong Lf^*\beta_*V\cong
Lf^*R\beta_*V\cong R\beta_{T_*}(h^*V)\cong \beta_{T_*}(h^*V),
\end{equation}
where the first (leftmost) follows from flatness of $\F$, 
the second by assumption
(\ref{assumption-item-flatness-of-F}), the third by cohomology and base change, 
and the fourth by assumption
(\ref{assumption-item-flatness-of-F}) again. Hence, $f^*\F\cong \beta_{T_*}(h^*V)$.
We also have the isomorphisms
\[
f^*(\F^*)\cong f^*(\beta_*(V^*(-\ell\D)))\cong 
R\beta_{T_*}(h^*[V^*(-\ell\D)])\cong
\beta_{T_*}((h^*V^*(-\ell\D_T))),
\] 
where the left was established above, the middle one follows from a sequence analogous to the
one in equation (\ref{eq-pullback-of-F-is-push-forward-of-pullback-of-V}), and the right 
by assumption (\ref{assumption-item-f;atness-of-F-dual}).
Now, $\beta_{T_*}((h^*V^*(-\ell\D_T)))$ is isomorphic to $\beta_{T_*}(h^*V^*\otimes\omega_{\beta_T})$,
by assumption (\ref{assumption-item-push-forward-of-V-dual-times-dualizing-sheaf}),
the sheaf $\beta_{T_*}(h^*V^*\otimes\omega_{\beta_T})$
is equal to the sheaf $\H^0[R\beta_{T_*}(h^*V\otimes\omega_\beta)]$,
which is isomorphic to
$\H^0[R\SheafHom(R\beta_{T_*}(h^*V),\StructureSheaf{\X_T})]$,
by Grothendieck-Verdier duality, and the latter is isomorphic to 
$\SheafHom(\beta_{T_*}(h^*V),\StructureSheaf{\X_T})$, using the isomorphism 
$R\beta_{T_*}h^*V\cong \beta_{T_*}h^*V$
established above.
We get the isomorphism  $f^*(\F^*)\cong \SheafHom(\beta_{T_*}(h^*V),\StructureSheaf{\X_T})$. We conclude the isomorphism 
$f^*(\F^*)\cong(f^*\F)^*$ from Equation (\ref{eq-pullback-of-F-is-push-forward-of-pullback-of-V}).

It remains to prove that $\F_T$ is reflexive. 
Dualizing the above isomorphisms we get
\[
(\F_T)^{**}=(f^*\F)^{**}\cong (f^*(\F^*))^*\cong\left(\beta_{T_*}((h^*V^*)(-\ell\D_T))\right)^*,
\]
and the right hand side is isomorphic to 
$\H^0\left[R\SheafHom\left(R\beta_{T_*}((h^*V^*)(-\ell\D_T)),\StructureSheaf{\X_T}
\right)\right]$, where we used the isomorphism
$\beta_{T_*}((h^*V^*)(-\ell\D_T))\cong R\beta_{T_*}((h^*V^*)(-\ell\D_T))$
established above. Grothendieck-Verdier duality yields the isomorphism
\[
\H^0\left[R\SheafHom\left(R\beta_{T_*}((h^*V^*)(-\ell\D_T)),\StructureSheaf{\X_T}
\right)\right] \cong
\H^0\left[
R\beta_{T_*}\left(R\SheafHom(h^*V^*(-\ell\D_T),\beta^!\StructureSheaf{\X_T})
\right)\right].
\]
The right hand side is isomorphic to $\beta_{T_*}(h^*V((\ell+c-1)\D_T))$, which in turn is
isomorphic to 
$\beta_{T_*}h^*V$, by assumption
(\ref{assumption-item-push-forward-of-V-and-V-kD-are-equal}). 
The latter is $\F_T$,
by Equation (\ref{eq-pullback-of-F-is-push-forward-of-pullback-of-V}).

\ref{lemma-item-pullback-of-F-modulo-torsion-is-locally-free}) 
Consider the short exact sequence
\[
0\rightarrow V(-(k+1)\D)\rightarrow V(-k\D)\rightarrow e_*e^*V(-k\D)\rightarrow 0.
\]
The sheaf $R^1\beta_*V(-(k+1)\D)$ vanishes, by assumption (\ref{assumption-item-flatness-of-F}).
We get the short exact sequence
\[
0\rightarrow \beta_*V(-(k+1)\D)\rightarrow \beta_*V(-k\D)\rightarrow \beta_*e_*e^*V(-k\D)\rightarrow 0.
\]
Pulling back via $\beta^*$ we get the commutative diagram with right exact top row:
\[
\xymatrix{
\beta^*\beta_*V(-(k+1)\D) \ar[r] &
\beta^*\beta_*V(-k\D) \ar[r] \ar[d]_{ev} &
\beta^*\beta_*e_*e^*V(-k\D) \ar[d]^{\gamma} \ar[r] & 0
\\
& V(-k\D) \ar[r]_{\rho} & e_*e^*V(-k\D).
}
\]
The vertical homomorphism are associated to the co-unit natural transformation $\beta^*\beta_*\rightarrow \mbox{id}$.
Let $\zeta:\Z\hookrightarrow \X$ be the closed immersion. 
Then $\beta\circ e=\zeta\circ p$. Hence, 
$\beta^*\beta_*e_*e^*V(-k\D)=\beta^*\zeta_*p_*e^*V(-k\D)$. Now $p_*e^*V(-k\D)$ is a locally free sheaf over
$\Z$, by assumption (\ref{assumption-item-flatness-of-F}), and
$\beta^*\zeta_*\StructureSheaf{\Z}\cong e_*\StructureSheaf{\D}=e_*p^*\StructureSheaf{\Z}$.
We get the right isomorphism below:
\[
\beta^*\beta_*e_*e^*V(-k\D)=\beta^*\zeta_*p_*e^*V(-k\D)\cong e_*p^*p_*e^*V(-k\D).
\]
The homomorphism $\gamma$ is the composition of the above isomorphism 
with the homomorphism $e_*p^*p_*e^*V(-k\D)\rightarrow e_*e^*V(-k\D)$.
The latter is surjective, by assumption. Hence, the homomorphism $\gamma$ is surjective.
We conclude that the composition 
$\rho\circ ev: \beta^*\beta_*V(-k\D)\rightarrow e_*e^*V(-kD)$
is surjective, by the commutativity of the diagram above. 
The surjectivity of the homomorphism $ev$ on stalks of points of $\D$
follows, by Nakayama's Lemma. The homomorphism $ev$ clearly induces 
an isomorphism on stalks of points outside the exceptional divisor $\D$.
\end{proof}

\begin{rem}
Let $\M\rightarrow S$ be a smooth morphism with connected fibers of dimension $\geq 3$.
Set $\X:=\M\times_S\M$, $\pi:\X\rightarrow S$ the natural morphism,
$\Z\subset \X$ the diagonal, $\beta:\Y\rightarrow \X$ the blow-up along $\Z$, and
$V$ a vector bundle over $\Y$ satisfying the assumptions of 
Lemma \ref{lemma-sufficient-conditions-for-reflexivity-over-pi}.
Then $\F:=\beta_*V$ is a family of reflexive sheaves over $\pi$, by 
Lemma \ref{lemma-sufficient-conditions-for-reflexivity-over-pi}. Let
$\pi_i: \X\rightarrow \M$ be the projection, $i=1,2$. 
We claim that $\F$ is a family of reflexive sheaves over $\pi_i$ as well.
Indeed, $\Z$ is smooth over $\M$ as well being a section of $\pi_i$. 
Hence, Lemma \ref{lemma-sufficient-conditions-for-reflexivity-over-pi}
applies with the morphisms $\pi_i$, $i=1,2$, as well. 
\end{rem}
%
\subsection{Some basic properties of families of reflexive sheaves}

\begin{lem}
\label{lemma-codimension-of-singular-locus-of-family-of-reflexive-sheaves}
Let $\F$ be a family of reflexive sheaves over $\pi$.
Let $U\subset \X$ be the open subset over which $\F$ is locally free and set
$\Z:=\X\setminus U$. Given a closed point $s\in S$, let $\X_s$ and $\Z_s$ be the fibers of $\X$ and $\Z$ over $s$.
Then the co-dimension of $\Z_s$ in $\X_s$ is  at least $3$.
\end{lem}

\begin{proof}
The restriction $\F_s$  of $\F$ to $\X_s$ is reflexive, by definition. 
Hence, $\F_s$ is locally free away from a closed analytic subset of co-dimension $\geq 3$.
If $\F_s$ is locally free at a point $x\in \X_s$, then $\F$ is locally free at that point.
Indeed, let $\{f_1, \dots,f_r\}$ be a basis of the stalk $\F_{s,x}$ of $\F_s$ at $x$
and $\phi:\oplus_{i=1}^r\StructureSheaf{\X_{s,x}}\rightarrow \F_{s,x}$ the corresponding isomorphism
of stalks.
Choose a subset $\{\tilde{f}_1, \dots, \tilde{f}_r\}$ of the stalk $\F_x$ of $\F$ at $x$, which maps to 
the above basis, and let $\tilde{\phi}:\oplus_{i=1}^r\StructureSheaf{\X_x}\rightarrow \F_x$ be the
corresponding homomorphism of stalks. Then $\tilde{\phi}$ is surjective, by Nakayama's lemma.
Let $N$ be the kernel of $\tilde{\phi}$. Tensoring the exact sequence
\[
0\rightarrow N \rightarrow \oplus_{i=1}^r\StructureSheaf{\X_x}\RightArrowOf{\tilde{\phi}} \F_x\rightarrow 0
\]
by the stalk $\StructureSheaf{S_s}$ we get the long exact sequence
\[
\SheafTor_1^{\StructureSheaf{\X_x}}(\F_x,\StructureSheaf{S_s})\rightarrow 
N\otimes_{\StructureSheaf{\X_x}}\StructureSheaf{S_s}\rightarrow 
\oplus_{i=1}^r\StructureSheaf{\X_{s,x}}\RightArrowOf{\phi} \F_{s,x}
\rightarrow 0.
\]
Now $\SheafTor_1^{\StructureSheaf{\X_x}}(\F_x,\StructureSheaf{S_s})$ vanishes, by flatness of $\F$ over $S$,
and $\phi$ is an isomorphism. Thus, 
$N$  restricts to zero along $\X_{s,x}$.
Hence, $N=0$, by Nakayama's lemma.
Consequently, $\tilde{\phi}$ is injective as well, and 
$\F$ is locally free at $x$. 
We conclude that the co-dimension of $\Z\cap\X_s$ in $\X_s$ is at least $3$.
\end{proof}

Let $\Z$ be a closed analytic subset of $\X$ and $I$ its ideal sheaf.
Given a point $x\in \X$, denote by $I_x$ the stalk of $I$ at $x$ and by
$\StructureSheaf{\X,x}$ the stalk of $\StructureSheaf{\X}$ at $x$. 
Note that $\StructureSheaf{\X,x}$ is a noetherian ring \cite[Prop. 1]{serre}.
Let $\E$ be a coherent sheaf on $\X$ and $\E_x$ its stalk at $x$. 
Denote by $\depth_{I_x}(\E_x)$ the maximal length of a regular sequence in $I_x$ for the 
$\StructureSheaf{\X,x}$-module $\E_x$. Given a closed analytic subset $\Z\subset\X$,
set $\depth_\Z(\E):=\inf_{x\in\Z}\depth_{I_x}(\E_x)$ 
\cite[Cor. 3.6]{trautmann,hartshorne-local-cohomology}.

\begin{lem}
\label{lemma-vanishing-of-sheaf-ext-i-Q-O-X}
Assume that $\Z$
intersects each fiber of $\pi$ in a subset of codimension $\geq c$.
Then $\depth_{\Z}\left(\StructureSheaf{\X}\right)\geq c.$ Equivalently, 
given a coherent sheaf $Q$ supported set theoretically on $\Z$, the 
extension sheaves $\SheafExt^i(Q,\StructureSheaf{\X})$ vanish, for $i<c$.
\end{lem}

\begin{proof}
Let $x$ be a point of $\X$ and let $Q_x$ and 
$\SheafExt^i(Q,\StructureSheaf{\X})_x$ be the stalks at $x$.
We have the isomorphism $\SheafExt^i(Q,\StructureSheaf{\X})_x\cong
\Ext^i(Q_x,\StructureSheaf{\X,x}),$ 
by \cite[Prop. III.6.8]{AG}.
There exists a regular sequence $a_1, \dots, a_c$ of length $c$ in $I_x$, by \cite[Cor. 20.F]{matsumura}, 
since $\Z$ has relative co-dimension $c$ in $\X$ over $S$ and $\X$ is smooth over $S$.
Thus, $\Ext^i(Q_x,\StructureSheaf{\X,x})=(0)$, for $i<c$,
by \cite[Ch. 6 Theorem 28]{matsumura}. 
The equivalence with the inequality $\depth_{\Z}\left(\StructureSheaf{\X}\right)\geq c$
follows from \cite[Prop. 3.3]{hartshorne-local-cohomology} 
\end{proof}

\hide{
\begin{rem}
\label{rem-existence-of-a-regular-sequence}
Here is a global argument supporting the vanishing of the extension sheaves
$\SheafExt^i(Q,\StructureSheaf{\X})$, for $i<c$ in the proof above. 
Suppose that $\pi:\X\rightarrow S$ is smooth and projective of relative dimension $n$.
Let $\StructureSheaf{\X}(1)$ be a $\pi$-ample line bundle over $\X$.
Then 
\begin{equation}
\label{eq-using-verdier-duality-to-prove-vanishing-of-extension-sheaves}
R\pi_*\left[R\SheafHom(Q(-d),\omega_\pi)\right]
\cong
R\SheafHom\left(
R\pi_*Q(-d),\StructureSheaf{S}[-n]\right),
\end{equation}
by Grothendieck-Verdier duality.
The left hand side is the limit of a spectral sequence whose $E_2^{p,q}$ 
terms are
$R^p\pi_*\SheafExt^q\left(Q(-d),\omega_\pi\right)$.
These vanish for $p>0$ and $d$ sufficiently large.
For such $d$ we get an isomorphism
\[
\H^i\left(R\pi_*\left[R\SheafHom(Q(-d),\omega_\pi)\right]\right)
\cong 
\pi_*\SheafExt^i\left(Q(-d),\omega_\pi\right).
\]
Furthermore, for $d$ sufficiently large, the sheaf
$\SheafExt^i\left(Q(-d),\omega_\pi\right)$ vanishes, if and only if 
its push-forward vanishes. Thus, it suffices to show 
that the $i$-th sheaf cohomologies of the right hand side of 
(\ref{eq-using-verdier-duality-to-prove-vanishing-of-extension-sheaves})
vanish for $i<c$. The sheaves $R^i\pi_*Q(-d)$ vanish for $i>n-c$, 
since fibers of $\Z\rightarrow S$ have dimension $\leq n-c$. 
Thus, the object $R\pi_*Q(-d)$ in $D(S)$ is quasi-isomorphic to a complex of locally free sheaves
\[
\dots \rightarrow V_{i-1}\rightarrow V_i \rightarrow \cdots \rightarrow V_{n-c}
\]
bounded from above. Dualizing, we get a complex representing 
$R\SheafHom\left(R\pi_*Q(-d),\StructureSheaf{S}\right)$, which is bounded from below at
degree $c-n$.
Thus, the sheaves $\H^i\left\{R\SheafHom\left(
R\pi_*Q(-d),\StructureSheaf{S}\right)
\right\}$
vanish for $i<c-n$.
We conclude that 
$\H^i\left\{R\SheafHom\left(
R\pi_*Q(-d),\StructureSheaf{S}[-n]\right)
\right\}$
vanish for $i<c$, as claimed.
\end{rem}
}

\begin{lem} 
\label{lemma-reflexive-F-is-a-subsheaf-of-a-locally-free-sheaf}
Let $\F$ be a family of reflexive sheaves over $\pi$. 
Then, locally over $\X$, it can be included in an exact sequence
$
0\rightarrow \F \rightarrow \E \rightarrow \G\rightarrow 0,
$
where $\E$ is locally free and $\G$ is a subsheaf of a locally free sheaf.
\end{lem}

\begin{proof}
We follow the first part of the proof of 
\cite[Prop. 1.1]{hartshorne-stable-reflexive}.
The statement is local, so we may assume that there exists a right exact sequence
$V_1\RightArrowOf{d} V_0\rightarrow \F^*\rightarrow 0$,
where $V_0$ and $V_1$ are locally free. 
Taking duals we get the left exact sequence
$
0\rightarrow \F^{**}\rightarrow V_0^*\RightArrowOf{d^*} V_1^*.
$
Set $\E:=V_0^*$ and let $\G$ be the image of $d^*$. The isomorphism $\F\cong\F^{**}$ 
yields the desired short exact sequence.
\end{proof}

\begin{lem}
\label{lemma-depth-of-structure-sheaf-at-least-2-implies-same-for-F}
Let $\F$ be a family of reflexive sheaves over $\pi$. 
Let $\Z\subset \X$ be a closed analytic subset such that $\depth_\Z(\StructureSheaf{\X})\geq 2.$
Then $\depth_\Z(\F)\geq 2.$
\end{lem}

\begin{proof}
We follow the proof of \cite[Prop. 1.3]{hartshorne-stable-reflexive}. 
Recall that $\depth_\Z(\F)\geq k$, if and only if 
the sheaf 
$\SheafExt^i(N,\F)$ vanishes, for every coherent sheaf $N$ supported on $\Z$ and 
for all $i<k$ \cite[Prop. 3.3]{hartshorne-local-cohomology}.
The statement is local and we may assume that there exists a short exact sequence
$0\rightarrow \F\rightarrow \E\rightarrow \G\rightarrow 0$
with $\E$ locally free and $\G$ a subsheaf of a locally free sheaf, by Lemma
\ref{lemma-reflexive-F-is-a-subsheaf-of-a-locally-free-sheaf}. 
Let $N$ be a coherent sheaf supported on $\Z$.
We get the exact sequence 
\[
\SheafHom(N,\G)\rightarrow \SheafExt^1(N,\F)\rightarrow \SheafExt^1(N,\E).
\]
Now $\SheafExt^i(N,W)$ vanishes, for $i\leq 1$, for any  locally free sheaf $W$, since
$\depth_\Z(\StructureSheaf{\X})\geq 2.$ Hence, $\SheafExt^1(N,\E)$ vanishes 
and so do $\SheafHom(N,\F)$ and
$\SheafHom(N,\G),$ since both $\F$ and $\G$ are subsheaves of  locally free sheaves.
We conclude that $\SheafExt^i(N,\F)$ vanishes for $i=0$ and $i=1$.
\end{proof}

\begin{lem}
\label{lemma-F-is-isomorphic-to-the-pushforward-of-its-restriction}
Let $\F$ be a family of reflexive sheaves over $\pi$.
Let $U\subset \X$ be the open subset where $\F$ is locally free, 
and let $\iota:U\rightarrow \X$ be the inclusion. 
Then the natural homomorphism $\F\rightarrow \iota_*\iota^*\F$ is an isomorphism.
\end{lem}

\begin{proof}
Set $\Z:=\X\setminus U$.
Recall that $\depth_\Z(\F)\geq k$, if and only if 
the cohomology sheaves
$\H^i_\Z(\F)$ with support along $\Z$ vanish for $i<k$
\cite[Theorem 3.8]{trautmann,hartshorne-local-cohomology}. 
We know that $\depth_\Z(\StructureSheaf{\X})\geq 2$,
by Lemmas \ref{lemma-codimension-of-singular-locus-of-family-of-reflexive-sheaves} and 
\ref{lemma-vanishing-of-sheaf-ext-i-Q-O-X}.
Hence, $\depth_\Z(\F)\geq 2$, by Lemma \ref{lemma-depth-of-structure-sheaf-at-least-2-implies-same-for-F},
and the sheaves 
$\H^i_\Z(\F)$ vanish for $i=0$ and $i=1$. The statement thus follows from the
long exact sequence 
\[
0\rightarrow \H^0_\Z(\F) \rightarrow \F \rightarrow \iota_*\iota^*\F\rightarrow 
\H^1_\Z(\F)\rightarrow 0
\]
\cite[Section 3]{serre}.
\end{proof}

%
\subsection{Families of Azumaya algebras}
Keep the notation of section
\ref{sec-families-of-reflexive-sheaves}.
In particular, $\Z$ is a subscheme of $\X$, smooth of relative co-dimension $c$ over $S$,
and $\beta:\Y\rightarrow \X$ is the blow-up centered at $\Z$. Assume that $c\geq 3$. 
Let $V_1$ and $V_2$ be coherent sheaves over $\Y$. Assume that $V_1$ is locally free.
Given an open subset $U$ of $\X$ we get the homomorphism
\[
\Gamma(\beta_*(V_1^*\otimes V_2),U):=\Gamma(V_1^*\otimes V_2,\beta^{-1}(U))=
\Gamma(\SheafHom(V_1,V_2),\beta^{-1}(U))
\RightArrowOf{\beta_*} \Gamma(\SheafHom(\beta_*V_1,\beta_*V_2),U).
\]
We denote the corresponding sheaf homomorphism by
\[
\beta_*:\beta_*(V_1^*\otimes V_2)\rightarrow \SheafHom(\beta_*V_1,\beta_*V_2).
\]
The above homomorphism is induced on the cohomology in degree zero by the following 
natural transformation  of exact functors from $D^b(\Y)$ to $D^b(\X)$.
Let $\eta:id\rightarrow \beta^!R\beta_*$ be the unit of the adjunction $R\beta_*\vdash\beta^!$.
We get
\[
R\beta_*R\SheafHom(V_1,V_2)\rightarrow
R\beta_*[R\SheafHom(V_1,\beta^!R\beta_*V_2)]\cong
R\SheafHom(R\beta_*V_1,R\beta_*V_2),
\]
where the left arrow is $R\beta_*R\SheafHom(V_1,\eta_{V_2})$ and the right isomorphism is provided by 
Grothendieck-Verdier duality.

Taking $V_1=V_2=V$, we get the homomorphism
\begin{equation}
\label{eq-beta-*}
\beta_*:\beta_*(V^*\otimes V)\rightarrow \SheafHom(\F,\F).
\end{equation}

\begin{lem}
\label{lemma-push-forward-of-End-V-is-End-F}
Assume that the sheaves $\F:=\beta_*V$ and  $\beta_*(V^*\otimes V)$ are reflexive over $\pi$. 
Then the homomorphism $\beta_*$ given in (\ref{eq-beta-*}) is an isomorphism. 
\end{lem}

\begin{proof}
The homomorphism $\beta_*$ restricts as an isomorphism over $\U:=\X\setminus \Z$.
Hence, the kernel of $\beta_*$ is supported over $\Z$. 
But the reflexive sheaf $\beta_*(V^*\otimes V)$ does not have a non-zero 
subsheaf supported over $\Z$, by Lemma 
\ref{lemma-depth-of-structure-sheaf-at-least-2-implies-same-for-F}. 
Hence, $\beta_*$ is injective.

Let $Q$ be the co-kernel of $\beta_*$ and consider the short exact sequence:
\[
0\rightarrow \beta_*(V^*\otimes V)\rightarrow 
\SheafHom(\F,\F)\rightarrow Q\rightarrow 0.
\]
$Q$ is supported, set theoretically, on a subset of $\Z$.
Hence, 
the sheaf $\SheafExt^1(Q,\beta_*(V^*\otimes V))$ vanishes, 
by Lemma \ref{lemma-depth-of-structure-sheaf-at-least-2-implies-same-for-F}, and 
the above short exact sequence splits, locally over $\X$.
But any local section of $\SheafHom(\F,\F)$, supported over $\Z$, has an image subsheaf of
$\F$, supported over $\Z$. Now $\F$ is assumed reflexive, hence
it does not have non-trivial subsheaves supported over $\Z$. Hence, $Q$ vanishes and $\beta_*$ is an isomorphism.
\end{proof}

\hide{
\begin{proof}
\underline{Injectivity:}
Associated to the co-unit
$\epsilon:L\beta^*R\beta_*\rightarrow id$ we get the natural transformation 
$\omega_\beta\otimes \epsilon:\beta^!R\beta_*\cong 
\omega_\beta\otimes L\beta^*R\beta_*\rightarrow (\omega_\beta\otimes id)$. 
Let $\eta:id\rightarrow \beta^!R\beta_*$ be the unit.
The composite
natural transformation $(\omega_\beta\otimes\epsilon)\eta:id\rightarrow (\omega_\beta\otimes id)$ 
corresponds to tensorization by the section $\det(d\beta)$  of $\omega_\beta$. 

$R\beta_*R\SheafHom(V,\eta_V)$ induces on zero cohomology the homomorphism $\beta_*$ given in
(\ref{eq-beta-*}). It suffices to show that 
the composite natural transformation
$R\beta_*R\SheafHom(V,(\omega_\beta\otimes\epsilon_V)\circ\eta_V):$
\begin{eqnarray*}
R\beta_*R\SheafHom(V,V) &\RightArrowOf{\eta_V} &R\beta_*[R\SheafHom(V,\beta^!R\beta_*V)] =
R\beta_*[R\SheafHom(V,\omega_\beta\otimes L\beta^*R\beta_*V)]
\\
&\RightArrowOf{\epsilon_V} &
R\beta_*R\SheafHom(V,\omega_\beta\otimes V)
\end{eqnarray*}
induces an injective  homomorphism
\[
\beta_*(V^*\otimes V) \rightarrow \beta_*(V^*\otimes V\otimes\omega_\beta)
\]
on cohomology in degree zero, in order to prove the injectivity of $\beta_*$.
This is the case, since tensorization by $\det(d\beta)$ induces 
an injective homomorphism $V^*\otimes V\rightarrow V^*\otimes V\otimes\omega_\beta$ and
the functor $\beta_*$ is left exact.

\underline{surjectivity:}
Let $Q$ be the co-kernel of $\beta_*$ and consider the short exact sequence:
\[
0\rightarrow \beta_*(V^*\otimes V)\rightarrow 
\SheafHom(\F,\F)\rightarrow Q\rightarrow 0.
\]
$Q$ is supported, set theoretically, on a subset of $\Z$, so 
 $\SheafHom(Q,\StructureSheaf{\X})$ and 
$\SheafExt^1(Q,\StructureSheaf{\X})$ vanish, by Lemma
\ref{lemma-vanishing-of-sheaf-ext-i-Q-O-X}. 
Hence, 
$[\beta_*(V^*\otimes V)]^*\rightarrow \SheafHom(\F,\F)^*$ is an isomorphism and
so $\SheafHom(\F,\F)^*$ is reflexive over $\pi$. Consequently,
the homomorphism $[\beta_*(V^*\otimes V)]^{**}\rightarrow \SheafHom(\F,\F)^{**}$ is an isomorphism as well. 
We get the isomorphism $[\beta_*(V^*\otimes V)]\rightarrow \SheafHom(\F,\F)^{**}$,
using the assumption that $\beta_*(V^*\otimes V)$ is reflexive.
The latter isomorphism factors through the homomorphism $\beta_*$, given in (\ref{eq-beta-*}), so 
it remains to prove that the natural homomorphism
$\SheafHom(\F,\F)\rightarrow \SheafHom(\F,\F)^{**}$
is injective. Its kernel is supported on a subset of $\Z$. Now $\F$ 
does not have a subsheaf, which is supported on a subscheme of positive relative co-dimension,
since $\F$ is assumed to be reflexive. Hence, $\SheafHom(\F,\F)$ does not have
a subsheaf, which is supported on a subscheme of positive relative co-dimension.
The natural homomorphism
$\SheafHom(\F,\F)\rightarrow \SheafHom(\F,\F)^{**}$ is thus injective.
\end{proof}
}

\begin{defi}
\label{def-families-of-reflexive-azumaya-algebras}
A {\em family $\A$ of reflexive sheaves of  Azumaya algebras over $\pi$} is a
family of reflexive sheaves over $\pi$, with an associative multiplication 
$m:\A\times \A\rightarrow \A$ and a unit section $1$, such that locally over $\X$, $(\A,m,1)$
is isomorphic to $(\SheafEnd(\F),m',id)$, for some reflexive sheaf $\F$ over $\pi$,
where the multiplication  $m':\SheafEnd(\F)\otimes \SheafEnd(\F)\rightarrow \SheafEnd(\F)$ is 
given by composition.
\end{defi}

\begin{cor}
\label{cor-pushforward-commutes-with-End}
Let $V$ be a locally free sheaf over $\Y$ satisfying the 
assumptions of Lemma \ref{lemma-sufficient-conditions-for-reflexivity-over-pi}.
Assume further that there exists an integer $t\geq 0$, such that $V^*\otimes V$
satisfies the conditions of Lemma \ref{lemma-sufficient-conditions-for-reflexivity-over-pi}
with both $k$ and $\ell$ replaced by $t$. Then $\beta_*(\SheafEnd(V))$ is a family of 
reflexive sheaves of  Azumaya algebras over $\pi$, which is isomorphic to $\SheafEnd(\beta_*V)$.
\end{cor}

\begin{proof}
Both $\F:=\beta_*V$ and $\beta_*(V^*\otimes V)$ are reflexive over $\pi$, by Lemma 
\ref{lemma-sufficient-conditions-for-reflexivity-over-pi}. The sheaf $\beta_*(V^*\otimes V)$
is globally isomorphic to $\SheafEnd(\F)$, by Lemma \ref{lemma-push-forward-of-End-V-is-End-F}.
\end{proof}

%
\subsection{Every family of reflexive Azumaya algebras comes from a family of reflexive sheaves}
Keep the notation of the previous subsection. Assume that the co-dimension $c$ of $\Z$ in $\X$ is 
at least $3$. Let $\nu_X\subset \StructureSheaf{\X}$ be the nilpotent radical ideal subsheaf.
Assume that $\nu_X^m=0$, for some positive integer $m$.

\begin{lem}
\label{lemma-isomorphism-of-Picards}
Let  $W$ be a Stein open subspace of $\X$, set $U:=\X\setminus \Z$, and let 
$\iota:[W\cap U]\hookrightarrow W$ be the inclusion. Denote by 
$\Pic(W\cap U)'$ the subset of $\Pic(W\cap U)$ consisting of line bundles $L$,
such that $\iota_*(L)$ is a coherent $\StructureSheaf{W}$-module. 
Then $\iota^*:\Pic(W)\rightarrow \Pic(W\cap U)'$ is an isomorphism.
\end{lem}

\begin{proof}
The pullback homomorphism $\iota^*:\Pic(W)\rightarrow \Pic(W\cap U)$ is injective, and its image is
contained in $\Pic(W\cap U)'$, since $\iota_*\iota^*(\widetilde{L})$ is isomorphic to $\widetilde{L}$,
by Lemma \ref{lemma-F-is-isomorphic-to-the-pushforward-of-its-restriction}.
Assume first that $S$ is reduced. Then so is $\X$.
The inequality $\depth_\Z(\StructureSheaf{\X})\geq 3$ holds, by Lemma
\ref{lemma-vanishing-of-sheaf-ext-i-Q-O-X}. Hence, $H^i_{W\cap\Z}(\StructureSheaf{W})$ vanishes, 
for $i\leq 2$, and the homomorphism
$\iota^*:H^i(W,\StructureSheaf{W})\rightarrow H^i(W\cap U,\StructureSheaf{W\cap U})$
is an isomorphism, for $i\leq 1$, by \cite[Theorem 3.8]{hartshorne-local-cohomology}.
The space $H^i_{W\cap\Z}(W,\Integers)$ vanishes, for $i\leq 5$, 
by the last example in chapter 1 of \cite{hartshorne-local-cohomology},
where we use the assumption that both $\Z$ and $\X$ are smooth over the base $S$.
Hence, the homomorphism $\iota^*:H^i(W,\Integers)\rightarrow H^i(W\cap U,\Integers)$
is an isomorphism, for $i\leq 4$.  
Considering the long exact sequences associated to the exponential sequence 
we get the commutative diagram:
\[
\xymatrix{
H^1(W,\Integers)\ar[r]\ar[d]_{\cong} & 
H^1(W,\StructureSheaf{W})\ar[r]\ar[d]_{\cong} &
H^1(W,\StructureSheaf{W}^*)\ar[r]\ar[d]&
H^2(W,\Integers)\ar[d]_{\cong}
\\
H^1(W\cap U,\Integers)\ar[r]& 
H^1(W\cap U,\StructureSheaf{W})\ar[r]&
H^1(W\cap U,\StructureSheaf{W}^*)\ar[r]&
H^2(W\cap U,\Integers).
}
\]
We conclude that 
the homomorphism $\iota^*:\Pic(W)\rightarrow \Pic(W\cap U)$ is an isomorphism. 

We prove the general case next. 
Let $L$ be a line bundle on $W\cap U$ and assume that $\iota_*(L)$ is coherent.
Let $\nu_L$ be the subsheaf of $L$ analogous to $\nu_W$. It is well defined, since 
$L$ is determined by some $1$-co-cycle for $\StructureSheaf{W}^*$ and 
a gluing transition function maps (locally) the subsheaf $\nu_W$ of $\StructureSheaf{W}$ to itself. 
We get the short exact sequence of coherent $\StructureSheaf{W}$-modules
\[
0\rightarrow \iota_*(\nu_L)\rightarrow \iota_*(L)\rightarrow \iota_*(L/\nu_L)\rightarrow 0.
\]
Set $\overline{L}:=L/\nu_L$.
Note that $\overline{L}$ is isomorphic to the restriction of $L$ to the reduced subscheme $[W\cap U]_{red}$
with structure sheaf $\StructureSheaf{W\cap U}/\nu_{W\cap U}$.

The restriction homomorphism $\Pic(W)\rightarrow \Pic(W_{red})$ is an isomorphism, by
\cite[Prop. 1]{ballico}.
There exists a line bundles $\widetilde{L}$ over $W$, and an isomorphism $\bar{g}$
from the restriction of $\widetilde{L}$ to $[W\cap U]_{red}$ onto $\overline{L}$, by the reduced case
proven above. We get the short exact sequence
\[
0\rightarrow \widetilde{L}^{-1}\otimes \iota_*(\nu_L)\rightarrow
\widetilde{L}^{-1}\otimes \iota_*(L)\rightarrow 
\widetilde{L}^{-1}\otimes \iota_*(\overline{L}) \rightarrow 0
\]
and the long exact sequence
\[
H^0(W,\widetilde{L}^{-1}\otimes \iota_*(L))\rightarrow
H^0(W,\widetilde{L}^{-1}\otimes \iota_*(\overline{L}))\rightarrow
H^1(W,\widetilde{L}^{-1}\otimes \iota_*(\nu_L)).
\]
The right hand space vanishes, since $W$ is Stein.
Hence, there exists a lift of $\bar{g}$ to a homomorphism $g:\widetilde{L}\rightarrow \iota_*(L)$
of $\StructureSheaf{W}$-modules.
The restriction of $g$ to $W\cap U$ is an isomorphism of line-bundles, since it further
restricts to an isomorphism $\bar{g}$ over the reduced subscheme. 
It follows that $\iota_*(L)$ is isomorphic to $\iota_*\iota^*(\widetilde{L})$, which in turn
is isomorphic to $\widetilde{L}$, by Lemma \ref{lemma-F-is-isomorphic-to-the-pushforward-of-its-restriction}. 
We conclude that $\iota^*$ maps $\Pic(W)$ onto $\Pic(W\cap U)'$.
\end{proof}

\begin{prop}
\label{prop-azumaya-algebra-is-End-F}
Let $\A$ be a family of reflexive sheaves of Azumaya algebras over $\pi$.
There exists a family $\F$ of twisted reflexive sheaves over $\pi$, such that $\A$
is isomorphic to $\SheafEnd(\F)$ as $\StructureSheaf{\X}$-algebras. Furthermore, 
$\F$ is unique up to tensorization by a line bundle.
\end{prop}

The above statement  is well known when $\A$ is locally free
\cite[Theorem 1.3.5]{caldararu}
or when $S$ is a point.

\begin{proof}
\underline{Step 1:}
Let $W$ be an open subset of $\X$ and $\F_1$, $\F_2$ two families of reflexive sheaves over the
restriction of $\pi$ to $W$. Let $\iota:[W\cap U]\hookrightarrow W$ be the inclusion,
where $U:=\X\setminus \Z$ is the open subset where $\F$ is locally free. 
Assume that there exist isomorphisms $\varphi_1:\SheafEnd(\F_1)\rightarrow (\A\restricted{)}{W}$
and $\varphi_2:\SheafEnd(\F_2)\rightarrow (\A\restricted{)}{W}$
of unital algebras.
Then there exists a line bundle $L$ over $W\cap U$ and an isomorphism
$g:\iota^*\F_1\rightarrow \iota^*\F_2\otimes L$, such that 
$Ad_g:\iota^*\SheafEnd(\F_1)\rightarrow \iota^*\SheafEnd(\F_2)$
is the restriction of $\varphi_2^{-1}\varphi_1$, by 
the locally free case of the proposition \cite[Theorem 1.3.5]{caldararu}.

The push-forward of $L$ to $W$ is coherent, since $L$ is the restriction of the  kernel of
the homomorphism
\[
\SheafHom(\F_2,\F_1)\rightarrow 
\SheafHom(\restricted{\A}{W},\SheafHom(\F_2,\F_1))
\]
sending a homomorphism $f$ to the homomorphism $\tilde{f}$ given by 
$\tilde{f}(a):=f\varphi_2^{-1}(a)-\varphi_1^{-1}(a)f$.
Denote by the same letter $L$ also a the line bundle 
on $W$ restricting to $L$ on $W\cap U$, which exists by Lemma
\ref{lemma-isomorphism-of-Picards}.
We get the isomorphisms
\[
\F_1\cong \iota_*\iota^*\F_1\cong \iota_*\iota^*[\F_2\otimes L]\cong \F_2\otimes L,
\]
where the first and last isomorphisms follow from Lemma 
\ref{lemma-F-is-isomorphic-to-the-pushforward-of-its-restriction}
and the middle one is $\iota_*(g)$.

\underline{Step 2:}
There exists an open covering $\{W_\alpha\}$ of $\X$ and a family $\F_\alpha$, of reflexive
sheaves over the restriction of $\pi$ to $W_\alpha$, each
admitting an isomorphism $\varphi_\alpha:\SheafEnd(\F_\alpha)\rightarrow \A_\alpha$,
where $\A_\alpha$ is the restriction of $\A$ to $W_\alpha$, by Definition 
\ref{def-families-of-reflexive-azumaya-algebras}.
Set $W_{\alpha\beta}:=W_\alpha\cap W_\beta$ and $U_{\alpha\beta}:=W_{\alpha\beta}\cap U$.
Let $\iota_{\alpha\beta}:U_{\alpha\beta}\rightarrow W_{\alpha\beta}$ be the inclusion. 
There exist line bundles $L_{\alpha\beta}$ over $W_{\alpha\beta}$
and isomorphisms 
\[
g_{\alpha\beta}:(\F_\alpha\restricted{)}{W_{\alpha\beta}}
\rightarrow (\F_\beta\restricted{)}{W_{\alpha\beta}}\otimes L_{\alpha\beta},
\]
such that $Ad_{g_{\alpha\beta}}:\SheafEnd(\F_\alpha\restricted{)}{W_{\alpha\beta}}
\rightarrow \SheafEnd(\F_\beta\restricted{)}{W_{\alpha\beta}}$ and
$\varphi_\alpha^{-1}\varphi_\beta$ agree over $U_{\alpha\beta}$,
by step 1. The equality
\begin{equation}
\label{eq-Ad-g-equals-coboundary-of-varphi}
Ad_{g_{\alpha\beta}}=\varphi_\alpha^{-1}\varphi_\beta 
\end{equation}
over $W_{\alpha\beta}$ follows, by Lemma \ref{lemma-F-is-isomorphic-to-the-pushforward-of-its-restriction}.

We may choose the line bundle $L_{\beta\alpha}$ to be $L_{\alpha\beta}^{-1}$.
We claim that 
\begin{enumerate}
\item
the collection $\{L_{\alpha\beta}\}$ defines a gerbe over $\X$ \cite[Sec. 1.2]{hitchin}, and 
\item
the collection $\{(\F_\alpha,g_{\alpha\beta})\}$ defines a sheaf on this gerbe (i.e., a twisted sheaf
\cite{caldararu}).
\end{enumerate}
The first statement means  that the line bundles
$L_{\alpha\beta\gamma}:=L_{\alpha\beta}L_{\beta\gamma}L_{\gamma\alpha}$
over $W_{\alpha\beta\gamma}$
come with nowhere vanishing global sections $\theta_{\alpha\beta\gamma}$.
Furthermore, $\delta\theta$ is trivial, i.e., 
\begin{equation}
\label{eq-coboundary-of-theta-is-trivial}
\theta_{\alpha\beta\gamma}^{-1}\theta_{\alpha\beta\delta}
\theta_{\alpha\gamma\delta}^{-1}\theta_{\beta\gamma\delta}=1,
\end{equation}
as a section of the naturally
trivial line bundle $L_{\alpha\beta\gamma}^{-1}L_{\alpha\beta\delta}
L_{\alpha\gamma\delta}^{-1}L_{\beta\gamma\delta}$
over $W_{\alpha\beta\gamma\delta}$. 

The sections $\theta_{\alpha\beta\gamma}$ arrise as follows.
Set
\[
(\delta g)_{\alpha\beta\gamma}:=g_{\alpha\beta}g_{\beta\gamma}g_{\gamma\alpha}:
(\F_\alpha\restricted{)}{W_{\alpha\beta\gamma}} \rightarrow (\F_\alpha\restricted{)}{W_{\alpha\beta\gamma}} 
\otimes L_{\alpha\gamma\delta}.
\]
Equation (\ref{eq-Ad-g-equals-coboundary-of-varphi}) implies that 
$Ad_{(\delta g)_{\alpha\beta\gamma}}$ restricts to $U_{\alpha\beta\gamma}$
as the identity endomorphism of $(\F_\alpha\restricted{)}{U_{\alpha\beta\gamma}}$.
Hence, the isomorphism $(\delta g)_{\alpha\beta\gamma}$ restricts to $U_{\alpha\beta\gamma}$ 
as tensorization by a nowhere vanishing section $\theta_{\alpha\beta\gamma}$
of $L_{\alpha\beta\gamma}$. The section extends 
to a nowhere vanishing section of $L_{\alpha\beta\gamma}$ over $W_{\alpha\beta\gamma}$, 
denoted again by $\theta_{\alpha\beta\gamma}$, since $\depth_\Z(\StructureSheaf{\X})\geq 2$.
It follows that $(\delta g)_{\alpha\beta\gamma}$ is equal to tensorization by $\theta_{\alpha\beta\gamma}$
over $W_{\alpha\beta\gamma}$, by Lemma 
\ref{lemma-F-is-isomorphic-to-the-pushforward-of-its-restriction}.
The proof of Equation 
(\ref{eq-coboundary-of-theta-is-trivial}) is similar.

The fact that $(\F_{\alpha},g_{\alpha\beta})$
is a sheaf over the gerbe means that $(\delta g)_{\alpha\beta\gamma}$ is tensorization by 
$\theta_{\alpha\beta\gamma}$. This holds, by construction of $\theta_{\alpha\beta\gamma}$.
\end{proof}
%
\subsection{Condition \ref{cond-open-and-closed} 
(\ref{cond-item-same-homogeneous-bundle}) is open}

We keep the notation of section \ref{sec-families-of-reflexive-sheaves}.
Let $\A$ be a locally free Azumaya algebra over $\Y$. Set $\E:=\beta_*\A$. 
Given a point $s\in S$, denote the fibers of $\X$, $\Y$, $\Z$, $\D$ and $\beta:\Y\rightarrow \X$ over $s$ 
by $\X_s$, $\Y_s$, $\Z_s$, $\D_s$ and $\beta_s:\Y_s\rightarrow \X_s$.
Let $A_s$ and $E_s$ be the restrictions of $\A$ and $\E$ to the fibers over $s$.

\begin{lem}
\label{lemma-condition-involving-W-is-open}
Assume that $E_0$ satisfies Condition \ref{cond-open-and-closed} 
(\ref{cond-item-same-homogeneous-bundle}) for some point $0\in S$
and the morphism $\pi$ is proper.
Then there exists an open neighborhood $U$ of $0$ in $S$, such that $\E$ is a family of reflexive 
sheaves of Azumaya algebras over $\pi:\restricted{\X}{U}\rightarrow U$ and 
$E_s$ satisfies Condition \ref{cond-open-and-closed} 
(\ref{cond-item-same-homogeneous-bundle}), for all $s\in U$. 
\end{lem}

\begin{proof}
Let $\V$ be a (possibly twisted) locally free sheaf over $\Y$, such that $\A\cong \SheafEnd(\V)$. 
Denote its restrictions to the fibers over $s\in S$ by $V_s$.
Let $F_0$ be a reflexive sheaf such that $E_0\cong\SheafEnd(F_0)$. 
The assumption that $E_0$ satisfies Condition
\ref{cond-open-and-closed} (\ref{cond-item-same-homogeneous-bundle})
implies that $(\beta_0^*F_0)/\rm{torsion}$ is locally free over $\Y_0$ and
$\SheafEnd(V_0)\cong \SheafEnd((\beta_0^*F_0)/\rm{torsion})$.
We may thus assume that the restrictions of $\V$ to fibers $\D_{x_0}$ of $p:\D\rightarrow \X$ over $x_0\in\X_0$
are isomorphic to the restrictions of $(\beta_0^*F_0)/\rm{torsion}$ to $\D_{x_0}$,
possibly after replacing $\V$ by $\V(n\D)$, for some integer $n$. 
We conclude that $E_0$ is isomorphic to $\SheafEnd(\beta_{0,*}V_0)$, by
Corollary \ref{cor-pushforward-commutes-with-End} (applied in the case where the base $S$ is a point). 
We may thus replace $F_0$ by $\beta_{0,*}V_0$ and assume their equality. 

Set $\F:=\beta_*\V$.  
Given a point $x\in \X$, denote by 
by $W_x$ the restriction of $\V$ to $D_x$. Fix a point $x_0\in\X_0$.
We know that 
$W_{x_0}$ has the cohomological properties listed in Condition \ref{cond-open-and-closed} 
(\ref{cond-item-same-homogeneous-bundle}). 
Hence, there exists an open subset $\widetilde{U}$  of $\X$, containing $x_0$, such that 
the pairs $(D_{x},W_{x})$ and $(D_{x_0},W_{x_0})$ are isomorphic, 
for every $x\in \widetilde{U}$, 
since $W_{x_0}$ is assumed to be stable and unobstructed and the pair 
$(D_{x_0},W_{x_0})$ is assumed to be infinitesimally rigid in Condition \ref{cond-open-and-closed} 
(\ref{cond-item-same-homogeneous-bundle}).
The fiber $\X_0$ is contained in $\widetilde{U}$, since $E_0$ satisfies Condition \ref{cond-open-and-closed} 
(\ref{cond-item-same-homogeneous-bundle}).
Hence, there exists an open neighborhood $U$ of $0$ in $S$,
such that the pairs $(D_{x},W_{x})$ and $(D_{x_0},W_{x_0})$ are isomorphic, 
for every $x\in \X_s$, for all $s\in U$, since the morphism $\pi$ is proper.

$\F:=\beta_*\V$ is a family of reflexive sheaves over $\restricted{\X}{U}\rightarrow U$,
by Lemma \ref{lemma-sufficient-conditions-for-reflexivity-over-pi}
(\ref{lemma-item-push-forward-F-is-a-family-of-reflexive-sheaves}),
and $\beta_*\A=\beta_*\SheafEnd(\V)$ is isomorphic to $\SheafEnd(\E):=\SheafEnd(\beta_*\V)$,
by Corollary \ref{cor-pushforward-commutes-with-End}.
Furthermore, the homomorphism
$\beta^*\beta_*\V\rightarrow \V(-m\D)$ is surjective, where $m$ is the integer in 
Equation (\ref{eq-push-forward-of-V-same-as-that-of-V-minus-mD}), by
Lemma \ref{lemma-sufficient-conditions-for-reflexivity-over-pi}
(\ref{lemma-item-pullback-of-F-modulo-torsion-is-locally-free}).
It follows that $\beta_s^*F_s/(\mbox{torsion})$ is isomorphic to $V_s(-m D_s)$ and is hence locally free.

The $\beta$-tightness condition is clearly open.
\end{proof}

%
\section{Condition \ref{cond-open-and-closed} is preserved under twistor deformations}
\label{flatness-and-cond-1.7-on-twistor-lines} 

%
\subsection{Hyperholomorphic sheaves}
Let $M$ be a hyper-K\"{a}hler manifold. 
Then $M$ is endowed with a metric and an action of the algebra $\HH$  of quaternions on $T^{\RR} M$,
such that the group of unit quaternions acts via integrable complex structures, and the metric is K\"{a}hler
with respect to each of these complex structures. We will refer to these as {\em induced complex structures}.
The unit quaternions form a $2$-sphere $S^2$ and 
the induced complex structures on $M$ come from a complex structure on $\M:=M\times S^2$.
The projection on the second factor is holomorphic, yielding the {\em twistor family}
$\M\rightarrow \PP^1$ \cite{HKLR}. A point of $M$ determines a section of the projection $M\times S^2\rightarrow S^2$,
which is a holomorphic section of the twistor family $\M\rightarrow \PP^1$, called a {\em horizontal section}
\cite[Sec. 3(F)]{HKLR}.
If $M$ is an irreducible holomorphic symplectic manifold and 
$\omega$ a K\"{a}hler class on $M$, then the K\"{a}hler metric associated to $\omega$ is 
a hyper-K\"{a}hler metric \cite{beauville-varieties-with-zero-c-1}. 
In this case we will denote the twistor family by $\M\rightarrow \PP^1_\omega$ to emphasize 
 the dependence on the K\"{a}hler class $\omega$.

Let $M$ be a hyper-K\"{a}hler manifold, $I_0$ one of the induced complex structures, and $V$
a holomorphic vector bundle over $(M,I_0)$.
A {\em hyperholomorphic} connection on $V$ is a Hermitian connection $\nabla$, whose
curvature form 
is of Hodge type $(1,1)$ with respect to all the induced complex structures
\cite{kaledin-verbitski-book}. This means that the $(0,1)$-part of the connection, with respect to
every induced complex structure on $M$, is an integrable complex structure on $V$,
and $V$ extends to a holomorphic bundle $\V$ over the twistor family $\M$.

Given a reflexive 
coherent sheaf $F_{I_0}$
on $(M,I_0)$, Verbitsky defines the notion of 
a {\em hyperholomorphic} connection on $F_{I_0}$ as a 
connection over the open subset $U\subset M$ where $F_{I_0}$ is locally free, 
with a condition on the curvature and the associated metric along the singular locus 
$Z:=M\setminus U$ \cite[Definition 3.15]{kaledin-verbitski-book}. 
If the sheaf $F_{I_0}$ is stable, with respect to the 
K\"{a}hler form associated to $I_0$, and the first two Chern classes of $F_{I_0}$ remain of Hodge type 
with respect to all the induced complex structures, then $F_{I_0}$ admits a hyperholomorphic connection
\cite[Theorem 3.19]{kaledin-verbitski-book}. The singular locus $Z$ of  $F_{I_0}$
remains analytic with respect to all induced complex structures
\cite[Claim 3.16]{kaledin-verbitski-book}, and 
the reflexive coherent sheaf $F_I$ over $(M,I)$, 
for each induced complex structure $I$, is determined by the $(0,1)$-part with respect to $I$
of the hyperholomorphic connection over $U$.

Let $F_{I_0}$ be a reflexive coherent sheaf admitting a hyperholomorphic connection. 
Assume that $F_{I_0}$ has an isolated singularity at a point $x$ of $M$. 
We get a reflexive coherent sheaf $F_I$ on $(M,I)$, with an isolated singularity at $x$, 
for each complex structure $I$
in the twistor family. Let $\beta_I:\widetilde{M}_I\rightarrow (M,I)$ be the blow-up at the point $x$ and
$D_I:=\PP(T_x(M,I))$ the exceptional divisor in $\widetilde{M}_I$.

Denote by $\HH$ the subalgebra of global endomorphisms of the real tangent bundle $T^{\RR} M$,
which is isomorphic to the algebra of quaternions, associated to the hyper-K\"{a}hler structure. 
Note that $I$ is an element of $\HH$.
Let $\HH^*$ be the multiplicative group of invertible elements of $\HH$.
Given $g\in \HH^*$, set $g(I):=g I g^{-1}$. Then $g:(T^{\RR}_xM,I)\rightarrow (T^{\RR}_xM,g(I))$
is an isomorphism of complex vector spaces, which descends to an isomorphism
\[
\bar{g} : D_I \rightarrow D_{g(I)}
\]
of the exceptional divisors of $\widetilde{M}_I$ and $\widetilde{M}_{g(I)}$.

\begin{thm}
\label{thm-verbitsky-8.15}
\cite[Theorems 6.1 and  8.15]{kaledin-verbitski-book}
The sheaf $\widetilde{F}_I:=(\beta_I^*F_I)/{\rm torsion}$ is locally free, for all induced complex structures $I$.
Furthermore, $\bar{g}^*(\widetilde{F}_{g(I)}\restricted{)}{D_{g(I)}}$ is isomorphic to 
$(\widetilde{F}_I\restricted{)}{D_I}$, for all $g\in H^*$.
\end{thm}

\begin{proof}
The sheaf $\widetilde{F}_I$ is locally free, by \cite[Theorem 6.1]{kaledin-verbitski-book}.
Set $Q:=(T^{\RR}_xM\setminus 0)/\HH^*$. $Q$ is isomorphic to the quaternionic projective space. 
Let $\CC^*_I$ be the subgroup of $\HH^*$ associated to the complex structure $I$,
so that $D_I= (T^{\RR}_xM\setminus 0)/\CC^*_I$.
The quotient map $(T^{\RR}_xM\setminus 0)\rightarrow Q$ factors through the quotient map
\[
q_I : D_I\rightarrow Q.
\]
Verbitsky constructs a finite set of  complex vector bundles $\{B_j\}$ over $Q$, with special connections $\nabla_j$, 
and non-negative integers $k_j$  with the following properties. 
Endow the pullback $q_I^*B_j$ with the holomorphic structure associated to $q_I^*\nabla_j$. Then 
the direct sum $\oplus_{j}(q_I^*B_j)\otimes \StructureSheaf{D_I}(k_j)$
is isomorphic to $(\widetilde{F}_I\restricted{)}{D_I}$, for every complex structure $I$ in the twistor family
\cite[Theorem 8.15]{kaledin-verbitski-book}.
The isomorphism between 
$(\beta_I^*F_I\restricted{)}{D_I}$ and 
$\bar{g}^*(\beta^*_{g(I)} F_{g(I)}\restricted{)}{D_{g(I)}}$
would thus follow from the equality 
of $q_I^*(B_j,\nabla_j)$ and $\bar{g}^*q_{g(I)}^*(B_j,\nabla_j)$. The latter equality follows from 
the equality
$q_I=q_{g(I)}\bar{g}$.
\end{proof}

\hide{
\begin{lem}
\label{lemma-reflexivity-criterion}
Let $X$ be a smooth complex variety, $Z\subset X$ a smooth subvariety of co-dimension $\geq 3$,
$\beta:Y\rightarrow X$ the blow-up of $X$ with center $Z$, $D\subset Y$ the exceptional divisor, and 
$p:D\rightarrow Z$ the projection. Denote by $D_z$ the fiber of $p$ over a point $z\in Z$.
Let $V$ be a locally free sheaf over $Y$ satisfying the following condition
\begin{equation}
\label{eq-vanishing-of-global-sections-of-V-kD}
H^0(D_z,V(kD\restricted{)}{D_z})=0,
\end{equation}
for every point $z\in Z$ and for every positive integer $k$.
Then $\beta_*V$ is a reflexive sheaf over $X$. Assume, in addition, 
that the restriction of $V$ to $D_z$ is generated by $H^0(D_z,\restricted{V}{D_z})$.
Then the natural homomorphism $ev:(\beta^*\beta_*V)/{\rm torsion}\rightarrow V$ is an isomorphism.
\end{lem}

\begin{proof}
$V$ is torsion free and $\beta$ is surjective, so $\beta_*V$ is torsion free. 
It suffices to prove that the sheaf $\beta_*V$ is normal, by \cite[Prop. 1.6]{hartshorne-stable-reflexive}.
Given any open subset $U$ of $X$ and any  closed subset $Z'\subset U$, of co-dimension $\geq 2$,
we need to show that the restriction homomorphism 
$(\beta_*V)(U)\rightarrow (\beta_*V)(U\setminus Z')$ is an isomorphism. 
Let $U$ be an open subset of $X$. 
Now $\beta_*V$ is locally free over the complement of $Z$. Hence, it suffices to consider $Z'=Z\cap U$.
We have the equalities 
$(\beta_*V)(U)=V(\beta^{-1}(U))$
and $(\beta_*V)(U\setminus Z)=V(\beta^{-1}(U)\setminus D)$. 
It suffices to prove that the restriction homomorphism
\[
V(\beta^{-1}(U)) \rightarrow 
V(\beta^{-1}(U)\setminus D)
\]
is an isomorphism. Now $V(\beta^{-1}(U\setminus D)$ is the union of $[V\otimes\StructureSheaf{Y}(kD)](\beta^{-1}(U))$,
for all $k\geq 0$, and the union is equal to $V(\beta^{-1}(U))$, by the vanishing in Equality
(\ref{eq-vanishing-of-global-sections-of-V-kD}) above. The homomorphism $ev$ is injective,
as it is an isomorphism over $Y\setminus D$, and it is surjective, by the assumption that 
$\restricted{V}{D_z}$ is generated by its global sections.
\end{proof}
}

%
\subsection{Preservation of Condition \ref{cond-open-and-closed} (\ref{cond-item-same-homogeneous-bundle})}
\begin{prop}
\label{prop-flatness-of-the-twistor-deformation}
Let $E$ be a reflexive sheaf of Azumaya algebras over $X^d$ satisfying Condition 
\ref{cond-open-and-closed} (\ref{cond-item-same-homogeneous-bundle}).
Assume that $E$ is $\omega$-hyperholomorphic. 
Let $\pi:\X\rightarrow \PP^1_\omega$ be the twistor deformation of $X$.
Then the hyperholomorphic deformation $\E$ of $E$ is a (flat) family of reflexive sheaves of Azumaya algebras 
over  $\X^d_\pi\rightarrow \PP^1_\omega$ (Definition \ref{def-families-of-reflexive-azumaya-algebras}).
Furthermore, $\E_t$ satisfies Condition 
\ref{cond-open-and-closed} (\ref{cond-item-same-homogeneous-bundle}), for all $t\in \PP^1_\omega$.
\end{prop}

\begin{proof}
The sheaf $\E$ is reflexive, by Verbitsky's construction. 
$\E$ admits the structure of a reflexive sheaf of Azumaya algebras, extending that of $E$, 
by \cite[Lemma 6.5 (3)]{markman-hodge}.

\underline{Step 1:} We prove first the flatness of $\E$.
Let $\E_0\subset \E$ and $E_0\subset E$ be the kernels of the trace homomorphisms.
Flatness of $\E$ would follow once we prove that $\E_0$ is flat over $\PP^1_\omega$.

Let $F$ and $W$ be as in Condition 
\ref{cond-open-and-closed} (\ref{cond-item-same-homogeneous-bundle}).
Set 
\begin{equation}
\label{eq-locally-free-Azumaya-albebra-is-end-of-pullback}
A_0:=\SheafEnd_0(\beta^*F/{\rm torsion}).
\end{equation} 
Then
$\beta_*A_0$ is reflexive, by Lemma \ref{lemma-sufficient-conditions-for-reflexivity-over-pi} 
(\ref{lemma-item-push-forward-F-is-a-family-of-reflexive-sheaves}).
Now $E_0$ and $\beta_*A_0$ are isomorphic over the complement of $\Delta$.
Hence, $\beta_*A_0$ is isomorphic to $E_0$, as both are reflexive.

The vanishing in equation (\ref{eq-push-forward-of-End-V-same-as-that-of-End-V-minus-kD}) 
implies that the homomorphism 
$\beta_*(A_0\otimes\StructureSheaf{Y}(-kD))\rightarrow \beta_*A_0$ is an isomorphism.
Furthermore, the co-unit of the adjunction $\beta^*\vdash \beta_*$ induces an isomorphism
\begin{equation}
\label{eq-computation-of-pullback-of-E}
(\beta^*E_0)/{\rm torsion} = (\beta^*\beta_*[A_0\otimes \StructureSheaf{Y}(-kD)])/{\rm torsion} 
\rightarrow A_0\otimes \StructureSheaf{Y}(-kD),
\end{equation}
by Lemma \ref{lemma-sufficient-conditions-for-reflexivity-over-pi}
(\ref{lemma-item-pullback-of-F-modulo-torsion-is-locally-free}), and the assumption that 
$\SheafEnd_0(W)\otimes\StructureSheaf{D}(k)$ is generated by
global sections.

Let $\widetilde{\Delta}\subset \X^d_\pi$ be the relative diagonal and 
$\tilde{\beta}: \Y\rightarrow \X^d_\pi$ the blow-up of $\X^d_\pi$ centered at $\widetilde{\Delta}$.
Given a complex structure $I\in \PP^1_\omega$, denote by 
$\beta_I:\Y_I\rightarrow \X_I^d$ the blow-up of the diagonal $\Delta_I$ in $\X_I^d$.
Denote by $\D\subset \Y$ the exceptional divisor and let $D_I$ be its fiber over $I$.
Denote by $\E_{0,I}$ the restriction of $\E_0$ to $\X_I^d$.

The fiber $f_x$  of the differentiable projection $\X\cong X\times S^2\rightarrow X$
over a point $x\in X$ is a holomorphic section of the twistor family $\pi:\X\rightarrow \PP^1_\omega$, 
called a {\em horizontal section}
\cite[Sec. 3(F)]{HKLR}. 
The preimage of $f_x$, via the projection $\X^d_\pi\rightarrow \X$ onto the $d$-th factor, is the twistor deformation 
$\X^{d-1}_\pi\rightarrow \PP^1_\omega$. 
The restriction of $E_0$ to $X^{d-1}\times \{x\}$ is a hyperholomorphic reflexive sheaf
and the restriction of $\E_0$ to $\tilde{\pi}^{-1}(f_x)$ is its hyperholomorphic deformation along 
$\X^{d-1}_\pi\rightarrow \PP^1_\omega$.
The projection $\tilde{\pi}:\X^d_\pi\rightarrow \X$ onto the $d$-th factor is a smooth and proper morphism
with hyperk\"{a}hler fibers, each intersecting transversally the singular locus $\widetilde{\Delta}$ of $\E$
at a single point. 
Theorem \ref{thm-verbitsky-8.15} states that 
$\left[(\tilde{\beta}^*\E_0\restricted{)}{\tilde{\beta}^{-1}(\tilde{\pi}^{-1}(f_x))}\right]/\rm{torsion}$
is locally free. It follows that $(\tilde{\beta}^*\E_0)/\rm{torsion}$ is locally free over an open subset 
of $\X_\pi^d$ containing $\tilde{\beta}^{-1}(\tilde{\pi}^{-1}(f_x))$, by Nakayama's lemma. 
Theorem \ref{thm-verbitsky-8.15} thus applies to conclude that 
the pullback $\tilde{\beta}^*(\E_0)/{\rm torsion}$ is locally free, and 
the isomorphism class of 
its restriction to each fiber $D_z$, of $\D$ over $z\in \widetilde{\Delta}$, 
belongs to the same\footnote{$\Aut(D_z)$ acts via pullback, and 
an $\Aut(D_{z_1})$-orbit is equal to an $\Aut(D_{z_2})$-orbit, if there exists 
an isomorphism $D_{z_2}\rightarrow D_{z_1}$, which pulls back the  restriction of $\tilde{\beta}^*(\E_0)/{\rm torsion}$
to $D_{z_1}$ to the  restriction of $\tilde{\beta}^*(\E_0)/{\rm torsion}$
to $D_{z_2}$. The equality of the two orbits follows directly from Theorem \ref{thm-verbitsky-8.15}
when $z_1$ and $z_2$ both belong to $f_x$, for the same point $x$ in the special fiber $X$ of the twistor family.
The equality of the two orbits holds, by assumption, if $z_1$ and $z_2$ both belong to the diagonal $\Delta$ of $X^d$,
since $E$ is assumed to satisfy Condition 
\ref{cond-open-and-closed} (\ref{cond-item-same-homogeneous-bundle}). 
Hence it holds for all $z_1$ and $z_2$ in $\widetilde{\Delta}$.
} 
$\Aut(D_z)$-orbit.
This restriction is isomorphic to $\SheafEnd_0(W)\otimes \StructureSheaf{D_z}(k)$, 
by the computation for the special fiber in Equation (\ref{eq-computation-of-pullback-of-E}). 
It follows that the restriction of each $\E_{0,I}$, $I\in \PP^1_\omega$, to the fibers of $p:\D\rightarrow \widetilde{\Delta}$ 
has the properties in 
Condition \ref{cond-open-and-closed} (\ref{cond-item-same-homogeneous-bundle})
with the same $W$. Set
\begin{equation}
\label{eq-traceless-summand-of-A-is-a-twist-of-pullback-of-that-of-E}
\A_0:=(\tilde{\beta}^*\E_0/{\rm torsion}) \otimes \StructureSheaf{\Y}(kD).
\end{equation}
We get that both $\tilde{\beta}_*\A_0$ and $\tilde{\beta}_*[\A_0\otimes \StructureSheaf{\Y}(-kD)]$ 
are isomorphic to $\E_0$, by the same argument used 
for the special fiber.
The higher sheaf cohomologies of 
$\SheafEnd_0(W)\otimes \StructureSheaf{D_z}(k)$  vanish, by Equation  (\ref{eq-push-forward-of-End-V-minus-kD-is-flat}).
Hence, the higher direct image sheaves $R^i\tilde{\beta}_*[\A_0\otimes \StructureSheaf{\Y}(-kD)]$ vanish, for $i>0$,
and $\E_0$ is flat over $\PP^1_\omega$, by Lemma
\ref{lemma-push-forward-is-flat}.

\underline{Step 2:} 
We lift next the structure of an Azumaya algebra from $\E$ to the locally free sheaf $\A:=\A_0\oplus \StructureSheaf{\Y}$.
Such a structure exists over $U:=\Y\setminus \D$, as $\A$ coincides with the pullback of $\E$ over that open set.
We need to show that the multiplication homomorphism 
$m:(\A\otimes\A\restricted{)}{U}\rightarrow \restricted{\A}{U}$ extends to a regular homomorphism 
$\tilde{m}$ over $\Y$. The extension is clear for the restriction of $m$ to the summands $\A_0\otimes \StructureSheaf{\Y}$,
$\StructureSheaf{\Y}\otimes \A_0$, and $\StructureSheaf{\Y}\otimes \StructureSheaf{\Y}$.
Hence, it suffices to prove that the restriction of $m$ to 
$\A_0\otimes \A_0$ extends. The latter decomposes as the sum of two meromorphic sections:
\begin{eqnarray*}
\tilde{m}_1  :  \A_0\otimes \A_0 &\rightarrow & \StructureSheaf{\Y},
\\
\tilde{m}_2  :  \A_0\otimes \A_0 &\rightarrow & \A_0.
\end{eqnarray*}

Denote by $m_i$ the corresponding summand of $m$ over $U$. 
Then $m_1$ corresponds to the isomorphism $m_1:\E\rightarrow \E^*$.
The regularity of $\tilde{m}_1$ is equivalent to an extension of $m_1$ to an isomorphism 
$\tilde{m}_1:\A_0\rightarrow \A_0^*$. Note that the determinant line bundle $\det(\A_0)$
is trivial, as it restricts to a trivial line bundle over $U$ and over the special fiber $X^d$,
by the isomorphism in Equation (\ref{eq-computation-of-pullback-of-E}). 
Furthermore, the polar divisor of $\tilde{m}_1$ is disjoint from $U$ and from the special fiber.
The polar divisor must be empty, since $\D$ is irreducible and meets the special fiber.
Hence, $\tilde{m}_1$ extends to an isomorphism  over the whole of $\Y$.

The summand $m_2$ comes from a homomorphism $m_2$ in $\Hom(\E_0,\SheafHom(\E_0,\E_0))$.
Dualizing, we get the homomorphism $m_2^*$ in $\Hom(\SheafHom(\E_0,\E_0),\E_0)$.
The equality 
(\ref{eq-traceless-summand-of-A-is-a-twist-of-pullback-of-that-of-E}) induces a homomorphism
$\tilde{\beta}^*(m_2^*)$ in 
\[
\Hom(\tilde{\beta}^*\SheafHom(\E_0,\E_0),\A_0(-k\D)).
\]
Hence, the regularity of the extension  $\tilde{m}_2$ would follow, once we construct a regular homomorphism
\begin{equation}
\label{eq-a-regular-homomorphism-extending-the-identity}
\SheafHom(\A_0,\A_0(-k\D)) \rightarrow \tilde{\beta}^*\SheafHom(\E_0,\E_0)
\end{equation}
extending the identity homomorphism over $U$. We have the isomorphisms
\[
\Hom\left(\SheafHom(\A_0,\A_0(-k\D)), \tilde{\beta}^*\SheafHom(\E_0,\E_0)\right)\cong
\Hom\left(\tilde{\beta}_*\left[\omega_{\tilde{\beta}}\otimes \SheafHom(\A_0,\A_0(-k\D))\right],
\SheafHom(\E_0,\E_0)\right),
\]
by Grothendieck-Verdier duality. Now, $\omega_{\tilde{\beta}}$ 
is isomorphic to $\StructureSheaf{\Y}((c-1)\D)$, where $c$ is 
the codimension $(d-1)\dim(X)$ of $\Delta$ in $X^d$.
We have the isomorphisms
\begin{eqnarray*}
\tilde{\beta}_*\left[\SheafHom(\A_0,\A_0(c-1-k)\D)\right] &\cong &
\tilde{\beta}_*\left[\SheafHom(\tilde{\beta}^*\E_0,\A_0((c-1-2k)\D))\right]
\cong
\\
\SheafHom(\E_0,\tilde{\beta}_*[\A_0((c-1-2k)\D)]),
\end{eqnarray*}
where the first follows from Equation 
(\ref{eq-traceless-summand-of-A-is-a-twist-of-pullback-of-that-of-E}) and the second from the adjunction
$\tilde{\beta}^*\vdash\tilde{\beta}_*$. We have a regular homomorphism 
\[
\SheafHom(\E_0,\tilde{\beta}_*[\A_0((c-1-2k)\D)])\rightarrow 
\SheafHom(\E_0,\E_0),
\]
since $\tilde{\beta}_*[\A_0((j)\D)]$ is a subsheaf of $\E_0$ for every integer $j$, by 
Equation (\ref{eq-push-forward-of-End-V-same-as-that-of-End-V-minus-kD}).
This completes the construction of the homomorphism (\ref{eq-a-regular-homomorphism-extending-the-identity}),
and hence the proof of the regularity of $\tilde{m}_2$, as well as the construction of the multiplication $\tilde{m}$.
The axioms of an Azumaya algebra are satisfied by $\tilde{m}$, since they hold over the dense open subset $U$.
Hence, $\A$ is an Azumaya algebra over $\Y$.

\underline{Step 3:} We show next that $\E$ arrises via the construction of 
Corollary \ref{cor-pushforward-commutes-with-End}. This will complete the proof that $\E$ 
is a family of reflexive sheaves of Azumaya algebras over $\X^d_\pi\rightarrow \PP^1_\omega$.
Let $\V$ be a (possibly twisted) locally free sheaf over $\Y$, such that $\A$ is isomorphic to
$\SheafEnd(\V)$. Let $\W$ be the restriction of $\V$ to $\D$. 
Let $z$ be a point in the diagonal of the special fiber $X^d$ and $D_z$ the fiber of 
$\D$ over $z$. 
The assumption that $E$ satisfies Condition 
\ref{cond-open-and-closed} (\ref{cond-item-same-homogeneous-bundle})
states that the sheaf $\beta^*(F)/\rm{torsion}$ is  locally free over $Y$. 
The Azumaya algebra $A:=\SheafEnd(\beta^*(F)/\rm{torsion})$ is, 
by construction, the restriction of the Azumaya algebra 
$\A$ to $Y$. Hence, the restriction $W_z$ of $\V$ to $D_z$ is isomorphic to 
$W\otimes\StructureSheaf{D_z}(j)$, for some integer $j$, 
where $W$ is the restriction of $\beta^*(F)/\rm{torsion}$ to $D_z$. 
We may assume that $j=0$, possibly after replacing $\V$ by $\V(-j\D)$.
We have already seen that $\A$ restricts as the same Azumaya algebra $\SheafEnd(W)$
to the fibers of $p:\D\rightarrow \widetilde{\Delta}$, modulo pullback by automorphisms of the fibers.
Hence, $\V$ restricts as the same vector bundle $W$ to fibers of $p$, again 
modulo pullback by automorphisms of the fibers. 
The vector bundle $W$ is assumed to have the properties of 
Condition 
\ref{cond-open-and-closed} (\ref{cond-item-same-homogeneous-bundle}).
We conclude that $\V$ satisfies the assumptions of Corollary \ref{cor-pushforward-commutes-with-End}.
Set $\F:=\tilde{\beta}_*\V$. We conclude that $\F$ is a family of reflexive sheaves over
$\X^d_\pi\rightarrow \PP^1_\omega$ and $\beta_*\A\cong\SheafEnd(\F)$, by
Corollary \ref{cor-pushforward-commutes-with-End}. 
We already know that $\E$ is isomorphic to $\tilde{\beta}_*\A$. Consequently,
$\E$ is a family of reflexive sheaves of Azumaya algebras over 
$\X^d_\pi\rightarrow \PP^1_\omega$ (Definition \ref{def-families-of-reflexive-azumaya-algebras}).

The fibers
$\E_t$ satisfy Condition 
\ref{cond-open-and-closed} (\ref{cond-item-same-homogeneous-bundle}), 
except possibly the tightness condition,
for all $t\in \PP^1_\omega$,
since $(\tilde{\beta}^*\F)/\rm{torsion}$ is isomorphic to $\V(-m\D)$, 
by Lemma \ref{lemma-sufficient-conditions-for-reflexivity-over-pi}
(\ref{lemma-item-pullback-of-F-modulo-torsion-is-locally-free}). 
The tightness condition will be proven in section \ref{sec-local-triviality}.
\end{proof}

%
\subsection{Locally trivial families of reflexive Azumaya algebras}
\label{sec-local-triviality}
\hide{
\begin{lem}
\label{lemma-polystable-vector-bundles-of-the-same-slope}
Let $X$ be a K\"{a}hler manifold with  K\"{a}hler class $\omega$.
Let $V_1$ and $V_2$ be two $\omega$-polystable vector bundles of the same slope $\mu$ and
$f:V_1\rightarrow V_2$ a homomorphism. Then $\rm{Im}(f)$ and $\coker(f)$ are polystable vector bundles
of slope $\mu$. Furthermore, the short exact sequence 
$0\rightarrow \rm{Im}(f)\rightarrow V_2\rightarrow \coker(f)\rightarrow 0$ splits.
\end{lem}

\begin{proof}
The Lemma is well known. The case where $V_1$ is stable is \cite[Ch. 4 Prop. 7]{friedman}.
The general case follows by induction on the number of stable direct summands of $V_1$.
\end{proof}

\begin{lem}
\label{lemma-filtered-subsheaf-is-polystable}
Let $X$ be a K\"{a}hler manifold with  K\"{a}hler class $\omega$.
Let $V$ be an $\omega$-polystable vector bundle of slope $\mu$.
Let $\Gamma$ be a subsheaf of $V$ admitting a decreasing filtration 
\[
\Gamma=\Gamma_0\supset \Gamma_1 \supset \cdots \supset \Gamma_n\supset \Gamma_{n+1}=0,
\]
such that each graded summand $\Gamma_j/\Gamma_{j+1}$ is the image of a homomorphism
\[
\gamma_j:P_j\rightarrow V/\Gamma_{j+1},
\]
where $P_j$ is an $\omega$-polystable vector bundle of the same slope $\mu$ of $V$.
Then $\Gamma$ and $V/\Gamma$ are $\omega$-polystable vector bundles of slope $\mu$, 
and the following short exact sequence admits a splitting:
\begin{equation}
\label{eq-split-short-exact-seq-of-polystable-bundles}
0\rightarrow \Gamma\rightarrow V \rightarrow V/\Gamma\rightarrow 0.
\end{equation}
\end{lem}

\begin{proof}
It suffices to prove that $\Gamma$ is locally free and $\omega$-polystable of slope $\mu$. 
The statement would then follow from 
Lemma \ref{lemma-polystable-vector-bundles-of-the-same-slope} applied to the inclusion $\Gamma\rightarrow V$.
We prove these properties of $\Gamma$ by induction on the length $n$ of the filtration. 
The case $n=1$ follows from Lemma \ref{lemma-polystable-vector-bundles-of-the-same-slope}.
That Lemma also implies that $\Gamma_n$ and $V/\Gamma_n$ are $\omega$-polystable 
vector bundles of slope $\mu$ and 
the short exact sequence 
\[
0\rightarrow \Gamma_n\rightarrow V\rightarrow V/\Gamma_n\rightarrow 0
\]
admits a splitting $a_n:V/\Gamma_n\rightarrow V$.
Now $V/\Gamma_n$ admits a filtration of length $n-1$ with the above properties.
Hence, $\Gamma/\Gamma_n$ and 
$V/\Gamma\cong(V/\Gamma_n)/(\Gamma/\Gamma_n)$ are $\omega$-polystable of slope $\mu$
and the short exact sequence
\[
0\rightarrow \Gamma/\Gamma_n\rightarrow V/\Gamma_n\rightarrow V/\Gamma\rightarrow 0
\]
admits a splitting $a_{n-1}:V/\Gamma\rightarrow V/\Gamma_n$, by the induction hypothesis.
The composite homomorphism $a_n\circ a_{n-1}:V/\Gamma\rightarrow V$ provides the desired splitting of 
the short exact sequence (\ref{eq-split-short-exact-seq-of-polystable-bundles}).
Furthermore, $a_n:V/\Gamma_n\rightarrow V$ maps $\Gamma/\Gamma_n$ into $\Gamma$ providing a splitting of 
$0\rightarrow \Gamma_n\rightarrow\Gamma\rightarrow \Gamma/\Gamma_n\rightarrow 0$.
Hence, $\Gamma$ is locally free and $\omega$-polystable.
\end{proof}
}

Let $X$ be an irreducible holomorphic symplectic manifold,
$F$ a (possibly twisted) reflexive sheaf over $X^d$, and let $\beta:Y\rightarrow X^d$ be
the blow-up of the diagonal image $\Delta$ of $X$. 
Assume that $F$ satisfies Condition
\ref{cond-open-and-closed} (\ref{cond-item-same-homogeneous-bundle}), 
with the possible exception of the $\beta$-tightness condition.
Then $G:=(\beta^*F)/\rm{torsion}$ is locally free. 
Let $\D^1(G)$ be the sheaf of differential operators of order $\leq 1$ with scalar symbol.

\begin{lem}
\label{lemma-vanishing-of-sheaf-of-differential-operators}
\begin{enumerate}
\item
\label{lemma-item-sheaf-of-differential-operators-vanishes}
The sheaf $R^1\beta_*(\D^1(G))$ vanishes.
\item
\label{lemma-item-R-1-beta-pushforward-of-End-G-does-not-have-global-sections}
The vector space $H^0(X^d,R^1\beta_*\SheafEnd(G))$ vanishes.
\end{enumerate}
\end{lem}

\begin{proof}
(\ref{lemma-item-sheaf-of-differential-operators-vanishes})
Consider the short exact sequence of the symbol map
\[
0\rightarrow \SheafEnd(G)\rightarrow \D^1(G)\rightarrow TY\rightarrow 0.
\]
Let $e:D\rightarrow Y$ be the inclusion of the exceptional divisor $D:=\PP (N_{\Delta/X^d})$
and let $p:D\rightarrow \Delta$ be the natural morphism.
Let $\iota:\Delta\rightarrow X^d$ be the inclusion.
We have the short exact sequence 
\begin{equation}
\label{eq-short-exact-seq-of-differential-of-beta}
0\rightarrow  TY \rightarrow \beta^*T[X^d]\rightarrow e_*Q\rightarrow 0,
\end{equation}
where $Q:=(p^*N_{\Delta/X^d})/\StructureSheaf{D}(-1)$ is the tautological quotient bundle
(this is well known and it will be proven in detail in Diagram (\ref{diag-1})). 
Hence, $\beta_*(e_*Q)$ is isomorphic to $\iota_*N_{\Delta/X^d}$
and the homomorphism $\beta_*(\beta^*T[X^d])\rightarrow \beta_*(e_*Q)$ is surjective. 
The sheaf $R^1\beta_*(\beta^*T[X^d])$ vanishes.
The sheaf $R^1\beta_*TY$ thus vanishes. 
We get the exact sequences
\begin{equation}
\label{eq-gamma}
0\rightarrow \beta_*\SheafEnd(G)\rightarrow \beta_*\D^1(G)\rightarrow
\beta_*TY\RightArrowOf{\gamma} R^1\beta_*(\SheafEnd(G))\rightarrow 
R^1\beta_*(\D^1(G)) \rightarrow 0,
\end{equation}
\[
0\rightarrow \beta_*TY \rightarrow T[X^d]\rightarrow \iota_*N_{\Delta/X^d}\rightarrow 0.
\]
The latter 
is obtained via push-forward of the short exact sequence 
(\ref{eq-short-exact-seq-of-differential-of-beta}) above. 

It remains to prove that the homomorphism $\gamma$ in Equation (\ref{eq-gamma}) 
above is surjective. 
Consider the commutative diagram with short exact rows:
\[
\xymatrix{
0 \ar[r] & (\beta_*TY)\cdot I_\Delta \ar[d]\ar[r] & T[X^d]\otimes I_\Delta\ar[r]\ar[d]&
\iota_*\left[N_{\Delta/X^d}\otimes N^*_{\Delta/X^d}\right]\ar[r]\ar[d]& 0
\\
0 \ar[r] & \beta_*TY \ar[r] & T[X^d] \ar[r]&\iota_* N_{\Delta/X^d}\ar[r] & 0.
}
\]
The left and middle vertical homomorphisms are injective, and the right one vanishes. 
The snake lemma yields the long exact sequence
\[
0\rightarrow \iota_*\left[N_{\Delta/X^d}\otimes N^*_{\Delta/X^d}\right] 
\rightarrow \iota_*\iota^*\beta_*TY\rightarrow \iota_*\iota^*T[X^d]\rightarrow \iota_* N_{\Delta/X^d}\rightarrow 0.
\]
We get the short exact sequence over $\Delta$
\begin{equation}
\label{eq-short-exact-seq-of-push-forward-of-TY-restricted-to-the-diagonal}
0\rightarrow N_{\Delta/X^d}\otimes N^*_{\Delta/X^d}
\rightarrow \iota^*\beta_*TY\rightarrow T\Delta
\rightarrow 0.
\end{equation}

The restriction homomorphism $TY\rightarrow e_*e^*TY$ induces the homomorphism
\begin{equation}
\label{eq-pushforward-of-TY-via-beta-surjects}
\beta_*(TY)\rightarrow \beta_*e_*e^*TY=\iota_*p_*e^*TY.
\end{equation}
We claim that the latter homomorphism is surjective. Indeed $e^*TY$ fits in the short exact sequence 
\[
0\rightarrow TD\rightarrow e^*TY\rightarrow \StructureSheaf{D}(D)\rightarrow 0.
\]
Now $p_*(\StructureSheaf{D}(D))$ vanishes. Hence, $p_*e^*TY$ is isomorphic to
$p_*TD$. The tangent bundle of the exceptional divisor fits in
\[
0\rightarrow T_p\rightarrow TD \rightarrow p^*T\Delta\rightarrow 0.
\]
The sheaf $R^1p_*T_p$ vanishes. We get the short exact sequence
\[
0\rightarrow p_*T_p\rightarrow p_*e^*TY\rightarrow T\Delta\rightarrow 0.
\]
Note that $p_*T_p$ is naturally isomorphic
to $\SheafEnd_0(N_{\Delta/X^d})$.
We conclude the surjectivity of (\ref{eq-pushforward-of-TY-via-beta-surjects})
from comparison of the short exact sequence above to the short exact sequence
(\ref{eq-short-exact-seq-of-push-forward-of-TY-restricted-to-the-diagonal}).

The sheaf $R^i\beta_*(\SheafEnd(G)(-D))$ vanishes for $i>0$, by Equation 
(\ref{eq-push-forward-of-End-V-minus-kD-is-flat}) in
Condition \ref{cond-open-and-closed}. 
We get the isomorphism
\begin{equation}
\label{eq-first-higher-direct-image-of-End-G-via-beta-and-via-p-are-equal}
R^1\beta_*\SheafEnd(G)\cong \iota_*R^1p_*(\SheafEnd(e^*G))
\end{equation}
from the short exact sequence 
\[
0\rightarrow \SheafEnd(G)(-D)
\rightarrow \SheafEnd(G) \rightarrow e_*e^*\SheafEnd(G)\rightarrow 0.
\]
We conclude that the homomorphism $\gamma$ factors through the pushforward $\iota_*(\tilde{\gamma})$ of 
a homomorphism
$\tilde{\gamma}: p_*e^*TY\rightarrow R^1p_*(\SheafEnd(e^*G))$, via the surjective homomorphism 
(\ref{eq-pushforward-of-TY-via-beta-surjects}). It suffices to prove that $\tilde{\gamma}$ is surjective.
We have seen that $p_*e^*TY$ is naturally isomorphic to $p_*TD$. 
Denote by
\[
\tilde{a}:p_*T\D\rightarrow R^1p_*(\SheafEnd(e^*G))
\]
the pullback of $\tilde{\gamma}$. Clearly, $\tilde{a}$ corresponds to the infinitesimal pullback action on $e^*G$
of local automorphisms of the exceptional divisor $D$.

Let $W_z$ be the restriction of $G$ to the fiber $D_z$ of $p$ over $z\in \Delta$.
The assumed infinitesimal rigidity of the pairs $(D_z,W_z)$ in 
Condition \ref{cond-open-and-closed} (\ref{cond-item-same-homogeneous-bundle})
implies that $R^1p_*(\SheafEnd(e^*G))$ is locally free and 
the homomorphism
\begin{equation}
\label{eq-differential-a}
a:\SheafEnd_0(N_{\Delta/X^d})\RightArrowOf{} R^1p_*(\SheafEnd(e^*G))
\end{equation}
is surjective, where $a$ is the infinitesimal action of the automorphism groups of the fibers of $p$. 
Now $a$ is the restriction of $\tilde{a}$ and so $\tilde{a}$ is surjective as well. 
It follows that $\tilde{\gamma}$ is surjective, which implies that so is $\gamma$, by the reduction
established above. 

(\ref{lemma-item-R-1-beta-pushforward-of-End-G-does-not-have-global-sections})
The vector bundle $N_{\Delta/X^d}$ is isomorphic to $T\Delta\otimes_\ComplexNumbers \ComplexNumbers^{d-1}$
and thus  $SL(d-1)$ naturally embeds in the automorphism group of the normal bundle.
The bundle $\ker(a)$ contains the trivial $\LieAlg{sl}(d-1)$ subbundle of 
$\SheafEnd_0(N_{\Delta/X^d})$, by the assumed invariance of the bundles $W_z$
in Condition \ref{cond-open-and-closed} (\ref{cond-item-same-homogeneous-bundle}).
It follows that the homomorphism $H^0(\Delta,\ker(a))\rightarrow H^0(\Delta,\SheafEnd_0(N_{\Delta/X^d}))$
is surjective. The vector bundle $TX$ of an irreducible holomorphic symplectic manifold is stable
(with respect to every K\"{a}hler class). 
Hence, 
$\SheafEnd_0(N_{\Delta/X^d})$ is polystable. The kernel $\ker(a)$ is a subbundle of degree $0$,
by the $\LieAlg{sp}(T_xX)$-equivariance assumption in 
Condition \ref{cond-open-and-closed} (\ref{cond-item-same-homogeneous-bundle}).
Thus, $\SheafEnd_0(N_{\Delta/X^d})$ decomposes as the direct sum of the kernel and image of $a$. 
We conclude the vanishing of $H^0(R^1p_*(\SheafEnd(e^*G)))$. The vanishing of
$
H^0(X^d,R^1\beta_*\left(\SheafEnd(G)\right))
$
follows from Equation (\ref{eq-first-higher-direct-image-of-End-G-via-beta-and-via-p-are-equal}).
\hide{
The sheaf $R^1p_*(\SheafEnd(e^*G))$ is a polystable vector bundle over $\Delta$,
by the splitting of the short exact sequence (\ref{eq-differential-a}).
We have seen that $R^1\beta_*(\SheafEnd(G))$ is isomorphic to $\iota_*R^1p_*(\SheafEnd(e^*G))$.
Hence, the homomorphism $\gamma:\beta_*TY\rightarrow R^1\beta_*(\SheafEnd(G))$,
given in Equation (\ref{eq-gamma}), is equal to the pushforward $\iota_*(\tilde{\gamma})$,
for some  homomorphism
$\tilde{\gamma}:\iota^*(\beta_*TY)\rightarrow R^1p_*(\SheafEnd(e^*G))$. 
In particular, the image of $\gamma$ is the push-forward of the image of $\tilde{\gamma}$.
The vanishing of $H^0(X^d,R^1\beta_*(\D^1(G)))$ would follow form that of 
$H^0(\Delta,R^1p_*(\SheafEnd(e^*G)))$, 
once we show that the image of
$\tilde{\gamma}$ is a direct summand of $R^1p_*(\SheafEnd(e^*G))$.

Consider the descending filtration of the sheaf 
$\beta_*TY$ by 
\[
\beta_*TY\supset T[X^d]\otimes I_\Delta \supset T[X^d]\otimes I_\Delta^2 \supset \cdots 
\]
}
\end{proof}

We keep the notation of Proposition \ref{prop-flatness-of-the-twistor-deformation}.
In particular, $\E$ is the hyperholomorphic deformation of $E$ over the $d$-th fiber product of the twistor family
$\pi:\X\rightarrow\PP^1_\omega$. Let $\F$ be a family of reflexive sheaves over
$\varphi:\X^d_\pi\rightarrow \PP^1_\omega$, such that $\E\cong \SheafEnd(\F)$.
Let $0\in \PP^1_\omega$ be the point corresponding to the complex structure of $X$,
so that $\E_0$ is $E$. We have the blow-up morphism
$\tilde{\beta}:\Y\rightarrow \X^d_\pi$ and $\F$ is isomorphic to
$\tilde{\beta}_*\V$, where $\V$ is a locally free sheaf over $\Y$, such that $\tilde{\beta}^*\F/{\rm torsion}$
is isomorphic to $\V(-m\D)$, where $\D$ is the exceptional divisor and $m$ is the integer in
Equation (\ref{eq-push-forward-of-V-same-as-that-of-V-minus-mD}). Let $\W$ be the restriction of $\V$ to $\D$.

\begin{prop}
\label{prop-local-triviality}
Assume that $E$ satisfies Condition \ref{cond-open-and-closed}
(\ref{cond-item-same-homogeneous-bundle}), with the possible exception of the $\beta$-tightness condition. 
Then the family $\E$ is locally, in the topology of $\X^d_\pi$, trivial over $\PP^1_\omega$.
In other words, for every point $x\in\X^d_\pi$ there are open neighborhoods $U$ of $x$ and
$\overline{U}$ of $\varphi(x)$, an isomorphism $f:U\rightarrow \overline{U}\times (U\cap \X_{\varphi(x)}^d)$, 
such that $\varphi\circ f^{-1}$ is the projection to $\overline{U}$, 
and an isomorphism
\[
\tilde{f}:\iota_U^*\E\rightarrow \pi_2^*\iota_{\overline{U}}^*\E,
\]
where $\iota_U:U\rightarrow \X^d_\pi$ and $\iota_{\overline{U}}:U\cap \X_{\varphi(x)}^d\rightarrow \X^d_\pi$
are the inclutions and $\pi_2$ is the projection from $\overline{U}\times (U\cap \X_{\varphi(x)}^d)$ onto the second factor.
\end{prop}

\begin{proof}
$\E$ is locally free away from the diagonal. Hence, it suffices to prove the statement locally around
diagonal points of $\X^d_\pi$.
For every point $x\in \X$ there exists an open neighborhood $U_1\subset \X$, an open contractible 
neighborhood $\overline{U}\subset \PP^1_\omega$ of $\pi(x)$, an isomorphism
$f_1:\overline{U}\times (U_1\cap \X_{\pi(x)})\rightarrow U_1$, 
such that $\pi\circ f_1$ is the projection onto $\overline{U}$,
and an isomorphism
\begin{equation}
\label{eq-local-trivialization-of-W}
\tilde{f}_1:\pi_1^*\left(\restricted{\W}{p^{-1}(U_1\cap\X_{\pi(x)})}\right)\rightarrow \restricted{\W}{p^{-1}(U_1)}
\end{equation}
over the open subset $p^{-1}(U_1)$ of $\D$, as in Equation
(\ref{eq-trivializing-isomorphism-of-vector-bundles}), by Condition \ref{cond-open-and-closed} 
(\ref{cond-item-same-homogeneous-bundle}\ref{condition-item-local-triviality}).
We get the open neighborhood $U:=(U_1)^d_\pi$ of the diagonal image of $x$ in $\X^d_\pi$. 
We may assume that $U_1$ and $\overline{U}$ are Stein. Then so is $U$.

A theorem of Flenner and Kosarew states that
for each $t\in \overline{U}$ there exists a maximal analytic subspace $S\subset \overline{U}$ containing $t$, 
such that the restriction of $\E$ to the subspace $\varphi^{-1}(S)\cap U$ is isomorphic to a trivial
family, i.e., to the pullback of a coherent sheaf over $\X_t^d\cap U$ via the projection
$U\rightarrow \X_t^d\cap U$ induced by $f_1$ \cite[Remark 5.4(2)]{flenner-kosarew}. Now $\overline{U}$
is smooth and one-dimensional. Hence, either $S$ is a fat subscheme supported on the point $t$,
or $S$ contains an open neighborhood of $t$ in $\overline{U}$. 
Let $z$ be a local parameter at $t$ in $\overline{U}$ and let
$S_k\subset \overline{U}$ be ${\rm Spec}(\ComplexNumbers[z]/(z^{k+1}))$.
It suffices to prove that the restriction of $\E$ to the subspace $\varphi^{-1}(S_k)\cap U$ 
is isomorphic to a trivial family, for all $k>0$.


Set $\widetilde{U}:=\tilde{\beta}^{-1}(U)$. The trivialization $f_1$ of $U_1\rightarrow \overline{U}$
induces a trivialization of the blow-up $f:\overline{U}\times (\Y_{t}\cap \widetilde{U})\rightarrow \widetilde{U}$,
such that $\varphi\circ\tilde{\beta}\circ f$ is the projection onto $\overline{U}$.
Let $\psi:\widetilde{U}\rightarrow \Y_t\cap \widetilde{U}$ be the projection induced by $f$.
We have the left exact sequence of set valued sheaves
\[
0\rightarrow 
H^1(\Y_t\cap \widetilde{U},\psi_*GL_r(\StructureSheaf{\widetilde{U}})) \rightarrow 
H^1(\widetilde{U},GL_r(\StructureSheaf{\widetilde{U}})) \rightarrow 
H^0(\Y_t\cap \widetilde{U},R^1\psi_*GL_r(\StructureSheaf{\widetilde{U}})).
\]
The sheaf $R^1\psi_*GL_r(\StructureSheaf{\widetilde{U}}))$ is trivial, since the inverse image
via $\psi$ of a contractible Stein open subset is contractible and Stein, and every vector bundle over a 
contractible Stein manifold is trivial, by Grauert's Theorem \cite[Theorem A]{cartan}. 
Let $\V$ be a vector bundle over $\Y$, such that $\F$ is isomorphic to $\beta_*\V$ and $\W$ 
is the restriction of $\V$ to $\D$. 
We conclude that the restriction of $\V$ to $\widetilde{U}$ is represented by a cohomology class in 
$H^1(\Y_t\cap \widetilde{U},\psi_*GL_r(\StructureSheaf{\widetilde{U}}))$.

Set $\widetilde{U}_k:=\widetilde{U}\cap (\varphi\circ\tilde{\beta})^{-1}(S_k)$. 
Denote the restriction of $\psi$ by 
$
\psi_k:\widetilde{U}_k\rightarrow (\Y_t\cap \widetilde{U}).
$
Then the restriction of $\V$ to $\widetilde{U}_k$ is represented by a class $c_k$ in 
$H^1(\Y_t\cap \widetilde{U},\psi_{k,*}GL_r(\StructureSheaf{\widetilde{U}_k}))$.
It suffices to prove that $c_k$ belongs to the image of 
$H^1(\Y_t\cap \widetilde{U},GL_r(\StructureSheaf{\Y_t\cap\widetilde{U}}))$
via the pullback sheaf homomorphism
\[
\psi_k^* : GL_r(\StructureSheaf{\Y_t\cap\widetilde{U}})\rightarrow 
\psi_{k,*}GL_r(\StructureSheaf{\widetilde{U}_k}).
\]
The proof is by induction on $k$. The case $k=0$ is clear. 
Assume that $k>0$ and the statement holds for $k-1$.
Denote by $\V_0$ the restriction of $\V$ to $\Y_t\cap\widetilde{U}$. Note that $\V_0$ 
is represented by the class $c_0$.
We have the short exact sequence
\[
0\rightarrow 
z^k\cdot \LieAlg{gl}_r(\StructureSheaf{\widetilde{U}})\rightarrow 
\psi_{k,*}GL_r(\StructureSheaf{\widetilde{U}_k})\rightarrow
\psi_{k-1,*}GL_r(\StructureSheaf{\widetilde{U}_{k-1}}) \rightarrow 0.
\]
Two $1$-cocycles $\theta_k$ and $\tilde{\theta}_k$ 
of the sheaf $\psi_{k,*}GL_r(\StructureSheaf{\widetilde{U}_k})$
map to the same cocycle $\theta_{k-1}$ 
of the sheaf $\psi_{k-1,*}GL_r(\StructureSheaf{\widetilde{U}_{k-1}})$,
which restricts to $\Y_t\cap\widetilde{U}$ as 
a $1$-cocycle $\theta_0$ representing $\V_0$, 
if and only if  $\theta_k$ and $\tilde{\theta}_k$
differ by a $z^k$ multiple of a $1$-cocycle of $\SheafEnd(\V_0)$. On the level of cohomology one checks 
that fibers of the restriction map
\[
H^1(\Y_t\cap\widetilde{U},\psi_{k,*}GL_r(\StructureSheaf{\widetilde{U}_k}))\rightarrow
H^1(\Y_t\cap\widetilde{U},\psi_{k-1,*}GL_r(\StructureSheaf{\widetilde{U}_{k-1}}))
\]
are $H^1(\Y_t\cap\widetilde{U},\SheafEnd(\V_0))$-torsors.
The classes $c_k$ and $\psi_k^*c_0$ belong to the same fiber, by the induction hypothesis.

Let $\bar{\psi}_k:\D\cap \widetilde{U}_k\rightarrow \D_t\cap \widetilde{U}_k$ be the restriction of
$\psi_k$.
Let $\bar{c}_k$ be the restriction of the class $c_k$ to 
$H^1(\D_t\cap \widetilde{U},\bar{\psi}_{k,*}GL_r(\StructureSheaf{\D\cap \widetilde{U}}))$. 
The classes $\bar{c}_k$ and $\bar{\psi}_k^*\bar{c}_0$ are {\em equal}, since 
we have the trivialization $\tilde{f}_1$ given in (\ref{eq-local-trivialization-of-W}).
Hence, the difference between $c_k$ and $\psi_k^*c_0$ is a class in the kernel of the restriction homomorphism 
\[
H^1(\Y_t\cap\widetilde{U},\SheafEnd(\V_0))\rightarrow
H^1(\D_t\cap\widetilde{U},\SheafEnd(\W_0)).
\]
The latter fits as the middle vertical homomorphism in the following commutative diagram with left exact rows:
\[
\xymatrix{
H^1(\X_t^d\cap U,\beta_*\SheafEnd(\V_0)) \ar[r]\ar[d] &
H^1(\Y_t\cap \widetilde{U},\SheafEnd(\V_0)) \ar[r]\ar[d] &
H^0(\X_t^d\cap U,R^1\beta_*\SheafEnd(\V_0))\ar[d]
\\
H^1(\Delta_t\cap U,p_*\SheafEnd(\W_0))\ar[r] &
H^1(\D_t\cap\widetilde{U},\SheafEnd(\W_0)) \ar[r]&
H^0(\Delta_t\cap U, R^1p_*\SheafEnd(\W_0)).
}
\]
The spaces in the left column vanish, since $\Delta_t\cap U$ and $\X_t^d\cap U$ are Stein.
The right vertical homomorphism is an isomorphism, 
as shown in Equation (\ref{eq-first-higher-direct-image-of-End-G-via-beta-and-via-p-are-equal}).
It follows that the middle vertical homomorphism is injective, and the
classes $c_k$ and $\psi_k^*c_0$ are equal.
\end{proof}

{\bf Completion of the proof of Proposition \ref{prop-flatness-of-the-twistor-deformation}:}
It remain to prove that $\F_t$ is $\beta_t$-tight (Definition \ref{def-tightness}). 
Now the set of $t\in \PP^1_\omega$, where $\F_t$ is $\beta_t$-tight is open and non-empty,
since $E$ is assumed to satisfy Condition Condition \ref{cond-open-and-closed}
(\ref{cond-item-same-homogeneous-bundle}). The set of $t\in \PP^1_\omega$, where 
$\F_t$ is not $\beta_t$-tight, is open as well, by the
local triviality Proposition \ref{prop-local-triviality}. 
The latter set must be empty, since $\PP^1_\omega$ is connected.
\EndProof

%
\subsection{Completion of the proof of the preservation of Condition
\ref{cond-open-and-closed}}
Keep the notation of Proposition \ref{prop-flatness-of-the-twistor-deformation}.

\begin{prop}
\label{prop-vanishing-of-global-sections-of-extension-sheaf}
Assume that $E$ satisfies Condition \ref{cond-open-and-closed}
(\ref{cond-item-same-homogeneous-bundle}). Then
$H^0(\X_t^d,\SheafExt^1(\F_t,\F_t))$ vanishes, for all $t\in \PP^1_\omega$.
\end{prop}

\begin{proof} $\F_t$ satisfies Condition \ref{cond-open-and-closed}
(\ref{cond-item-same-homogeneous-bundle}), for all $t\in\PP^1_\omega$, by Proposition
\ref{prop-flatness-of-the-twistor-deformation}.
Set $F:=\F_t$.
Keep the notation of Lemma \ref{lemma-vanishing-of-sheaf-of-differential-operators}.
The functor $R\beta_*:D^b(Y)\rightarrow D^b(X^d)$ maps the object $G\otimes\StructureSheaf{Y}(-mD)$
to $F$, by Condition 
\ref{cond-open-and-closed} 
(\ref{cond-item-same-homogeneous-bundle}\ref{condition-item-push-forward-of-V-same-as-that-of-V-minus-mD})
and Lemma \ref{lemma-push-forward-is-flat}.
We get the induced homomorphism 
$\Hom(G\otimes\StructureSheaf{Y}(-mD),G\otimes\StructureSheaf{Y}(-mD)[1])\rightarrow \Hom(F,F[1])$,
and hence the homomorphism
$\Ext^1(G,G)\rightarrow \Ext^1(F,F)$. 
The analogous homomorphisms, over open subsets of $X^d$, induce the sheaf homomorphism
$\beta^1_*:R^1\beta_*(\SheafEnd(G))\rightarrow \SheafExt^1(F,F).$
The latter fits in the following commutative diagram with exact rows:
\begin{equation}
\label{eq-commutative-diagram-related-to-tightness}
\xymatrix{
0\ar[r] & 
\beta_*\SheafEnd(G) \ar[r]\ar[d]^{\cong} &
\beta_*\D^1(G) \ar[r]^{\beta_*(\sigma_G)}\ar[r] \ar[d]^{\beta_*}&
\beta_*TY \ar[r]^-\gamma\ar[d]_-{\beta_*(d\beta)} &
R^1\beta_*(\SheafEnd(G)) \ar[r] \ar[d]_{\beta_*^1}& 0
\\
0\ar[r] & 
\SheafEnd(F) \ar[r] &
\D^1(F) \ar[r]_{\sigma_F} &
TX^d \ar[r]_{a_F} & 
\SheafExt^1(F,F) \ar[r] & 0.
}
\end{equation}
The left vertical homomorphism is an isomorphism, by Lemma 
\ref{lemma-push-forward-of-End-V-is-End-F}.
The homomorphism $\gamma$ is surjective, by Lemma
\ref{lemma-vanishing-of-sheaf-of-differential-operators}(\ref{lemma-item-sheaf-of-differential-operators-vanishes}).
The equality $\ker(a_F)={\rm Im}(\sigma_F)$ is seen as follows.
We have the long exact sequence
\[
0\rightarrow \SheafEnd(F)\longrightarrow \widetilde{\D}^1(F)\LongRightArrowOf{\tilde{\sigma}_F} \SheafEnd(F)\otimes TX^d
\LongRightArrowOf{\tilde{a}_F} \SheafExt^1(F,F),
\]
by definition of $\widetilde{\D}^1(F)$ as $\SheafHom(J^1(F),F)$. Now, 
\[
{\rm Im}(\sigma_F)={\rm Im}(\tilde{\sigma}_F)\cap id_F\cdot TX^d =
\ker(\tilde{a}_F)\cap id_F\cdot TX^d =\ker(a_F).
\] 
The homomorphism $a_F$ is surjective and the middle vertical homomorphism $\beta_*$
in the above diagram is an isomorphism, by the $\beta$-tightness assumption in 
Condition \ref{cond-open-and-closed}
(\ref{cond-item-same-homogeneous-bundle}). 
Note that the three left vertical arrows are injective and the left two are surjective. 
The Snake Lemma (applied to the quotient of 
Diagram (\ref{eq-commutative-diagram-related-to-tightness}) by its left column)
yields that $\beta_*^1$ is injective and its co-kernel is $\iota_*N_{\Delta/X^d}$.
\begin{equation}
\label{eq-filtration-of-SheafExt-1-F-F}
0\rightarrow R^1(\beta_*\SheafEnd(G))\LongRightArrowOf{\beta^1_*} 
\SheafExt^1(F,F)\longrightarrow \iota_*N_{\Delta/X^d}\rightarrow 0.
\end{equation}
Now $H^0(N_{\Delta/X^d})$ vanishes, and $H^0(X^d,R^1(\beta_*\SheafEnd(G)))$
vanishes, by Lemma \ref{lemma-vanishing-of-sheaf-of-differential-operators}
(\ref{lemma-item-R-1-beta-pushforward-of-End-G-does-not-have-global-sections}). 
We conclude the desired vanishing of 
$H^0(X^d,\SheafExt^1(F,F))$.
\end{proof}

The following tightness criterion will be needed in Section \ref{sec-Example}.
\begin{lem}
\label{lemma-tightness-criterion}
Assume that $F$ satisfies Condition \ref{cond-open-and-closed}
(\ref{cond-item-same-homogeneous-bundle}), with the possible exception of $\beta$-tightness.
Then $F$ is $\beta$-tight, if and only if the sheaf $\SheafExt^1(F,F)$ fits in the short exact sequence 
(\ref{eq-filtration-of-SheafExt-1-F-F}), and the resulting isomorphism 
$\psi:\coker(\beta^1_*)\rightarrow \iota_*N_{\Delta/X^d}$ is the inverse of the homomorphism
$\bar{a}_F:\iota_*N_{\Delta/X^d}\rightarrow \coker(\beta^1_*)$ induced by the homomorphism $a_F$
in Diagram  (\ref{eq-commutative-diagram-related-to-tightness}).
\end{lem}

\begin{proof}
If $F$ is $\beta$-tight, then the sheaf $\SheafExt^1(F,F)$ fits in the short exact sequence 
(\ref{eq-filtration-of-SheafExt-1-F-F}), by the proof of 
Proposition \ref{prop-vanishing-of-global-sections-of-extension-sheaf}. Assume that 
$\SheafExt^1(F,F)$ fits in the short exact sequence 
(\ref{eq-filtration-of-SheafExt-1-F-F}) and $\psi^{-1}=\bar{a}_F$. 
Then $a_F$ is surjective and $\beta^1_*$ is injective. It follows 
the $\sigma_F$ induces an isomorphism from the 
cokernel of $\beta_*$ onto the kernel of 
$\bar{a}_F$, by the Snake Lemma applied to the quotient of
Diagram (\ref{eq-commutative-diagram-related-to-tightness}) by its left column.
The kernel of $\bar{a}_F$ vanishes, since $\bar{a}_F$ is an isomorphism.
Hence, $\beta_*:\beta_*\D^1(G)\rightarrow \D^1(F)$ is surjective.
\end{proof}

\hide{
\begin{proof}
Let $T_\varphi$ be the vertical tangent bundle of $\varphi$ and let $\widetilde{\Delta}\subset \X^d_\pi$
be the diagonal image of $\X$. 
The Atiyah class of $\F$ induces a sheaf homomorphism from $T[\X^d_\pi]$ to 
$\SheafExt^1(\F,\F)$, which restricts to a homomorphism
\[
a_\F \ : \ T_\varphi \  \rightarrow \ \SheafExt^1(\F,\F).
\]
We get the decreasing filtration 
$F^iT_\varphi:=T_\varphi\cdot I^i_{\widetilde{\Delta}}$, which is the relative analogue 
of the filtration $F^iTX^d$ used in Condition \ref{cond-atiyah-class}. 
We set 
$F^i\SheafExt^1(\F,\F)$ to be the image of $F^iT_\varphi$ via $a_\F$. 
The sheaves $\SheafExt^1(\F,\F)$ and $Gr^i\SheafExt^1(\F,\F)$ are flat over 
$\PP^1_\omega$ and they restrict to the fiber $\X^d_t$ of $\varphi$ as
$\SheafExt^1(\F_t,\F_t)$ and $Gr^i\SheafExt^1(\F_t,\F_t)$, by the local triviality statement in Proposition
\ref{prop-local-triviality}. $Gr^i\SheafExt^1(\F_0,\F_0)$ is assumed to be a locally free
$\StructureSheaf{\Delta_0}$-module,
by Condition 
\ref{cond-open-and-closed} (\ref{cond-item-same-homogeneous-bundle})
(see Condition \ref{cond-atiyah-class} (\ref{cond-item-graded-summands-of-sheaf-Ext-1-F-F-are-locally-free})).
Hence, $Gr^i\SheafExt^1(\F,\F)$ is a locally free
$\StructureSheaf{\widetilde{\Delta}}$-module, by Proposition
\ref{prop-local-triviality}. The homomorphism
$a_{\F_t}$ is assumed surjective for $t=0$, by Condition 
\ref{cond-open-and-closed} (\ref{cond-item-same-homogeneous-bundle})
(see Condition \ref{cond-atiyah-class} (\ref{cond-item-Atiyah-homomorphism-is-surjective})).
Hence, the homomorphism $a_\F$ is surjective and so is $a_{\F_t}$, for all $t\in \PP^1_\omega$,
by Proposition \ref{prop-local-triviality}. The graded summands $a_\F^i$ are thus surjective as well
and $\ker(a_\F^i)$ is a subbundle of $Gr^iT_\varphi$. Let $\tilde{\iota}:\X\rightarrow \widetilde{\Delta}$ be the 
diagonal embedding.
Note the isomorphism
\begin{equation}
\label{eq-graded-summands-of-tangent-sheaf-are-polystable}
Gr^iT_\varphi \ \ \cong \ \ \tilde{\iota}_*\left[T_\pi^{\oplus d}\otimes \Sym^i\left(T_\pi^{\oplus (d-1)}\right)\right].
\end{equation}
The kernel $\ker(a^i_{\F_t})$ has a trivial determinant line bundle,
by Condition \ref{cond-atiyah-class} (\ref{cond-item-kernel-of-Atiyah-homomorphism-is-a-subrepresentation}).
We have already seen that 
$\ker(a^i_\F)$ is a locally free $\StructureSheaf{\widetilde{\Delta}}$-module.
Hence, 
$\ker(a^i_{\F_t})$ has a trivial determinant line bundle, for all $t\in\PP^1_\omega$, by the triviality of $\Pic^0(\X_t)$. 
The vector bundle $Gr^iT\X_t^d$ is slope-polystable of slope $0$ (with respect to any K\"{a}hler class). 
Hence, $\ker(a^i_{\F_t})$ is a direct sum of slope-stable subbundles of $Gr^iT\X_t^d$ of slope $0$. 
It follows that the short exact sequence
\[
0\rightarrow \ker(a^i_{\F_t})\rightarrow Gr^iT\X_t^d \rightarrow 
Gr^i\SheafExt^1(\F_t,\F_t)\rightarrow 0
\]
splits, and $Gr^i\SheafExt^1(\F_t,\F_t)$ is also isomorphic to the direct sum of
slope-stable subbundles of $Gr^iT\X_t^d$ of slope $0$.
The isomorphism classes of slope-stable subbundles of $Gr^iT\X_t^d$ of slope $0$
are in one-to-one correspondence with irreducible representations of the 
holonomy group $Sp(\dim(\X_t))$ of $T\X_t$, by the isomorphism 
(\ref{eq-graded-summands-of-tangent-sheaf-are-polystable}) and the fact that $\X_t$ is hyper-K\"{a}hler. 
In particular, the isomorphism classes of 
slope-stable subbundles of $Gr^iT\X_t^d$ of slope $0$ form a finite set, constant with respect to $t\in \PP^1_\omega$.
The subset of isomorphism classes of slope-stable subbundles of $Gr^i\SheafExt^1(\F_t,\F_t)$ of slope $0$
is hence also constant with respect to $t\in \PP^1_\omega$. 
The vector space $H^0(\widetilde{\Delta}_t,Gr^i\SheafExt^1(\F_t,\F_t))$
vanishes, if and only if the trivial line bundle is {\em not} one of the slope-stable
direct summands of $Gr^i\SheafExt^1(\F_t,\F_t)$.
Hence, the vanishing of $H^0(\widetilde{\Delta}_t,Gr^i\SheafExt^1(\F_t,\F_t))$, for $t=0$,
implies the vanishing for all $t\in\PP^1_\omega$. The latter vanishing is assumed for $t=0$ in
Condition \ref{cond-atiyah-class} (\ref{cond-item-global-sections-of-graded-summands-vanish}).
Hence, $H^0(\X_t^d,\SheafExt^1(\F_t,\F_t))$ vanishes, for all $t\in \PP^1_\omega$.
\end{proof}
}

\begin{prop}
\label{prop-condition-remains-valid-under-twistor-deformations}
(Conditional on Conjecture \ref{conj-Ext-1-is-constant}). 
Let $E$ be a reflexive sheaf of Azumaya algebras satisfying Condition \ref{cond-open-and-closed}, and
$\omega$ a K\"{a}hler class, such that $E$ is $\omega$-slope-stable. Then every 
member $\E_t$, $t\in \PP^1_\omega$, of the twistor deformation of $E$ satisfies Condition
\ref{cond-open-and-closed}. 
\end{prop}

\begin{proof}
(\ref{cond-item-same-homogeneous-bundle}) 
This part was established in Proposition \ref{prop-flatness-of-the-twistor-deformation}.

(\ref{cond-item-characteristic-class}) The condition is topological, hence remains valid under 
the flat twistor deformation (Proposition \ref{prop-flatness-of-the-twistor-deformation}).

(\ref{cond-item-Brauer-class-is-theta-X-eta})
See \cite[Construction 6.7]{markman-hodge}.

(\ref{cond-item-infinitesimally-rigid}) 
Let $\F_t$ be a reflexive sheaf over $\X^d_t$ such that $\E_t\cong\SheafEnd(\F_t)$. 
We have the left exact sequence
\[
0\rightarrow H^1(\X_t^d,\E_t)\rightarrow \Ext^1(\F_t,\F_t)\rightarrow H^0(\X_t^d,\SheafExt^1(\F_t,\F_t)).
\]
The vanishing of $Ext^1(\F_t,\F_t)$ 
follows from Conjecture \ref{conj-Ext-1-is-constant}
and Proposition \ref{prop-vanishing-of-global-sections-of-extension-sheaf}.

(\ref{condition-stability}) This fact is due to Verbitsky \cite[Cor. 6.10]{kaledin-verbitski-book,markman-hodge}.
\end{proof}

We provide next a criterion for the reduction of Conjecture \ref{conj-Ext-1-is-constant}
to Conjecture \ref{conj-Ext-1-is-constant-isolated-singularity-case} in case $d=2$.
Let $E$ be a reflexive sheaf of Azumaya algebras over $X\times X$, satisfying Condition
\ref{cond-open-and-closed}. Let $F$ be a twisted reflexive sheaf, such that $E\cong\SheafEnd(F)$.
Given a point $x\in X$, denote by $F_x$ the restriction of $F$ to $\{x\}\times X$.
Let 
\begin{equation}
\label{eq-Kodaira-Spencer-homomorphism}
\kappa_x:T_xX\rightarrow \Ext^1(F_x,F_x)
\end{equation} 
be the Kodaira-Spencer homomorphism.

\begin{lem}
\label{lem-surjectivity-of-Kodaira-Spencer-implies-reduction-of-conjecture}
\begin{enumerate}
\item
\label{lemma-item-surjectivity-of-Kodaira-Spencer-implies-vanishing}
Assume that $E$ is not locally free along the diagonal and that
the Kodaira-Spencer homomorphism $\kappa_x$ 
 is surjective, for all $x\in X$. Then $H^1(E_x)$ vanishes for all $x\in X$.
\item
\label{lemma-item-vanishing-of-H-1-for-E-x-implies-that-of-E}
Assume that $F_x$ is $\omega$-slope stable 
and $H^1(E_x)$ vanishes for all $x\in X$.
Then Conjecture \ref{conj-Ext-1-is-constant} for $E$ follows from Conjecture 
\ref{conj-Ext-1-is-constant-isolated-singularity-case}.
\end{enumerate}
\end{lem}

\begin{proof}
(\ref{lemma-item-surjectivity-of-Kodaira-Spencer-implies-vanishing})
Let $\beta:Y\rightarrow X\times X$ be the blow-up of the diagonal $\Delta$. 
$F$ is $\beta$-tight, by Condition \ref{cond-open-and-closed}.
Let $e_1:T\Delta\rightarrow (TX^2\restricted{)}{\Delta}$ be the inclusion as the first direct summand
({\em not} as the subbundle tangent to the diagonal).
Set $G:=\beta^*F/\mbox{torsion}$. Let
$\beta^1_*$ be the homomorphism in Diagram (\ref{eq-commutative-diagram-related-to-tightness}).
The composition 
\[
T\Delta\RightArrowOf{e_1} 
(TX^2\restricted{)}{\Delta} \rightarrow N_{\Delta/X^2}
\IsomRightArrow \coker(\beta^1_*)
\]
is an isomorphism, since the right homomorphism is an isomorphism, 
by the short exact sequence (\ref{eq-filtration-of-SheafExt-1-F-F}).
Restriction of the above isomorphism to $\{x\}\times X$ yields 
the isomorphism
\[
T_xX\IsomRightArrow  (N_{\Delta/X^2})_x \IsomRightArrow
H^0(X,\SheafExt^1(F_x,F_x)/\beta^1_*[R^1\beta_*\SheafEnd(G_x)]),
\]
where $G_x$ is the restriction of $G$ to the strict transform in $Y$ of $\{x\}\times X$
and $\beta$ now denotes the blow-up of $X$ at $x$. 
The above isomorphism coincides with the composition
\[
T_xX\RightArrowOf{\kappa_x} \Ext^1(F_x,F_x)\RightArrowOf{\psi} H^0(X,\SheafExt^1(F_x,F_x))\rightarrow 
H^0(X,\SheafExt^1(F_x,F_x)/\beta^1_*[R^1\beta_*\SheafEnd(G_x)]).
\]
Hence, $\kappa_x$ is injective, and so an isomorphism. It follows that 
the natural homomorphism $\psi$ above is injective. The cohomology 
$H^1(X,\SheafEnd(F_x))$ thus vanishes, by the left exactness of the sequence
\[
0\rightarrow H^1(X,\SheafEnd(F_x))
\rightarrow Ext^1(F_x,F_x)\RightArrowOf{\psi} H^0(X,\SheafExt^1(F_x,F_x)).
\]

(\ref{lemma-item-vanishing-of-H-1-for-E-x-implies-that-of-E})
Set $E_x:=\SheafEnd(F_x)$.
The class $c_2(E_x)$ is invariant with respect to a finite index subgroup of $Mon(X)$, since $c_2(E)$ is
by Condition \ref{cond-open-and-closed}. Hence, $E_x$ is $\omega$-polystable hyperholomorphic
as a sheaf, and $\omega$-stable as an Azumaya algebra, by the assumed slope-stability of $F_x$.
Let $\E$ be the sheaf of Azumaya algebras over the fiber square $\X^2_\pi$
of the twistor family $\pi:\X \rightarrow \PP^1_\omega$.
The point $x\in X$ determines a holomorphic section $\PP^1_\omega\rightarrow\X$ and hence also
a holomorphic embedding  $\iota_x:\X\hookrightarrow\X^2_\pi$ \cite[Sec. 3(F)]{HKLR}.
Let $\E_{x,t}$ be the restriction of $\iota_x^*(\E)$ to the fiber $\X_t$
over $t\in \PP^1_\omega$. We have seen above that  $H^1(X,E_x)$ vanishes.
Conjecture \ref{conj-Ext-1-is-constant-isolated-singularity-case} implies that 
$H^1(\X_t,\E_{x,t})$ vanishes, for all $t\in \PP^1_\omega$. 

Let $\pi_1:\X_t\times\X_t\rightarrow\X_t$ be the projection onto the first factor. We 
get the left exact sequence
\[
0\rightarrow H^1(\X_t,\pi_{1,*}\E_t) \rightarrow 
H^1(\X_t^2,\E_t)\rightarrow 
H^0(\X_t,R^1\pi_{1,*}\E_t).
\]
The sheaf $R^1\pi_{1,*}\E_t$ vanishes, by the vanishing of $H^1(\X_t,\E_{x,t})$, for all $x\in \X_t$.
The sheaf $\pi_{1,*}\E_t$ is the trivial line bundle, by the $\omega_t$-slope stability of $\E_{x,t}$ as a
reflexive sheaf of Azumaya algebras (which follows from the fact that $F_x$ is
$\omega$-stable hyperholomorphic). Hence, $H^1(\X_t,\pi_{1,*}\E_t)$ vanishes.
We conclude the desired vanishing of $H^1(\X_t^2,\E_t)$, for all $t\in \PP^1_\omega$.
\end{proof}

\begin{condition}
\label{enhanced-open-and-closed-condition}
$E$ is a reflexive sheaf of Azumaya algebras over $X\times X$, which satisfies Condition \ref{cond-open-and-closed}, 
the restriction $E_x$ of $E$ to $\{x\}\times X$  is $\omega$-slope-stable 
for all $x\in X$, for some K\"{a}hler class $\omega$ on $X$, 
and $H^1(X,E_x)$ vanishes, for all $x\in X$. 
The class $\tilde{\theta}$ in $(\Lambda/r\Lambda)^{\oplus 2}$, used in 
Equation (\ref{eq-theta-X-eta}), projects to a class of order $r$ in the second direct summand.
\end{condition}

\begin{rem}
\label{rem-enhanced-condition}
When $d=2$ and the sheaf $E$ is not locally free along the diagonal 
we can require Condition \ref{enhanced-open-and-closed-condition}, which is an 
enhancement of  Condition \ref{cond-open-and-closed}.
Then the proof of Proposition \ref{prop-condition-remains-valid-under-twistor-deformations}
combines with  Lemma \ref{lem-surjectivity-of-Kodaira-Spencer-implies-reduction-of-conjecture}
(\ref{lemma-item-vanishing-of-H-1-for-E-x-implies-that-of-E}) 
to show that the enhanced Condition \ref{enhanced-open-and-closed-condition}
is preserved under twistor deformations,
provided Conjecture \ref{conj-Ext-1-is-constant-isolated-singularity-case} holds.
The condition about the class $\tilde{\theta}$ will be used later 
to show that Condition \ref{enhanced-open-and-closed-condition}
is preserved along generic twistor paths.
\end{rem}
%
\section{Existence of a coarse moduli space of rigid Azumaya algebras}
\label{sec-existence-of-coarse-moduli-of-triples}
We present a proof in this section of the 
existence of the coarse moduli space
$\tfM_\Lambda$ consisting of isomorphism classes of triples 
$(X,\eta,E)$ satisfying  Condition \ref{cond-open-and-closed}. 
Our construction is rather {\it ad-hoc}. No doubt it would be more satisfying to 
realize $\tfM_\Lambda$ as an open subset of a moduli space of 
simple Azumaya algebras, but we were unable to find a reference for such a 
construction in our analytic setting. 

%
\subsection{Local moduli of locally free Azumaya algebras over the blow-up of the diagonal}
\label{sec-local-moduli}

\begin{defi}
Let $f:V \to W$ be a morphism of smooth and compact analytic spaces. Define 
$\Def(V,f,W):  ({\mathbf{GAn}})\to ({\mathbf{Sets}})$ to be the deformation functor from 
the category of complex analytic germs to sets, which associates 
to a germ $S$ the set of isomorphism classes of deformations $\wt{f}:\wt{V}\to\wt{W}$ over $S$ of
$f$.
\end{defi}

We shall consistently use the following notation below. For a complex analytic manifold $V$,  
$\Def(V)$
will stand for the set-valued functor of flat deformations of $V$ over complex analytic germs.
The Kuranishi deformation space of $V$ will be denoted $\D_V$.

\begin{lem}
\label{isom-1}
Let $f:V \to W$ be a morphism as above satisfying the conditions
\begin{enumerate}
\item 
$R^0f_*\ko_V \cong \ko_W$, $R^if_*\ko_V =0$ if $i>0$;
\item 
$H^0(W,T_W)=0$.
\end{enumerate}
Assume that the Kuranishi deformation family of  $V$ is universal. 
Then, the forgetful morphism $\alpha: \Def(V,f,W)\rightarrow \Def(V)$
of the two deformation functors  is an isomorphism. In particular, $\Def(V,f,W)$
is representable by a universal deformation family.
\end{lem}
\begin{proof}
 The morphism $\alpha$ is formally smooth, by hypothesis (1) and  \cite[Theorem 3.3]{ran}. 
Furthermore, the relative dimension of $\alpha$ 
is  $\dim (H^0(V,f^*T_W))$ (\cite[(2.2)]{ran}), which is
$0$ by hypothesis (2).

Denote by $\D_V$ and $\D_W$ the Kuranishi deformation spaces and 
let $v: \V \to \D_V$ and $w: \W \to \D_W$ be the Kuranishi families of $V$ and $W$. According to 
\cite[Prop. 11.4]{flips}, there are natural morphisms $F$ and $\hat{f}$ making the following
diagram commutative
\begin{equation}\label{family}
\xymatrix{
\V \ar[d]_{v} \ar[r]^{\hat{f}} & \W \ar[d]^{w}
\\
\D_V \ar[r]^{F} & \D_W,
}
\end{equation}
and such that $\hat{f}|_V=f$. 
Consequently, we get a family of morphisms 
\[
\xymatrix{
\V \ar[rr]^{\wt{f}} \ar[rd]& & F^*\W \ar[ld]\\
 &\D_V.&
}
\]
We claim that the above family 
represents $\Def(V, f, W)$, which, of course, also implies the lemma. To see this, let us 
recall the construction of the family (\ref{family}). Denote by $\Gamma\subset V\times W$ 
the
graph of the morphism $f$, and let $D$ be the component of the relative Douady space
of $\V\times \W / \D_V\times \D_W$ parametrizing graphs of morphisms which contains 
$\Gamma$. Kollar-Mori argue that the  projection morphism 
$$pr: [\Gamma]\in D \to 0 \in \D_V$$
admits an analytic section $\sigma:\D_V\rightarrow D$, which yields (\ref{family}). Now, let
$\wt{V}\stackrel{g}{\to} \wt{W} \to T$ be a deformation of the morphism $f$ over a base $T$.
This corresponds to a section $\tilde{\sigma}$ of $pr$ over $T$, in the sense that it corresponds to a 
commutative diagram
\[
\xymatrix{& D\ar[d]^{pr}\\
T\ar[ru]^{\tilde{\sigma}} \ar[r]_\kappa& \D_V}
\]
But, as the morphism
$\alpha$ is formally smooth of relative dimension zero, the section $\tilde{\sigma}$
must coincide with $\sigma\circ \kappa$ in the formal neighborhood of $0\in T$. Thus, the family
$\wt{V}\stackrel{g}{\to} \wt{W} \to T$ is isomorphic to the pullback of the family (\ref{family}) via
the classifying morphism. This completes the proof of the claim.
\end{proof}

As in section \ref{sec-introduction}, let $E$ be a reflexive sheaf of Azumaya algebras on $X^d$ 
satisfying Condition \ref{cond-open-and-closed}, 
$F$ a reflexive twisted sheaf representing $E$, 
and $A=\E nd (\beta^*F/tor)$  the 
locally-free sheaf of Azumaya algebras on $Y$, the blow-up of $X^d$
along the small diagonal. 

\begin{lem}
\label{lemma-A-is-rigid-if-E-is}
$A$ is rigid as a sheaf of Azumaya algebras. 
\end{lem}

\begin{proof}
$E$ is rigid by assumption, and
$\beta_*A=E$ from Corollary \ref{cor-pushforward-commutes-with-End}. 
The space $H^0(X^d,R^1\beta_*A)$
vanishes, by Lemma \ref{lemma-vanishing-of-sheaf-of-differential-operators} 
(\ref{lemma-item-R-1-beta-pushforward-of-End-G-does-not-have-global-sections}).
The Leray spectral sequence yields the left exact sequence
\[
0\rightarrow H^1(X^d,\beta_*A)\rightarrow H^1(Y,A)\rightarrow H^0(X^d,R^1\beta_*A).
\]
We conclude the vanshing of $H^1(Y,A)$.
\end{proof}


\begin{defi}
Let $A=\E nd (\beta^*F/tor)$ be as above. Define
$\Def(Y,A): ({\mathbf{GAn}})\to ({\mathbf{Sets}})$ to be the deformation functor 
from the category of complex analytic germs to sets, which associates 
to a germ $S$ the isomorphism classes of deformations $(\wt{Y},\wt{A})$ over $S$ of 
the pair $(Y,A)$.
\end{defi}


\begin{lem}\label{isom-2}
The natural morphism $\Def(X)\rightarrow\Def(Y)$ is an isomorphism, as is 
the forgetful morphism $\Def(Y,A)\rightarrow \Def(Y)$. 
Consequently, we get the isomorphism 
$\Def(Y,A)\rightarrow  \Def(X)$ of the two deformation functors.
\end{lem}

\begin{proof} 
Denote by $f:\PP_A\to Y$ the projective bundle on $Y$ associated to $A$. 
We shall construct isomorphisms  between the following pairs of deformation functors:
\begin{enumerate}
\item 
\label{1}
$\Def(Y)$ and $\Def(X)$,
\item 
\label{3}
$\Def(Y,A)$ and $\Def(\PP_A,f,Y)$,
\item 
\label{4}
$\Def(\PP_A, f, Y)$ and $\Def(\PP_A)$,
\item 
\label{5}
$\Def(\PP_A)$ and  $\Def(Y)$,
\end{enumerate}
such that the composition 
$\alpha_{\ref{1}}\circ\alpha_{\ref{5}}\circ\alpha_{\ref{4}}\circ\alpha_{\ref{3}}:\Def(Y,A)\rightarrow \Def(X)$ is the forgetful morphism, where $\alpha_i$ denotes the morphism from the first deformation functor to the second in
the $i$-th pair.


(\ref{1}) The obvious morphism from $\Def(X)$ to $\Def(Y)$ will be shown to
be an isomorphism. Most of the proof involves the cohomological identification of the 
differential $d\alpha_1$.
Let $K$ be  the sheaf of  germs of tangent vectors to $X^d$ that are tangent to the small
diagonal $\gd$. The sheaf $K$ controls the
deformations of the embedding $\gd \subset X^d$: $H^1(X^d,K)$ is the 
tangent space of $\Def(\gd, i, X^d)$, and $H^2(X^d,K)$ its obstruction space
\cite[3.4.4]{sernesi}. 
The Cartesian square
$$
\xymatrix{
D  \ar[r] \ar[d]_{p}& Y \ar[d]^{\beta}
\\
\gd \ar[r]^{i} & X^d
}
$$
gives rise to the diagram
\begin{equation}\label{diag-1}
\xymatrix{
 & & & 0\ar[d] &\\
  &0 \ar[d] & & \ko(D)|_D\ar[d] &\\
0 \ar[r] & \beta^*K/tor \ar[r] \ar[d] & \beta^*T_{X^d} \ar[r] \ar@{=}[d] & 
p^* N_{\gd/X^d} \ar[d] \ar[r] & 0\\
0 \ar[r] & T_Y \ar[r] \ar[d]&  \beta^*T_{X^d} \ar[r]  & \E xt ^1(\Omega_\beta, \ko_Y) 
\ar[r] \ar[d]& 0\\ 
& \ko(D)|_D \ar[d]& & 0 & \\
 & 0. & & & 
}
\end{equation}
Note that $\Omega_\beta$ fits in the short exact sequence 
$0\rightarrow \beta^*\Omega_{X^d}
\RightArrowOf{d^*\beta}\Omega_Y\rightarrow \Omega_\beta\rightarrow 0$,
which explains the lower horizontal sequence above. 
Taking cohomology, (\ref{diag-1}) yields the diagram \footnote{Note that the canonical morphism
$K\to\beta_*(\beta^*K/tor)$ is an isomorphism.} 

\begin{equation}\label{diag-2}
\xymatrix{
0 \ar[r] & H^1(K) \ar[r]\ar[d]_{\cong} & H^1(T_{X^d}) \ar[r] \ar@{=}[d] & 
H^1(N_{\gd/X^d}) \ar[d]^{\cong} \ar[r] & 0\\
0 \ar[r] & H^1(T_Y) \ar[r]  & H^1(T_{X^d}) \ar[r] & 
H^1(\E xt ^1(\Omega_\beta, \ko_Y)) \ar[r]  & 0
}
\end{equation}
with exact rows, as we now explain.
The  top row is exact, since $H^0(D,p^*N_{\gd/X^d})$ is isomorphic to 
$H^0(\gd,N_{\gd/X^d})$, which vanishes, and the homomorphism 
$H^1(X^d,T_{X^d})\rightarrow H^1(\gd,N_{\gd/X^d})\cong H^1(X,T_X)^{d-1}$ 
is surjective, by the natural decomposition of $T_{X^d}$.
The left and right vertical arrows in diagram (\ref{diag-2}) are isomorphisms,  
since $\ko(D)|_D$ is cohomologically trivial. We conclude that the bottom row is exact as well.
Unfurling the bottom row further thus gives the inclusion
$$0 \to H^2(T_Y) \to H^2(T_{X^d})=0,$$
whence $\Def(Y)$ is unobstructed.
One also easily reads off from the diagram (\ref{diag-2}) 
that each of the projections 
$H^1(T_{X^d})=H^1(T_X)^d\rightarrow H^1(T_X)$
composes with the inclusion $H^1(T_Y)\rightarrow H^1(T_{X^d})$ to an isomorphism
$d\alpha_{\ref{1}}:H^1(T_Y)\rightarrow H^1(T_X)$,
yielding the desired identification of the differential of $\alpha_1$. 
We conclude that $\alpha_1$ is an isomorphism, as $\Def(X)$ and $\Def(Y)$ are unobstructed.

(\ref{3}) 
There are obvious morphisms of functors in both
directions which are inverse to each other.

(\ref{4})  
We use Lemma \ref{isom-1} to prove that the forgetful morphism 
$\alpha_{\ref{4}}:\Def(\PP_A, f, Y)\rightarrow \Def(\PP_A)$ is an isomorphism.
Clearly hypotheses (1) and (2) of Lemma \ref{isom-1}
hold. Let us verify that 
$H^0(T_{\PP_A})=0$, so that the Kuranishi family of $\PP_A$ is universal \cite[Thm 10.5]{BPV}. 
The following
sequence is the push-forward of the relative Euler sequence of $f$ on $Y$:
\begin{equation}\label{diag-3}
0 \to \ko_Y \to A \to f_*T_f \to 0.
\end{equation}
This implies that $H^0(T_f)$ consists of traceless sections of $A$. But $A$ is simple, so
 $H^0(T_f)=0$. The vanishing of $H^0(T_{\PP_A})$ now follows from the sequence
 \begin{equation}\label{diag-4}
 0 \to T_f \to T_{\PP_A} \to f^*T_Y \to 0,
 \end{equation}
 using the fact that $T_Y$ has no nontrivial global sections.
 
 (\ref{5}) Composing the isomorphism $\alpha_{\ref{4}}^{-1}:\Def(\PP_A)\rightarrow \Def(\PP_A, f, Y)$ with 
 the forgetful morphism, we get a natural morphism
 $\alpha_{\ref{5}}: \Def(\PP_A) \to \Def(Y)$, whose differential is the following map  
 in the long exact cohomology sequence of (\ref{diag-4}):
 $$
 \to H^1(T_f) \to H^1(T_{\PP_A}) \stackrel{d\alpha_{\ref{5}}}{\to} H^1(T_Y)\to
 $$
 \cite[(1.5)]{Wa}.
Now, $H^1(T_f)=H^1(A)$ by sequence (\ref{diag-3}),  and $H^1(A)$ vanishes, 
by Lemma \ref{lemma-A-is-rigid-if-E-is}. 
Thus, $d\alpha_{\ref{5}}$ is injective and 
$h^1(T_{\PP_A})\leq h^1(T_Y)=h^1(T_X)$, where the
equality uses the isomorphism of the pair (\ref{1}). 
We claim that the composition $\alpha_{\ref{1}}\alpha_{\ref{5}}: \Def(\PP_A) \to \Def(X)$
induces a surjective morphism  $\xi:\D_{\PP_A}\rightarrow\D_X$
between the Kuranishi deformation spaces. 
Indeed, any point 
$x$ in $\D_X$ can be joined to the origin by an open connected subset $C^0$ of a 
path $C$ of generalized twistor lines that is
wholly contained within $\D_X$ by \cite[Prop. 3.10]{huybrechts-torelli}. Then, 
given our assumptions on $E=\beta_*A$, the
construction of Verbitsky \cite[Cor. 6.10]{kaledin-verbitski-book,markman-hodge} 
(see Section \ref{flatness-and-cond-1.7-on-twistor-lines}) produces a 
projective bundle satisfying Condition
(\ref{cond-open-and-closed}) over this path. Thus, the point $x$ lies in the image of $\xi$.
The conclusion is that $h^1(T_X)\leq h^1(T_{\PP_A})$, so that these two dimensions are
equal and $d\alpha_{\ref{5}}$ is an isomorphism.
The result follows, as we have already established that both functors $\Def(\PP_A)$ and $\Def(Y)$
are unobstructed.
\end{proof}

%
\subsection{Local moduli space of reflexive Azumaya algebras}

Keep the notation of the previous subsection. 
Let $A=\E nd(\beta^*F/tor)$ be the locally free Azumaya algebra over $Y$ as above. For each point 
$x\in X \hookrightarrow X^d$ in the diagonal, 
let $D_x:=\beta^{-1}(x)\cong \PP^{2n-1}$, and write $W_x$ for the restriction of 
$\beta^*F/tor$ to this fiber.  
Let us reiterate the
vanishing in Condition \ref{cond-open-and-closed}.
There exists a non-negative integer $k$, such that the traceless endomorphism bundle
$\SheafEnd_0(W)=(A_0)_x$ satisfies 
\begin{eqnarray}
\label{az-push-fwd-1}
H^0((A_0)_x \otimes \StructureSheaf{D_x}(j))&=& 0, \hspace{4ex} \mbox{for} \ j<k, 
\\
\label{az-push-fwd-2}
H^i((A_0)_x)\otimes \StructureSheaf{D_x}(j))&=& 0, \hspace{4ex}  \mbox{for} \ i>0, \; j\geq 1
\end{eqnarray}
and $(A_0)_x\otimes \StructureSheaf{D_x}(k)$ is generated by its global sections.
We note that these conditions are {\em open} in families $(\wt{Y}, \wt{A})$.

\begin{lem}\label{flat-push}
Let $\wt{A}$ be a family of locally free sheaves of Azumaya algebras as above, satisfying
the conditions of the previous paragraph. Then
$\wt{\beta}_*\wt{A}$ is a flat family of sheaves of reflexive 
Azumaya algebras on $\wt{X}\times_S\dots\times_S\wt{X}$.
\end{lem}

\begin{proof}
This is an immediate consequence of Lemma 
\ref{lemma-sufficient-conditions-for-reflexivity-over-pi}
(\ref{lemma-item-push-forward-F-is-a-family-of-reflexive-sheaves}).
Indeed, $\wt{A}\cong \wt{A}^*$, so (\ref{az-push-fwd-1}) and (\ref{az-push-fwd-2}) 
furnish the hypotheses of that result (with $k=l$).
\end{proof}

\begin{rem}\label{correct-base-change}
Suppose $\wt{\beta}: \wt{Y} \to \wt{X}\times_S \dots\times_S\wt{X}$ 
is a deformation of $\beta: Y \to X^d$
over $S$. Denote by the subscript $T$ 
base-changes under any morphism $i:T \to S$. 
At several points below, we shall use without comment  the isomorphism 
$i^*\wt{\beta}_*\wt{A }\cong (\wt{\beta}_T)_*(\wt{A }_T)$ established in 
equation (\ref{eq-pullback-of-F-is-push-forward-of-pullback-of-V})
of the proof of Lemma \ref{lemma-sufficient-conditions-for-reflexivity-over-pi}.
\end{rem}
 
\begin{lem}\label{locally-trivial}
Let $E$ be a reflexive sheaf of Azumaya algebras on $X^d$
satisfying Condition  \ref{cond-open-and-closed}, and
$\wt{A}$ a sideways deformation of $A=\E nd (\beta^*F/tor)$, 
over $\wt{Y}\to S$. Then
$\wt{E}:=\wt{\beta}_*\wt{A}$ is a family of reflexive sheaves of Azumaya algebras over $S$ which 
satisfies Condition  \ref{cond-open-and-closed} fiberwise
over an open neighborhood of $0$ in $S$.
\end{lem}

\begin{proof}
As noted above, Conditions
(\ref{az-push-fwd-1}) and (\ref{az-push-fwd-2}) are open, so may be assumed to hold fiberwise over $S$.
The previous lemma shows that $\wt{\beta}_*\wt{A}$ is a flat family of reflexive sheaves.
Items (\ref{cond-item-characteristic-class}) and  (\ref{cond-item-Brauer-class-is-theta-X-eta}) 
are open in families, being topological, while
(\ref{cond-item-infinitesimally-rigid}) is open by semi-continuity. Item
(\ref{condition-stability}) is well-known to be open. 
Thus, these conditions hold over $S$.
The family $\wt{E}$
satisfies item (\ref{cond-item-same-homogeneous-bundle}) of
Condition \ref{cond-open-and-closed} over an open neighborhood of $0$ in $S$, by
Lemma \ref{lemma-condition-involving-W-is-open}.
\end{proof}

 Let $\gL_0$ be an Artin local ring, 
$\fm_0\subset \gL_0$ its maximal ideal,
and $\CC$ its residue field. Let $(\gL,\fm)\twoheadrightarrow (\gL_0,\fm_0)$ be an
extension with kernel $I$ satisfying $I\fm=0$, so that $I$ is a $\CC$-vector space.
Suppose $p: Z \to \mbox{Spec}\,\gL$ 
is a flat and proper morphism of analytic spaces or schemes, 
and  $G_0$ a coherent twisted sheaf on 
$Z_0:=Z\otimes_\gL \gL_0$ that is $\gL_0$-flat. Denote by $Z_{00}$ the reduction
$Z\otimes_\gL \CC$; in general, let the subscript ``$00$'' denote restriction
of an object to the closed fiber of $p$. 

\begin{prop}\label{defo-thry}
There is an obstruction class $o(G_0)$ in 
$I\otimes_\CC \Ext_{Z_{00}}^2(G_{00}, G_{00})$ 
such that $o(G_0)=0$ if and only if $G_0$ has a deformation $F$ as a 
twisted sheaf over $Z$.
If the obstruction class vanishes, the set of isomorphism classes of deformations is
an affine space over $I\otimes_\CC \nolinebreak\Ext_{Z_{00}}^1(G_{00}, G_{00})$. 
If a deformation $G$ exists, then the automorphism group 
$Aut(G/G_0)$ is naturally isomorphic to $I\otimes_\CC \Hom_{Z_{00}}(G_{00}, G_{00})$.
\end{prop} 
\begin{proof}
This may be found in Section 2.2.5 of Lieblich's thesis \cite{thesis}. 
The proof is very general and works in our analytic setting too.
\end{proof}

The following theorem of Artin will be used in the next result.

\begin{thm}\cite[Theorem 1.5]{Art}\label{convergence}
Let
\begin{equation}
\xymatrix{
U \ar[dr] \ar@{-->}[rr] & & V \ar[dl]
\\
 & W &
}
\end{equation}
be a diagram of analytic spaces, where the solid arrows are given maps.  Consider the
problem of finding a map $\phi: U \to V$ locally at a point $x\in U$ which makes the 
diagram commute. If a formal lift $\overline{\phi}$ exists, then for any integer $c>0$,
there is a map $\phi$ locally, which agrees with $\overline{\phi}$ (modulo $\hat{\fm}^c_{U,x})$.
\end{thm}

Returning to our situation, suppose $E=\E nd(F)$ is a reflexive sheaf of 
Azumaya algebras on $X^d$ satisfying Condition \ref{cond-open-and-closed}. 
Recall that part (\ref{cond-item-infinitesimally-rigid}) of that condition assumes that 
$\Ext^1(F, F)=0$.

\begin{lem}\label{uniqueness-of-defo}
Let $\pi: \wt{X}\rightarrow S$ be a deformation of $X$ and let $\wt{E}$, $\wt{E}'$  
be families of reflexive sheaves of Azumaya algebras over 
$\wt{X}^d_\pi \rightarrow S$,
both restricting to the special fiber as $E$.
Then 
there is an isomorphism $\widetilde{E}\cong \widetilde{E}'$ over an open neighborhood of $0$ in $S$.
\end{lem}

\begin{proof}
There exist families of twisted sheaves $\wt{F}$, $\wt{F}'$ such that
$\wt{E}\cong \E nd(\wt{F})$, $\wt{E}'\cong \E nd(\wt{F}')$ 
by Proposition \ref{prop-azumaya-algebra-is-End-F}. We shall prove the equivalent statement
that  $\wt{F}$ and $\wt{F}'$ agree over an open neighborhood of $0$ in $S$, up to a line bundle.

Denote the restrictions of $\wt{F}$, $\wt{F}'$ to the central fiber by $F_0$, $F_0'$. 
There exists a unique line bundle $L_0$ over the central fiber such that 
$F_0\cong F_0'\otimes L_0$, by the first step of the proof of
Proposition \ref{prop-azumaya-algebra-is-End-F}. 
Let $\wt{M}$ be the line bundle $\det(\wt{F})\otimes \det(\wt{F}')^{-1}$ (see \cite{det}).
Then, $L_0$ is an $n$-th root  of the the restriction of $\wt{M}$ to the
central fiber $M_0$, where $n=\mbox{rk}(F)$, since determinants base-change correctly.
The obstructions to deforming $M_0$ sideways to successive infinitesimal neighborhoods 
are $n$ times the obstructions to infinitesimally deforming $L_0$ sideways 
\cite[Section 7.1]{Buf}. The former clearly vanish, so the latter do too.
Therefore, $L_0$ deforms formally to $\hat{L}$ and this deformation 
is unique, by Proposition \ref{defo-thry}, as $H^1(\ko_{X^d})=0$.
 The relative Picard functor of the family 
$\wt{X}^d_\pi \to S$ is representable by an analytic space $p:\P ic \to S$ 
by \cite{Bi}. 
There is a formal section of $p$ given by $\hat{L}$; this can be extended to a section
over an open neighborhood of $0$ in $S$, by Theorem \ref{convergence}, or what is
the same thing, a line bundle $\wt{L}$ over this open neighborhood. 
We may assume that this neighborhood is the whole of  $S$. 
Note that we have used here the uniqueness
of $\hat{L}$. We claim that $\wt{L}$ is unique too. Indeed, if $\wt{L}'$ is another
line bundle over $S$ extending $\hat{L}$, then the 
completion $\H om_\pi(\wt{L},\wt{L'})^\wedge$ of the sheaf of fiberwise homomorphisms
$\H om_\pi(\wt{L},\wt{L'})$ at the origin
is non-zero, by the Formal Functions Theorem. Thus its stalk is non-zero.
This is enough to say that $\wt{L}$ and $\wt{L}'$ are isomorphic over an open neighborhood
of $0$, since these sheaves are simple, and are known to be isomorphic over the central fiber.

Consider now the twisted sheaves $\wt{F}$,  and $\wt{F}'':=\wt{F}'\otimes \wt{L}$. These
are deformations of $F_0$, and as $\Ext^1(F_0, F_0)=0$, they agree formally by
Proposition \ref{defo-thry}.  
This formal agreement, together with
the fact that  $\wt{F}$ and $\wt{F}'$ are simple, implies that these sheaves agree 
over a neighborhood of $0$, by the same argument used above to prove the uniqueness of 
the line bundle $\widetilde{L}$.
%
\end{proof}

\begin{defi}
Let $E$ be a reflexive sheaf of Azumaya algebras on $X$ which
satisfies Condition \ref{cond-open-and-closed}. Define
$\Def(X,E): ({\mathbf{GAn}})\to ({\mathbf{Sets}})$ to be the deformation functor which 
associates to a germ $S$ the set of isomorphism classes of deformations 
$(\wt{X},\wt{E})$ over $S$ of the pair $(X,E)$. Here $\pi:\wt{X}\to S$ is a flat deformation
of $X$, and $\wt{E}$ is a family of reflexive Azumaya algebras over 
$\wt{X}^d_\pi \to S$.
\end{defi}

Let $\varrho: \Def(Y,A) \to  \Def(X,E)$ be the assignment which takes a pair 
$(\wt{Y}, \wt{A})$ over $S$ to $(\wt{X},  \SheafEnd (\wt{F}))$, where $\wt{F}$ is
a reflexive twisted sheaf such that $\wt{\beta}_*\wt{A}=\E nd(\wt{F})$. Note that $\wt{F}$ exists by
Proposition \ref{prop-azumaya-algebra-is-End-F}. 
This is a natural transformation of functors, since $\wt{A}$ base changes correctly as mentioned
in Remark \ref{correct-base-change}. Write $\phi:  \Def(X,E) \to \Def(X)$ for the 
forgetful functor.

\begin{cor}
\label{cor-varrho-is-an-isomorphism-of-deformation-functors}
The transformation $\varrho: \Def(Y,A) \to  \Def(X,E)$ is an isomorphism of deformation functors.
\end{cor}
\begin{proof}
The composite $\Def(Y,A) \to  \Def(X,E) \to \Def(X)$ is an isomorphism
by Lemma \ref{isom-2}, so $\varrho$
is injective. Lemma \ref{uniqueness-of-defo} says that
the forgetful functor $\phi$ is also injective. Thus $\varrho$ is surjective.
\end{proof}

%
\subsection{A global moduli space}
The proof of the next Theorem 
will require the following terminology. 
Let $\omega$ be a K\"{a}hler class on $X$. The twistor line $\PP^1_{\omega}$ associated to $\omega$ 
(in Section \ref{flatness-and-cond-1.7-on-twistor-lines}) is 
{\em generic}, if the Picard group $\Pic(\X_t)$ is trivial for some fiber $\X_t$ of the twistor family
$\X\rightarrow \PP^1_\omega$. A path of twistor lines in $\fM_\Lambda$ is {\em generic}, if 
it consists of generic twistor lines and each pair of consecutive twistor lines meets at a marked pair with
a trivial Picard group.

\begin{thm}
\label{thm-existence}
There exists a coarse moduli space $\tfM_\Lambda$ parametrizing triples $(X_0,\eta_0,E)$
where $X_0$ is a holomorphic symplectic manifold of $K3^{[n]}$-type, 
$\eta_0: H^2(X_0,\mathbb{Z})\to \Lambda$ is a marking, and $E$ is an Azumaya
algebra on $X_0^d$ satisfying Condition \ref{cond-open-and-closed}.  
Let $\tfM^0_\Lambda$ be a connected component of $\tfM_\Lambda$
and $\fM^0_\Lambda$ the connected component of $\fM_\Lambda$
containing the image of $\tfM^0_\Lambda$ via 
the forgetful morphism 
$$
\phi: \tfM_\Lambda \to \fM_\Lambda. 
$$
Then the morphism $\phi: \tfM^0_\Lambda \to \fM^0_\Lambda$
is surjective and a local homeomorphism.
\end{thm}

\begin{proof} 
Putting together Corollary \ref{cor-varrho-is-an-isomorphism-of-deformation-functors}
and Lemma \ref{isom-2}, we have that
$\Def(X,E)$ and $\Def(X)$ are naturally isomorphic. Thus the 
deformation problem $\Def(X,E)$ admits a {\em universal} local moduli space. Adding the marking
and glueing the local charts $\D_{(X,\eta,E)}$ by the
same procedure as used in the construction of $\fM_\Lambda$
(see \cite[1.18]{huybrects-basic-results}), we obtain $\tfM_\Lambda$. 
As a topological space, $\tfM_\Lambda$ is the quotient $Z/\sim$ of the 
disjoint union $Z$ of the local Kuranishi spaces $\D_{(X,\eta, E)}$ by the
equivalence  relation $\sim$ corresponding to isomorphism of triples. 
The map $\D_{(X,\eta,E)}\rightarrow \tfM_\Lambda$ is injective, since
each composition
\begin{equation}
\label{eq-local-chart}
\D_{(X,\eta,E)}\rightarrow \D_{(X,\eta)}\RightArrowOf{P}\Omega
\end{equation}
is injective, and these compositions
glue to a well defined map $\widetilde{P}:\tfM_\Lambda\rightarrow \Omega$. 
The equivalence relation $\sim$ is open,\footnote{The statement that the relation is open 
translates, by definition,
to the following statement.
If $t'_1:=(X'_1,\eta'_1,\E'_1)\in \D_{(X_1,\eta_1,E_1)}$ is isomorphic to $t'_2:=(X'_2,\eta'_2,\E'_2)\in \D_{(X_2,\eta_2,E_2)}$,
then for every open neighborhood $V_1$ of $t'_1$ in $Z$, there exists an open neighborhood
$V_2$ of $t'_2$ in $Z$, such that for any $t''_2:=(X''_2,\eta''_2,\E''_2)$ in $V_2$ there exists 
a point $t''_1=(X''_1,\eta''_1,\E''_1)$ in $V_1$, such that $(X''_1,\eta''_1,\E''_1)$ is isomorphic to $(X''_2,\eta''_2,\E''_2)$.
This is known for marked pairs, and extends to marked triples by Lemma \ref{uniqueness-of-defo}.} 
so the quotient map $Z\rightarrow \tfM_\Lambda$ is an open map. 
We conclude that each Kuranishi deformation space $\D_{(X,\eta,E)}$ is embedded
as an open subset of the quotient space $\tfM_\Lambda$. 
The gluing transformations are holomorphic, as each of the compositions
(\ref{eq-local-chart}) is an open holomorphic embedding.

The statement that $\phi$ is a local homeomorphism is clear from the construction.

\smallskip
The morphism $\phi:\tfM^0_\Lambda\rightarrow \fM^0_\Lambda$ is surjective.
The proof is identical to that \cite[Theorem 7.11]{markman-hodge}. 
There it is shown that given any marked pair $(X,\eta)$ in $\fM^0_\Lambda$,
there exists a generic twistor path $C$ from $(X_0,\eta_0)$ to $(X,\eta)$, and a flat deformation $\E$ of
$E_0$ along this path, via slope-stable hyperholomorphic reflexive sheaves 
$\E_t$, $t\in C$, of Azumaya algebras.
The properties in Condition \ref{cond-open-and-closed} hold for all $\E_t$, $t\in C$, 
by Proposition \ref{prop-condition-remains-valid-under-twistor-deformations}.
The family $\E$ provides a lift of the path $C$ to a path in $\tfM^0_\Lambda$. Hence, $\phi$ is surjective.
The proof of Theorem 7.11 uses
Theorem 7.10 in \cite{markman-hodge}, which establishes 
the existence of a triple $(X_0,\eta_0,E_0)$ determining the component $\tfM^0_\Lambda$, 
such that the order of the Brauer class $\theta_{(X_0,\eta_0)}$ 
is equal to the rank of $E_0$. 
Here  the existence of such a triple, in any non-empty component $\tfM^0_\Lambda$, is established as follows
(eliminating the need for an analogue of \cite[Theorem 7.10]{markman-hodge}). 
The order of $\theta_{(X,\eta)}$ 
is $r$, whenever $\Pic(X)$ is trivial, by our assumption that the element $\tilde{\theta}$ has order $r$
in Condition \ref{cond-open-and-closed} (\ref{cond-item-Brauer-class-is-theta-X-eta}).
Points in $\fM^0_\Lambda$ with a trivial Picard group form a dense subset.
Hence, the existence of such a triple $(X_0,\eta_0,E_0)$ follows from the fact that the morphism
$\phi:\tfM^0_\Lambda\rightarrow \fM^0_\Lambda$ is a local homeomorphism.
\end{proof}

%
\section{Stability and separability}

Recall that a sheaf of Azumaya algebras may be represented as the endomorphism
sheaf of some twisted sheaf, which is unique, up to tensorization by a line bundle,
by Proposition \ref{prop-azumaya-algebra-is-End-F}.
The following is an analogue of \cite[Prop. 6.6]{kosarew-okonek}.

\begin{prop}
\label{prop-non-separated}
Let 
$E_1$, $E_2$ be two reflexive sheaves of Azumaya algebras on $X^d$
satisfying Condition \ref{cond-open-and-closed} (and so with the same Brauer class (\ref{eq-theta-X-eta})). 
Assume that 
the points associated  to $(X,\eta,E_1)$ and $(X,\eta,E_2)$ are non-separated 
in $\tfM_\Lambda$. Then there exist twisted reflexive sheaves 
$F_i$, such that $E_i\cong\SheafEnd(F_i)$, $i=1,2$, and 
$\det(F_1^*\otimes F_2)$ is the trivial line-bundle, as well as 
non-trivial homomorphisms
$\varphi:F_1\rightarrow F_2$ and $\psi:F_2\rightarrow F_1$, satisfying
$\varphi\psi=0$ and $\psi\varphi=0$.
\end{prop}

\begin{proof}
Our assumption implies that 
there exists a family $\pi:\X\rightarrow S$, over a one dimensional disk $S$, 
with special fiber $\X_0$ isomorphic to $X$,
families $\E_1$ and $\E_2$ of reflexive sheaves  of Azumaya algebras over $\X^d_\pi\rightarrow S$, 
and isomorphisms of $E_i$ with the restriction of $\E_i$ to $\X_0^d$,
as well as an isomorphism $f:(\E_1\restricted{)}{U}\rightarrow (\E_2\restricted{)}{U}$
between the restrictions of the families to the complement $U\subset \X_\pi^d$ of the  fiber $\X_0^d$.

Let $\F_i$ be a family of $\theta$-twisted reflexive sheaves over $\X_\pi^d\rightarrow S$, 
such that $\SheafEnd(\F_i)$ is isomorphic to $\E_i$ as sheaves of Azumaya algebras.
Such sheaves $\F_i$ exist, by Proposition \ref{prop-azumaya-algebra-is-End-F}.
We may assume that $(\F_1\restricted{)}{U}$ is isomorphic to 
$(\F_2\restricted{)}{U}$, possibly after tensoring $\F_2$ by a line bundle $L$ over $\X_\pi^d$. 
Indeed, the existence of the isomorphism $f$ implies that there exists a line-bundle $\bar{L}$
over $U$, such that $(\F_1\restricted{)}{U}$ is isomorphic to 
$(\F_2\restricted{)}{U}\otimes \bar{L}$, by the
uniqueness statement in Proposition \ref{prop-azumaya-algebra-is-End-F}. 
The existence of such a line bundle $L$  over $\X_\pi^d$ follows from 
and the surjectivity of the 
restriction homomorphism from $\Pic(\X_\pi^d)$ to $\Pic(U)$. 
Let $F_i$ be the restriction of $\F_i$ to the special fiber $X_0^d$.
The existence of the non-trivial homomorphisms $\varphi$ and $\psi$ follows, by the semi-continuity theorem. 
The sheaves $F_i$ are simple, by the stability 
Condition \ref{cond-open-and-closed} (\ref{condition-stability}). Hence, $\varphi\psi$ and
$\psi\varphi$ either vanish or they are isomorphisms. The vanishing follows, since we 
assumed that the Azumaya algebras $E_1$ and $E_2$ are not isomorphic.
The triviality of $\det(F_1^*\otimes F_2)$ follows, by continuity.
\end{proof}

\begin{cor}
\label{cor-triple-is-a-separated-point}
Let $E$ be a  sheaf of Azumaya algebras over $X^d$, satisfying Condition
\ref{cond-open-and-closed},
with a Brauer class $\theta\in H^2(X^d,\StructureSheaf{X^d}^*)$.
If the rank of $E$ is equal to the order of $\theta$, then $(X,\eta,E)$
is a separated point of $\tfM_\Lambda$, provided $(X,\eta)$ is a separated point of $\fM_\Lambda$.
\end{cor}

\begin{proof}
Assume that $(X,\eta,E)$ and $(X',\eta',E')$ are non-separated points of $\tfM_\Lambda$.
Then $(X,\eta)$ and  $(X',\eta')$ are non-separated points of $\fM_\Lambda$ and are hence equal.
Let $F$ be a $\theta$-twisted sheaf representing $E$.
Then $F$ does not admit any non-trivial subsheaf of lower rank, since the rank of any subsheaf divides 
the order of $\theta$. The Corollary now follows from Proposition \ref{prop-non-separated}.
\end{proof}

\begin{lem}
\label{lemma-if-stable-with-respect-to-same-Kahler-class-then-separated}
Let $(X,\eta)$ be a  point of $\fM_\Lambda$, not necessarily separated, 
and $E_1$, $E_2$ two  reflexive sheaves of Azumaya algebras 
over $X^d$ satisfying Condition \ref{cond-open-and-closed}. 
Set $t_i=(X,\eta,E_i)$, $i=1,2$. 
If $E_1$ and $E_2$ are both $\omega$-slope-stable, with respect to the same
K\"{a}hler class $\omega$ on $X$,
then either $t_1=t_2$, or  the points $t_1$ and $t_2$ are separated in $\tfM_\Lambda$.
\end{lem}

\begin{proof}
The proof is by contradiction.
Assume that $t_1\neq t_2$ and that the points $t_1$ and $t_2$ are non-separated. Let $F_i$ be twisted sheaves
satisfying $\SheafEnd(F_i)\cong E_i$, $i=1,2$, such that the line bundle $\det(F_1^*\otimes F_2)$
is trivial, admitting non-zero homomorphisms
$\varphi:F_1\rightarrow F_2$ and
$\psi:F_2\rightarrow F_1$ satisfying $\varphi\psi=0$ and $\psi\varphi=0$, as in 
Proposition \ref{prop-non-separated}.

\hide{
Let $h_{12}:\SheafEnd(F_1)\rightarrow \SheafEnd(F_2)$ be the sheaf homomorphism given by 
$h_{12}(g)=\varphi g \psi$.
Define $h_{21}:\SheafEnd(F_2)\rightarrow \SheafEnd(F_1)$, by $h_{21}(g)=\psi g \varphi$. 
Denote by $\mu$ the $\omega$-slope function.
Both $\SheafEnd(F_1)$ and $\SheafEnd(F_2)$ are polystable with trivial determinant,
by Definition \ref{def-slope-stability}.
Hence, $\mu[Im(h_{12})]=0$.
We have the injective homomorphism
$\gamma:
Im(h_{21})\rightarrow \SheafHom\left(Im(\psi),Im(\varphi)\right),
$
sending $h_{12}(g):=\varphi g \psi$ to the induced homomorphism
from $F_2/\ker(\psi)\cong Im(\psi)$ to $Im(\varphi)$.
The kernel of $h_{12}$ consists of endomorphisms of
$F_1$, which map $Im(\psi)$ into $\ker(\varphi)$. 
One easily checks that the rank of the image of $h_{12}$ is equal
to $\rank(Im(\varphi))\rank(Im(\psi))$. Hence, 
$0=\mu(Im(h_{12}))\leq \mu\left[\SheafHom\left(Im(\psi),Im(\varphi)\right)\right]$.
Interchanging the roles of $F_1$ and $F_2$ and working with $h_{21}$ we get the inequality
$0\leq \mu\left[\SheafHom\left(Im(\varphi),Im(\psi)\right)\right].$
We conclude the equality
\[
\mu\left[\SheafHom\left(Im(\varphi),Im(\psi)\right)\right]=0.
\]
}

The slope function is additive with respect to tensor products and $\mu(F_1^*\otimes F_2)=0$. 
Hence,
\[
\mu\left[\SheafHom(Im(\varphi),F_1)\right]=
\mu\left[\SheafHom(Im(\varphi),F_1)\otimes F_1^*\otimes F_2\right]=
\mu\left[\SheafHom(Im(\varphi),F_2)\right].
\]
We conclude the equality of degrees
$
\deg\left[\SheafHom(Im(\varphi),F_1)\right]=\deg\left[\SheafHom(Im(\varphi),F_2)\right].
$
The degree function is additive with respect to short exact sequences, and so
$\deg\left[\SheafHom(Im(\varphi),F_1)\right]=-\deg\left[\SheafHom(\ker(\varphi),F_1)\right]$.
We get the equality
\[
\deg\left[\SheafHom(Im(\varphi),F_2)\right]=-\deg\left[\SheafHom(\ker(\varphi),F_1)\right].
\]
The left hand side of the above equality is strictly positive, 
since $F_2$ is stable, 
while the right hand side is strictly negative, since
$F_1$ is $\omega$-slope-stable. A contradiction.
\end{proof}

Once Theorem \ref{thm-introduction-Global-Torelli-for-triples} (\ref{thm-item-inseparable})
is established, we  conclude the equality $t_1=t_2$ in the statement of 
Lemma \ref{lemma-if-stable-with-respect-to-same-Kahler-class-then-separated}.

%
\section{A global Torelli theorem}

Let $\tfM^0_\Lambda$ be a connected component of $\tfM_\Lambda$. Then $\tfM^0_\Lambda$ contains 
a point associated to an isomorphism
class of a triple $(X_0,\eta_0,E_0)$, where the order of the
Brauer class $\theta_{X_0,\eta_0}\in H^2(X_0^d,\StructureSheaf{X_0^d}^*)$ is equal to the rank $r$ of $E_0$,
as shown in the proof of Theorem \ref{thm-existence}. 
Let $\fM^0_\Lambda$ be the connected component, of the moduli space $\fM_\Lambda$, containing the image
of $\tfM^0_\Lambda$ via the forgetful morphism. Denote the forgetful morphism by
\[
\phi \ : \ \tfM^0_\Lambda \ \ \ \rightarrow \ \ \ \fM^0_\Lambda.
\]
Let $\Omega\subset \PP(\Lambda\otimes_\Integers\ComplexNumbers)$ be the period domain
(\ref{eq-period-domain}) and $P:\fM^0_\Lambda\rightarrow \Omega$ the period map.
Let $\widetilde{P}:\tfM^0_\Lambda\rightarrow \Omega$
be the composition $P\circ \phi$.

Consider the relation $(X_1,\eta_1,E_1)\sim (X_2,\eta_2,E_2)$ between 
any pair of inseparable points in $\tfM^0_\Lambda$. 
The relation $\sim$ is clearly symmetric and reflexive. Following is the main result of this paper.

\begin{thm}
\label{main-thm}
Assume that Conjecture \ref{conj-Ext-1-is-constant} holds. 
\begin{enumerate}
\item
\label{thm-item-Hausdorff-reduction}
The relation $\sim$ is an equivalence relation. 
The quotient space $\hfM^0_\Lambda := \tfM^0_\Lambda/\sim$ admits a natural structure of a
Hausdorff complex manifold, and the quotient map $\tfM^0_\Lambda\rightarrow \hfM^0_\Lambda$ is open. 
The period map $\widetilde{P}$ factors through a morphism
$\overline{P}:\hfM^0_\Lambda\rightarrow \Omega$. 
\item
\label{thm-item-isomorphism}
The morphism $\overline{P}:\hfM^0_\Lambda\rightarrow \Omega$ is an isomorphism of complex manifolds.
\item 
\label{thm-item-generic-point-is-separated}
Let $x$ be a point of $\Omega$, such that 
$\Lambda^{1,1}(x)=0$,  or  $\Lambda^{1,1}(x)$ is cyclic generated by a class of non-negative self intersection.
Then the fiber $\widetilde{P}^{-1}(x)$ consists of a single separable point of $\tfM_\Lambda^0$.
\end{enumerate}
\end{thm}

\noindent
{\bf Proof of part \ref{thm-item-Hausdorff-reduction} of Theorem 
\ref{main-thm}.}
The morphism $\phi$ is a local homeomorphism, by Theorem \ref{thm-existence}. 
Hence, so is $\widetilde{P}:\tfM^0_\Lambda\rightarrow \Omega$.
We show next that  a point
$(X,\eta,E)$ of $\tfM^0_\Lambda$ is a separated point, if $\Pic(X)$ is trivial. 
The point $(X,\eta)$ of 
$\fM^0_\Lambda$ is known to be separated \cite[Prop. 4.7]{verbitsky-torelli,huybrechts-torelli}. 
Condition \ref{cond-open-and-closed} (\ref{cond-item-Brauer-class-is-theta-X-eta}) 
assures that the Brauer class of the Azumaya algebra $E$ has order $r$
(Remark \ref{rem-trivial-picard-implies-Brauer-class-has-order-r}).
The triple $(X,\eta,E)$ thus corresponds to a separated point,
by Corollary \ref{cor-triple-is-a-separated-point} above.

We follow the procedure of {\em Hausdorff reduction} used by Verbitsky in his proof of the
Global Torelli Theorem \cite[Prop. 4.9 and Cor. 4.10]{verbitsky-torelli,huybrechts-torelli}.
There it is proven that the analogous relation $\sim$ for $\fM^0_\Lambda$ is an equivalence relation,
$\fM^0_\Lambda/\sim$ is a complex Hausdorff manifold, the quotient morphism 
$\fM^0_\Lambda\rightarrow \fM^0_\Lambda/\sim$ is open, and the period map $P$ factors through 
$\fM^0_\Lambda/\sim$.
The only two facts used, are that $P:\fM^0_\Lambda\rightarrow\Omega$ is a local homeomorphism,
and that points of $\fM^0_\Lambda$, such that $\Pic(X)$ is trivial, are separated. 
These two facts hold for $\tfM^0_\Lambda$ as well, as shown above. Hence, the proofs of 
\cite[Prop. 4.9 and Cor. 4.10]{huybrechts-torelli} apply verbatim to prove part \ref{thm-item-Hausdorff-reduction}.
\EndProof

\medskip
\noindent
{\bf Proof of part \ref{thm-item-isomorphism} of Theorem 
\ref{main-thm}.}

\begin{prop}
\label{prop-lifting-of-twistor-lines}
Consider an equivalence class $[X,\eta,E]\in  \overline{\fM}^0_\Lambda$ of a triple $(X,\eta,E)$
and assume that its period $\overline{P}([X,\eta,E])$ is contained in a generic twistor line $Tw\subset\Omega$.
Then there exists a lift  $\iota:Tw\rightarrow \overline{\fM}^0_\Lambda$, such that
$\overline{P}\circ\iota:Tw\rightarrow \Omega$ is the inclusion, and $\iota(Tw)$ contains 
$[X,\eta,E]$.
\end{prop}

\begin{proof}
The analogous statement for the Hausdorff reduction of $\fM^0_\Lambda$ is
\cite[Prop. 5.4]{huybrechts-torelli}. There a simple point set topology argument reduces 
the proof to the case where $\Pic(X)$ is trivial. In that case 
the existence of a unique lift  $\iota:Tw\rightarrow \fM^0_\Lambda$ through $(X,\eta)$ follows from 
the fact that the K\"{a}hler cone of $X$ is equal to its positive cone, and from the
construction of a twistor family through $(X,\eta)$ for any K\"{a}hler class $\omega$ of $X$.
Thus, it remains to prove that when $\Pic(X)=(0)$, any twistor line through 
$(X,\eta)$ lifts further to a twistor line through $(X,\eta,E)$ in $\tfM^0_\Lambda$. 
The triviality of $\Pic(X)$ implies that the order of $\theta_{(X,E)}$ is $r$, hence $E$ is 
$\omega$-slope-stable with respect to every K\"{a}hler class on $X$
\cite[Prop. 7.8]{markman-hodge}. The lifting of each twistor line to a deformation of the triple 
$(X,\eta,E)$ over the twistor family is proven, using results of Verbitsky, 
in \cite[Cor. 6.10 and Prop. 7.8]{markman-hodge}. The lifted deformation lies in
$\widetilde{\fM}^0_\Lambda$, by Proposition \ref{prop-condition-remains-valid-under-twistor-deformations}.
\end{proof}

The proof of part \ref{thm-item-isomorphism} of Theorem 
\ref{main-thm} is now identical to the argument in \cite[Sec.5.4]{huybrechts-torelli}
establishing Corollary \cite[5.9]{huybrechts-torelli}. Simply replace \cite[Prop. 5.4]{huybrechts-torelli}
by Proposition \ref{prop-lifting-of-twistor-lines} above.
In the language of Verbitsky, Proposition \ref{prop-lifting-of-twistor-lines} establishes the fact that
the map $\overline{P}:\overline{\fM}^0_\Lambda\rightarrow \Omega$ is 
{\em compatible with generic hyperk\"{a}hler lines} \cite[Definition 6.2]{verbitsky-torelli}.
Verbitsky proves that any map $\psi:M\rightarrow \Omega$ from a Hausdorff manifold $M$,
which is compatible with generic hyperk\"{a}hler lines, is a covering
\cite[Theorem 6.14]{verbitsky-torelli}. Theorem 
\ref{main-thm} (\ref{thm-item-isomorphism})  follows, since $\Omega$ is simply connected.
\EndProof

\medskip
\noindent
{\bf Proof of part
\ref{thm-item-generic-point-is-separated} of Theorem \ref{main-thm}.}
The case $\Lambda^{1,1}(x)=0$ follows from 
Corollary \ref{cor-triple-is-a-separated-point} and Remark \ref{rem-trivial-picard-implies-Brauer-class-has-order-r}.
Assume that $\Lambda^{1,1}(x)$ is spanned by the non-zero class $\bar{c}$, and $(\bar{c},\bar{c})\geq 0$. 
The fiber $P^{-1}(x)$ intersects the connected component $\fM_\Lambda^0$ in a single
separable point $(X,\eta)$, by Verbitsky's Global Torelli Theorem 
\cite[Theorem 2.2]{verbitsky-torelli,huybrechts-torelli,markman-survey}.
Let $t_1:=(X,\eta,E_1)$ and $t_2:=(X,\eta,E_2)$ be points in the intersection of the fiber $\widetilde{P}^{-1}(x)$
with the connected component $\tfM^0_\Lambda$. Then either $t_1=t_2$, 
or $t_1$ and $t_2$ are inseparable,
by part  (\ref{thm-item-isomorphism}) of the Theorem. 
Let $\omega_i$ be a K\"{a}hler class, such that $E_i$ is $\omega_i$-slope-stable, $i=1,2$.
It suffices to show that $E_2$ is also $\omega_1$-slope-stable, by Lemma 
\ref{lemma-if-stable-with-respect-to-same-Kahler-class-then-separated}.

Let $c:=\eta^{-1}(\bar{c})\in H^{1,1}(X,\Integers)$ be a generator. 
We may assume that $c$ belongs to the closure of the positive cone, possibly after replacing 
$\bar{c}$ by $-\bar{c}$.
Let $B\subset E_2$ be a maximal parabolic sheaf of Lie subalgebras.
Then $\int_X\omega_2^{2n-1}c_1(B)<0,$ since $E_2$ is $\omega_2$-slope-stable.
For every K\"{a}hler class $\omega$ and for every class 
$\alpha$ in the closure of the positive cone of $X$ in $H^{1,1}(X,\RealNumbers)$, 
the following inequality holds:
\[
\int_X\omega^{2n-1}\alpha>0,
\]
by \cite[(1.10)]{huybrects-basic-results}.
Hence, $c_1(B)=kc,$ for some negative integer $k$.
We conclude that the $\omega_1$-degree   $\int_X\omega_1^{2n-1}c_1(B)=k\int\omega_1^{2n-1}c$ of  $B$
is negative as well. Hence, $E_2$ is $\omega_1$-slope-stable. \nolinebreak
\EndProof

%
\section{Proof of Theorem \ref{thm-main-example} establishing the main example}
\label{sec-Example}

We will need the following Lemmas.
Let $V$ be a  complex vector space of even dimension $2n$, $n\geq 2$, 
and $\sigma$ a non-degenerate anti-symmetric
bilinear pairing on $V$. 
Denote by $\V$ the trivial bundle over $\PP{V}$ with fiber $V$ and let 
$\ell$ be the tautological line sub-bundle of $\V$. Denote by $\ell^\perp$ the 
sub-bundle of $\V$, which is $\sigma$-orthogonal to $\ell$. Note that 
$\sigma$ induces a non-degenerate anti-symmetric bilinear pairing on 
the quotient bundle 
\begin{equation}
\label{eq-W-is-a-subquotient}
W:=\ell^\perp/\ell.
\end{equation}

\begin{lem}
\label{lemma-cohomology-of-ell-perp-times-ell-to-the-j}
\begin{enumerate}
\item
\label{lemma-item-vanishing-of-cohomologies-of-ell-perp-times-ell-to-the-j}
$H^i(\ell^{j+1}\otimes\ell^\perp)=0$, if $i>0$, $j\leq 2$, and $(i,j)\neq (1,0)$.
\item
$H^1(\ell\otimes \ell^\perp)$ is one-dimensional.
\item
\label{lemma-item-global-sections-of-ell-perp-times-ell-to-the-j}
$H^0(\ell^{j+1}\otimes \ell^\perp)\cong 
\ker\left[\Sym^{-j-1}V^*\otimes V^*\rightarrow \Sym^{-j}V^*\right],$
for $j\leq -2$, and it vanishes for $j>-2$.
\end{enumerate}
\end{lem}

\begin{proof}
Consider the short exact sequence
$0\rightarrow \ell^{j+1}\otimes \ell^\perp\rightarrow \ell^{j+1}\otimes \V^*\RightArrowOf{\alpha} \ell^j\rightarrow 0.$,
where we embed $\ell^\perp$ as a sub-bundle of $\V^*$ via $\sigma$.
The middle and right terms do not have higher cohomology for $j\not\in \{2n-1,2n\}$, and in particular for $j\leq 2$.
The homomorphism
$\alpha$ induces a surjective homomorphism on global sections, except when $j=0$.
\end{proof}

\begin{lem}
\label{lemma-cohomology-of-ell-perp-squared}
\begin{enumerate}
\item
\label{lemma-item-higher-cohomology-of-ell-perp-times-ell-perp-times-ell-to-the-j}
$H^i(\ell^\perp\otimes\ell^\perp\otimes\ell^j)=0,$ for $j\leq 0$, for all $i>0$, except when $(i,j)=(1,0)$.
$H^1(\ell^\perp\otimes\ell^\perp)\cong\Wedge{2}V^*$. 
\item
\label{lemma-item-global-sections-of--ell-perp-times-ell-perp-times-ell-to-the-j}
$H^0(\ell^\perp\otimes\ell^\perp\otimes\ell^j)=0,$ for $j\geq -1$.
\end{enumerate}
\end{lem}

\begin{proof}
(\ref{lemma-item-higher-cohomology-of-ell-perp-times-ell-perp-times-ell-to-the-j}) 
Consider the short exact sequence
\[
0\rightarrow \ell^\perp\otimes\ell^\perp\otimes\ell^j\rightarrow
\V^*\otimes \ell^\perp\otimes\ell^j\RightArrowOf{a} \ell^\perp\otimes \ell^{j-1}\rightarrow 0.
\]
The higher cohomologies of the middle and right terms vanish for $j\leq 0$, 
by Lemma \ref{lemma-cohomology-of-ell-perp-times-ell-to-the-j} 
(\ref{lemma-item-vanishing-of-cohomologies-of-ell-perp-times-ell-to-the-j}).
When $j=0$, $H^0(\V\otimes \ell^\perp\otimes\ell^j)=0$ and 
$H^0(\ell^\perp\otimes\ell^{j-1})\cong \Wedge{2}V^*$, 
by Lemma \ref{lemma-cohomology-of-ell-perp-times-ell-to-the-j} 
(\ref{lemma-item-global-sections-of-ell-perp-times-ell-to-the-j}).
When $j\leq 0$, $H^0(\ell^\perp\otimes \ell^{j-1})$
is equal to the image of $\Sym^{-j}(V^*)\otimes \Wedge{2}V^*$ in $\Sym^{-j+2}V^*$,
by Lemma \ref{lemma-cohomology-of-ell-perp-times-ell-to-the-j} 
(\ref{lemma-item-global-sections-of-ell-perp-times-ell-to-the-j}). 
Similarly, $H^0(\V^*\otimes\ell^\perp\otimes\ell^j)$ is the image of 
$V^*\otimes \Sym^{-j-1}(V^*)\otimes \Wedge{2}V^*$ in $V^*\otimes \Sym^{-j+1}V^*$.
Thus, 
the homomorphism $a$ induces a surjective homomorphism on global sections, 
for $j\leq -1$. 

(\ref{lemma-item-global-sections-of--ell-perp-times-ell-perp-times-ell-to-the-j})
The vanishing is clear for $j\geq 0$. When $j=-1$, the homomorphism $a$ induces an isomorphism
on global sections, by  Lemma \ref{lemma-cohomology-of-ell-perp-times-ell-to-the-j} 
(\ref{lemma-item-global-sections-of-ell-perp-times-ell-to-the-j}).
\end{proof}

\begin{lem}
\label{lemma-cohomologies-of-end-ell-perp-mod-ell}
\begin{enumerate}
\item
\label{lemma-item-vanishing-of-global-sections-of-traceless-endomorphisms-for-j-less-than-2}
$H^0(\PP{V},\SheafEnd_0(W)\otimes \StructureSheaf{\PP{V}}(j))=0$, for $j<2$.
\item
\label{lemma-item-higher-cohomologies-for-j-equal-0}
$H^0(\PP{V},\SheafEnd(W))$ is one dimensional, 
$H^1(\PP{V},\SheafEnd(W))\cong \Wedge{2}V^*/\sigma$, 
$H^2(\PP{V},\SheafEnd(W))$ vanishes, and the pair $(\PP{V},W)$ is infinitesimally rigid.
\item
\label{lemma-item-End-W-twisted-by-2-is-globally-generated}
The sheaf 
$\SheafEnd_0(W)\otimes \StructureSheaf{\PP{V}}(2)$ is generated by its global sections.
\item
\label{lemma-item-higher-cohomologies-vanish-for-j-equal-minus-2}
$H^i(\PP{V},\SheafEnd_0(W)\otimes \StructureSheaf{\PP{V}}(j))=0$, for $i>0$ and for $j\geq 1$.
\item
\label{lemma-item-cohomologies-of-W-tensor-lb}
$H^i(W\otimes \StructureSheaf{\PP{V}}(j))$ vanishes, for $i>0$ and $j\geq 0$. 
$W\otimes \StructureSheaf{\PP{V}}(1)$ is generated by global sections.
\end{enumerate}
\end{lem}

\begin{proof}
Consider the short exact sequence
\begin{equation}
\label{eq-W}
0\rightarrow \ell\rightarrow\ell^\perp \rightarrow W\rightarrow 0.
\end{equation}
Tensoring with $W\otimes\ell^j$ we get the short exact sequence
\begin{equation}
\label{eq-short-exact-for-End-W}
0\rightarrow W\otimes\ell^{j+1}\rightarrow \ell^\perp\otimes W\otimes \ell^j\rightarrow W\otimes W\otimes \ell^j\rightarrow 0.
\end{equation}
Note that $\SheafEnd(W)$ is isomorphic to $W\otimes W$.
We need to compute the zero-th sheaf cohomologies of $W\otimes W\otimes \ell^j$
for $j\geq -2$ and the higher
sheaf cohomology for $j=0$ and $j\leq -2$.
We have the short exact sequences
\begin{equation}
\label{eq-seq-of-ell-j+1-times-W}
0\rightarrow \ell^{j+2}\rightarrow \ell^{j+1}\otimes \ell^\perp\rightarrow W\otimes \ell^{j+1}\rightarrow 0,
\end{equation}
\begin{equation}
\label{eq-short-exact-seq-with-ell-perp-times-W-times-ell-j}
0\rightarrow \ell^\perp\otimes \ell^{j+1}\rightarrow \ell^\perp\otimes\ell^\perp\otimes \ell^j\rightarrow 
\ell^\perp\otimes W\otimes \ell^j\rightarrow 0.
\end{equation}

Part (\ref{lemma-item-higher-cohomologies-vanish-for-j-equal-minus-2}): 
When $j\leq -1$, the higher cohomologies of the left and middle terms in the two sequences above vanish,
by Lemmas \ref{lemma-cohomology-of-ell-perp-times-ell-to-the-j} and 
\ref{lemma-cohomology-of-ell-perp-squared}. Hence, so do the higher cohomologies 
of the right terms. The latter are the left and middle terms of 
the sequence (\ref{eq-short-exact-for-End-W}). Hence, the higher cohomologies
of $W\otimes W\otimes \ell^{j}$ vanish, for $j\leq -1$.

Part (\ref{lemma-item-higher-cohomologies-for-j-equal-0}):
The quotient homomorphism $\ell^\perp\otimes \ell\rightarrow W\otimes \ell$ 
induces an isomorphism $H^i(\ell^\perp\otimes \ell)\cong H^i(W\otimes \ell)$, for all $i$,
as seen from the exact sequence
(\ref{eq-seq-of-ell-j+1-times-W}) and the vanishing of $H^i(\ell^2)$, for all $i$.
Hence, $H^i(W\otimes \ell)$ vanishes, for $i\neq 1$, and $H^1(W\otimes \ell)$ is one dimensional,
by Lemma \ref{lemma-cohomology-of-ell-perp-times-ell-to-the-j}.
$H^0(\ell^\perp\otimes W)$ clealy vanishes. Consider now the long exact sequence of cohomologies 
associated to the sequence 
(\ref{eq-short-exact-for-End-W}) 
when $j=0$.  In that case we have seen that $H^1(W\otimes \ell^{j+1})$ is one-dimensional, and
$H^0(\ell^\perp\otimes W\otimes\ell^j)$ vanishes, and so 
$H^0(W\otimes W)$ is one-dimensional, as it injects into the one-dimensional space $H^1(W\otimes\ell^{j+1})$.

The vector space $H^1(\ell^\perp\otimes W)$ is isomorphic to
$\Wedge{2}V^*/\ComplexNumbers\sigma$, by the computation of the first sheaf cohomologies
of the left and middle terms in the short exact sequence
(\ref{eq-short-exact-seq-with-ell-perp-times-W-times-ell-j}) 
in Lemmas \ref{lemma-cohomology-of-ell-perp-times-ell-to-the-j} and 
\ref{lemma-cohomology-of-ell-perp-squared}. $H^1(\ell^\perp\otimes W)\rightarrow H^1(W\otimes W)$
is an isomorphism, using the long exact cohomology sequence associated to the 
short exact sequence (\ref{eq-short-exact-for-End-W}), the vanishing of 
$H^2(\ell\otimes W)$, and the established fact that $H^0(W\otimes W)\rightarrow H^1(\ell\otimes W)$
is an isomorphism. We conclude that $H^1(W\otimes W)$ is isomorphic to $\Wedge{2}V^*/\ComplexNumbers\sigma$.

The diffrential $\LieAlg{sl}(V)\rightarrow H^1(\SheafEnd(W))$ of the pullback action of $\Aut(\PP{V})$
is identified as the composition
\[
\LieAlg{sl}(V)\rightarrow V^*\otimes V \LongRightArrowOf{\sigma\otimes 1}V^*\otimes V^* \rightarrow 
\Wedge{2}V^*\rightarrow \Wedge{2}V^*/\ComplexNumbers \sigma\cong H^1(\SheafEnd(W)),
\]
where the right isomorphism is the one constructed above.
The differential is thus surjective. Hence, the pair $(\PP{V},W)$ is infinitesimally rigid.

It remains to prove the vanishing of $H^2(W\otimes W)$. The homomorphism 
$H^2(\ell^\perp\otimes W)\rightarrow H^2(W\otimes W)$ is bijective, by the vanishing of 
$H^i(\ell\otimes W)$ established above for $i\geq 2$. In turn, 
$H^2(\ell^\perp\otimes W)$ is isomorphic to $H^2(\ell^\perp\otimes\ell^\perp)$, 
since $H^i(\ell^\perp\otimes\ell)$ vanishes for $i\geq 2$, by Lemma 
\ref{lemma-cohomology-of-ell-perp-times-ell-to-the-j}. Finally, $H^2(\ell^\perp\otimes\ell^\perp)$ vanishes, by Lemma
\ref{lemma-cohomology-of-ell-perp-squared}.

Part (\ref{lemma-item-vanishing-of-global-sections-of-traceless-endomorphisms-for-j-less-than-2}): 
The vanishing of $H^0(\SheafEnd_0(W))$ has been established in part 
(\ref{lemma-item-higher-cohomologies-for-j-equal-0}). Hence, the space $H^0(\SheafEnd_0(W)\otimes \ell^j)$
vanishes, for $j\geq 0$. It remains to prove the vanishing of 
$H^0(W\otimes W\otimes \ell^{-1})$.
$H^i(\ell)$ and $H^i(\ell^\perp)$ vanish for all $i$. Hence, $H^i(W)$ vanish for all $i$, by the exactness
of the sequence (\ref{eq-W}).
We get the isomorphism 
$H^0(W\otimes W\otimes \ell^{-1})\cong H^0(\ell^\perp\otimes W\otimes \ell^{-1})$, 
by the exactness of the sequence
(\ref{eq-short-exact-for-End-W}). The vector space $H^0(\ell^\perp\otimes W\otimes \ell^{-1})$
is isomorphic to $H^0(\ell^\perp\otimes\ell^\perp\otimes\ell^{-1})$, by the exactness of sequence 
(\ref{eq-short-exact-seq-with-ell-perp-times-W-times-ell-j}). 
Now $H^0(\ell^\perp\otimes\ell^\perp\otimes\ell^{-1})$ vanishes, by Lemma 
\ref{lemma-cohomology-of-ell-perp-squared} 
(\ref{lemma-item-global-sections-of--ell-perp-times-ell-perp-times-ell-to-the-j}).

Part (\ref{lemma-item-End-W-twisted-by-2-is-globally-generated}): 
It suffices to prove that $W\otimes \ell^{-1}$ is generated by its global sections. 
This was proven in \cite[Lemma 4.6]{markman-hodge}.

Part (\ref{lemma-item-cohomologies-of-W-tensor-lb}): The vanishing was proven in the proof
of part (\ref{lemma-item-higher-cohomologies-vanish-for-j-equal-minus-2}). 
The generation is proven in \cite[Lemma 4.6]{markman-hodge}.
\end{proof}

Denote by $Sp(V,\sigma)$ the subgroup of $GL(V)$ leaving $\sigma$ invariant.

\begin{lem}
\label{lemma-W-is-slope-stable}
\label{lemma-stability-of-ell-perp-over-ell}
The vector bundle $W$ is slope stable.
\end{lem}

\begin{proof}
Given a subsheaf $F$ of $W$, denote by $\widetilde{F}$ its inverse image
subsheaf in $\ell^\perp$. Clearly, the degree of $W$ is zero, and 
$\deg(\ell^\perp)=-1$. Let $F$ be a saturated non-zero subsheaf of $W$ of degree $d$ and rank $<2n-2$.
Then $\deg(\widetilde{F})=d-1$. The sheaf $\widetilde{F}$ is a saturated subsheaf of the 
trivial bundle $\V$. Hence, $d-1\leq 0$, and $d-1=0$, if and only if
$\widetilde{F}$ is a trivial bundle. Now $H^0(\widetilde{F})$ vanishes, since so does $H^0(\ell^\perp)$.
Hence, $\widetilde{F}$  is non-trivial and $d-1<0$. We conclude that $d\leq 0$ and $W$ is semi-stable.

The sub-bundle $\ell$ and the symplectic form $\sigma$ are invariant with respect to 
the $Sp(V,\sigma)$ action on $\V$, which corresponds to the diagonal action on $\PP{V}\times V$.
Hence, so is $\ell^\perp$, and 
the bundle $W$ is homogeneous with respect to an induced $Sp(V,\sigma)$-action.
Let $F\subset W$ be the maximal polystable subsheaf \cite[Lemma 1.5.5]{huybrechts-lehn-book}.
Then $F$ is an $Sp(V,\sigma)$-invariant subsheaf. Given a point $x\in \PP{V}$, the stabilizer of 
$x$ in $Sp(V,\sigma)$ acts transitively on the set of non-zero vectors in the fiber of $W$ over $x$. 
Hence, $F=W$ and $W$ is poly-stable. 
Finally, $\End(W)$ is one-dimensional, by Lemma \ref{lemma-cohomologies-of-end-ell-perp-mod-ell},
and so $W$ is stable.
\end{proof}

Let $M$ be the coarse moduli space of slope stable vector bundles over $\PP(V)$ of rank $\dim(V)-2$
and Chern character equal to $\dim(V)-2\sum_{i=0}^{\dim(V)/2}h^{2i}/(2i)!$, where $h$ is the first Chern class
of $\StructureSheaf{\PP(V)}(1)$.
Let $\Sigma\subset \PP\left(\Wedge{2}V^*\right)$ be the Zariski open subset of lines in $\Wedge{2}V^*$
spanned by non-degenerate $2$-forms. A point $\bar{\sigma}$ of $\Sigma$
determines the vector bundle $W_{\bar{\sigma}}$ given in Equation (\ref{eq-W}).
A relative analogue of the above construction combines with
Lemma \ref{lemma-W-is-slope-stable} to yield a morphism
\[
\kappa \ : \ \Sigma \ \ \rightarrow \ \ M,
\]
sending the point $\bar{\sigma}$ to the isomorphism class of the vector bundle $W_{\bar{\sigma}}$.

\begin{lem}
\label{lemma-kappa-is-an-open-embedding}
The morphism $\kappa$ is an open embedding.
\end{lem}

\begin{proof}
The obstruction space $H^2(\PP(V),\SheafEnd(W_{\bar{\sigma}}))$ vanishes,
by Lemma \ref{lemma-cohomologies-of-end-ell-perp-mod-ell} (\ref{lemma-item-higher-cohomologies-for-j-equal-0}).
Hence, the image of $\kappa$ is contained in the smooth locus of $M$.
The differential $d\kappa_{\bar{\sigma}}:T_{\bar{\sigma}}\Sigma\rightarrow H^1(\PP(V),\SheafEnd(W_{\bar{\sigma}}))$
is an isomorphism, 
by Lemma \ref{lemma-cohomologies-of-end-ell-perp-mod-ell} (\ref{lemma-item-higher-cohomologies-for-j-equal-0})
again. Hence the image of $\kappa$ is an open subset of $M$.
The Euler characteristic $\chi(W_{\bar{\sigma}}\otimes\ell)$ is $-1$, as seen in the proof of
Lemma \ref{lemma-cohomologies-of-end-ell-perp-mod-ell} (\ref{lemma-item-higher-cohomologies-for-j-equal-0}).
Consequently, there exists a universal family $\W$ over $\PP(V)\times M$, by the appendix in
\cite{mukai-hodge}.

Let 
$p_2:\PP(V)\times M\rightarrow M$ be the projection.
Denote by $\LieAlg{sl}(V)_M$ the trivial vector bundle over $M$ with fiber $\LieAlg{sl}(V)$. The action of $PGL(V)$ on
$\PP(V)$ induces an action of $PGL(V)$ on $M$, and we denote by
\[
a \ : \ \LieAlg{sl}(V)_M \rightarrow R^1p_{2,*}\SheafEnd(\W)
\]
the differential of this action.
We have seen that the homomorphism $a$ is surjective over the open subset $\kappa(\Sigma)$. The fiber of
$\ker(a)$ over $\kappa(\bar{\sigma})$ is the Lie sub-algebra $\LieAlg{sp}(V,\bar{\sigma})\subset \LieAlg{sl}(V)$
leaving the line $\bar{\sigma}\subset \Wedge{2}V^*$ invariant. The composition
$\ker(a)\rightarrow \LieAlg{sl}(V)_M \rightarrow \End\left(\Wedge{2}V^*\right)_M$
corresponds to a homomorphism
$
\varphi : \left(\Wedge{2}V^*\right)_M\rightarrow \SheafHom\left(\ker(a),\left(\Wedge{2}V^*\right)_M\right),
$
whose kernel is a line subbundle $L$ of the trivial vector bundle $\left(\Wedge{2}V^*\right)_M$. The line subbundle $L$ 
corresponds to a morphism 
$M\rightarrow \Sigma$, which is the inverse of $\kappa$.
\end{proof}

Let $\pi:\X\rightarrow B$ be a smooth and proper family of irreducible holomorphic symplectic manifolds
of relative dimension $2n$ over a complex analytic space $B$.
Let $T_\pi$ be the vertical tangent bundle.
Set $D:=\PP(T_\pi)$ and let $p:D\rightarrow \X$ be the natural morphism.
Let $\ell$ be the tautological line subbundle of $p^*T_\pi$.
Let $L\subset \Wedge{2}T^*_\pi$ be the image of the canonical homomorphism
$\pi^*\pi_*\Wedge{2}T^*_\pi\rightarrow \Wedge{2}T^*_\pi$. Then $L$ is a line subbundle of $\Wedge{2}T^*_\pi$
and each fiber of $L$ is spanned by the holomorphic symplectic form of the fiber of $\pi$.
Let $\ell^\perp$ be the symplectic orthogonal to $\ell$ with respect to $L$.
Set $\W:=\ell^\perp/\ell$.

Let $\ell_x$ be the restriction of $\ell$
to the fiber $\PP[T_x(\X_{\pi(x)})]$ of $p$ over $x\in \X$. 
Given a non-degenerate element $\sigma_x\in \Wedge{2}T^*_x\X_{\pi(x)}$, let
$\ell_x^{\perp_{\sigma_x}}$ be the subbundle, of the trivial bundle with fiber $T_x\X_{\pi(x)}$ over
$\PP[T_x(\X_{\pi(x)})]$, which is 
$\sigma_x$-orthogonal to $\ell_x$.

\begin{lem}
\label{lemma-W'-and-W-coincide}
Let $\W'$ be a rank $2n-2$ vector bundle over $D$. Given a point $x\in \X$, denote by $\W'_x$ the restriction of $\W'$ to the fiber $\PP[T_x(\X_{\pi(x)})]$ of $p$ over $x$. 
Assume that for every $x\in\X$, the bundle $\W'_x$ is isomorphic to $\ell_x^{\perp_{\sigma_x}}/\ell_x$,
for some non-degenerate element $\sigma_x\in \Wedge{2}T^*_x\X_{\pi(x)}$. 
Then $\W'$ is isomorphic to $\W\otimes p^*Q$, for some line bundle $Q$ over $\X$.
\end{lem}

\begin{proof}
It suffices to show  $\W'_x$ is isomorphic to $\W_x$, for all $x\in \X$. 
The vector bundle $\W'$ determines a line subbundle $L'$ of $\Wedge{2}T^*_\pi$, 
each fiber of which is spanned by a non-degenerate form, by Lemma 
\ref{lemma-kappa-is-an-open-embedding}. 
We thus have a natural isomorphism $T_\pi\otimes L'\rightarrow T^*_\pi$.
It follows that the tensor power $(L')^{2n}$ is isomorphic to the square 
$\omega_\pi^2$ of the relative canonical bundle over $\X$. 
The Picard group of $\X_b$ is torsion free, for all $b\in B$, and the canonical line bundle of
$\X_b$ is trivial.
Hence, $L'$ restricts to $\X_b$ as the trivial line bundle, for all $b\in B$.
The vector bundle $\Wedge{2}T^*\X_b$ admits a unique trivial line subbundle. Hence
$L_b=L'_b$, for all $b\in B$.
We conclude that the line subbundles $L$ and $L'$ are equal, and hence $\W'_x$ and $\W_x$ are isomorphic 
for all $x\in \X$. 
\end{proof}

%

\begin{cor}
\label{cor-families-of-W-are-locally-trivial}
Let $\W$ and $\W'$ be the vector bundles over $D$ given in Lemma \ref{lemma-W'-and-W-coincide}.
Then $\W'$ is locally trivial in the topology of $\X$ in the sense of Definition
\ref{def-vector-bundle-over-D-is-locally-trivial-over-X}.
\end{cor}

\begin{proof}
It suffices to prove the statement for $\W$, by Lemma \ref{lemma-W'-and-W-coincide}. 
A smooth family of holomorphic symplectic manifolds always admits local sysmplectic trivializations. Hence, 
each point $x\in \X$ admits an open neighborhood $U\subset \X$ and an open
neighborhood $\overline{U}\subset\PP^1_\omega$ of $\pi(x)$,
with an isomorphism 
\[
f:U\rightarrow (U\cap \X_{\pi(x)})\times \overline{U},
\]
such that $\pi\circ f^{-1}$ is the projection onot $\overline{U}$, and 
$f$ restricts to a symplectomorphism
\[
f_b:(U\cap \X_{\pi(x)})\rightarrow (U\cap \X_b), 
\]
for all $b\in \overline{U}$. 
The vector bundle $\W$ 
depends only on the symplectic structure and hence admits a trivializing isomorphism
$
\tilde{f}:\pi_1^*\left(\restricted{\W}{p^{-1}(U\cap\X_{\pi(x)})}\right)\rightarrow \restricted{\W}{p^{-1}(U)}
$
of the form required in Equation
(\ref{eq-trivializing-isomorphism-of-vector-bundles}) over the inverse image $p^{-1}(U)$ of $U$ in $D$.
\end{proof}

\noindent
{\bf Proof  of Theorem \ref{thm-main-example}.}
Part \ref{thm-item-E-satisfies-open-closed-condition})
Condition \ref{cond-open-and-closed} 
(\ref{cond-item-characteristic-class}) is verified in \cite[Prop. 4.2]{markman-hodge}.
The sheaf $F$, given in Equation (\ref{eq-sheaf-F}),
is a simple and infinitesimally rigid reflexive sheaf of rank $2n-2$ 
\cite[Prop. 4.5]{markman-hodge} and \cite[Lemma 4.2]{MM}.
Condition \ref{cond-open-and-closed} (\ref{cond-item-infinitesimally-rigid}) thus holds. 
The sheaf $F$ does not have any non-zero subsheaf of lower rank, 
when $\Pic(S)$ is trivial \cite{markman-stability}, and so 
Condition \ref{cond-open-and-closed} (\ref{condition-stability}) holds. 

Condition \ref{cond-open-and-closed} (\ref{cond-item-same-homogeneous-bundle}): 
We prove first that $F$ is $\beta$-tight (Definition \ref{def-tightness}). 
The extension sheaves $\SheafExt^i(F,F)$ where 
calculated in \cite[Prop. 3.15 ]{MM}, for all $i$, and the short exact sequence 
(\ref{eq-filtration-of-SheafExt-1-F-F}) for $\SheafExt^1(F,F)$ was established there. 
The sheaf $F$ is thus $\beta$-tight, by Lemma \ref{lemma-tightness-criterion}.

Let $\ell$ be the tautological line sub-bundle of the trivial vector bundle
over $\PP[T_zS^{[n]}]$ with fiber $T_zS^{[n]}$. 
Let $\ell^\perp$ be the symplectic orthogonal of $\ell$ and set $W:=\ell^\perp/\ell$.
The sheaf $\beta^*F/{\rm torsion}$ restricts to $\PP[T_zS^{[n]}]$ as $W\otimes\ell^{-1}$, by 
\cite[Prop. 4.5]{markman-hodge}.  Condition \ref{cond-open-and-closed} (\ref{cond-item-same-homogeneous-bundle})
now follows, with $m=0$, $m'=2$, and $k=2$, from 
Lemmas 
\ref{lemma-cohomologies-of-end-ell-perp-mod-ell} and \ref{lemma-W-is-slope-stable}.

It remains to verify Condition \ref{cond-open-and-closed} (\ref{cond-item-Brauer-class-is-theta-X-eta}).
Let $\widetilde{H}(S,\Integers)$ be the Mukai lattice, $v=(1,0,1-n)$ the Mukai vector of the
ideal sheaf of a length $n$ subscheme, $v^\perp$ the sublattice of $\widetilde{H}(S,\Integers)$ 
orthogonal to $v$, and 
\[
\mu:v^\perp  \rightarrow H^2(S^{[n]},\Integers)
\]
the Mukai isometry \cite[Eq. (1.6) and Theorem 8.1]{yoshioka-abelian-surface}.
Set $\Lambda:=v^\perp$ and let
$\eta_0:H^2(S^{[n]},\Integers)\rightarrow \Lambda$ be the inverse $\mu^{-1}$. 
Let $w\in v^\perp$ be the Mukai vector $(1,0,n-1)$ and set
\begin{equation}
\label{eq-monodromy-orbit-of-characteristic-class-tilde-theta}
\tilde{\theta}:= (-w,w)+(2n-2)[v^\perp\oplus v^\perp].
\end{equation}
Then $\tilde{\theta}$ is a class of order $2n-2$ in $\Lambda^{\oplus 2}/(2n-2)\Lambda^{\oplus 2}$.
Given a marked pair $(X,\eta)$, deformation equivalent to $(S^{[n]},\eta_0)$, 
we get the class $\theta_{(X,\eta)}$ in $H^2(X\times X,\StructureSheaf{X\times X}^*)$, 
given in (\ref{eq-theta-X-eta}). The sheaf $E:=\SheafEnd(F)$, the marking $\eta$, 
and the class $\tilde{\theta}$ satisfy 
Condition \ref{cond-open-and-closed} (\ref{cond-item-Brauer-class-is-theta-X-eta}), 
by \cite[Lemma 7.3]{markman-hodge}.

We claim that the sheaf $E:=\SheafEnd(F)$ satisfies Condition \ref{enhanced-open-and-closed-condition}.
Indeed, the Kodaira-Spencer class $\kappa_x$, given in Equation 
(\ref{eq-Kodaira-Spencer-homomorphism}), is surjective for all $x\in S^{[n]}$, by \cite[Lemma 4.3]{MM}.
Lemma \ref{lem-surjectivity-of-Kodaira-Spencer-implies-reduction-of-conjecture} implies the vanishing of 
$H^1(E_x)$, for all $x\in S^{[n]}$. 
Slope-stability of $E_x$, for all $x\in S^{[n]}$, for some K\"{a}hler class on $S^{[n]}$, is shown as follows.
Fix a K\"{a}hler class $\omega$, such that the twistor line $\PP^1_\omega$ is generic. 
Let $\E$ be a hyperholomorphic deformation of $E$ along $\PP^1_\omega$.
There exists a point $t\in \PP^1_\omega$, such that $\Pic(\X_t)$ is trivial, by our choice of $\omega$.
Let $\E_{t,x'}$ be the restriction of $\E_t$ to $\{x'\}\times\X_t$, $x'\in\X_t$.
The Brauer class of the sheaves $\E_{t,x'}$ 
has order $2n-2$, for all $x'\in \X_t$. Here we used the condition on the class $\tilde{\theta}$ in 
Condition \ref{enhanced-open-and-closed-condition}, which requires that it projects to a class 
of order $2n-2$ in $\Lambda/(2n-2)\Lambda$. The latter condition 
is evidently satisfied by the class 
(\ref{eq-monodromy-orbit-of-characteristic-class-tilde-theta}). 
It follows that $\E_{t,x'}$ is  $\omega_t$-slope-stable with respect to the K\"{a}hler class $\omega_t$ on $\X_t$ 
associated to the twistor line, by Remark \ref{rem-trivial-picard-implies-Brauer-class-has-order-r}.
Now slope-stability is preserved along hyperholomorphic deformations, and 
each $E_x$, $x\in X$, is obtained from some $\E_{t,x'}$, $x'\in \X_t$, by a hyperholomorphic deformation 
associated to $\E$ via the horizontal section through $x$ of the twistor family $\pi:\X\rightarrow\PP^1_\omega$.

Part \ref{thm-item-existence-of-E-over-every-X}) Follows from the surjectivity of
the morphism $\phi:\tfM^0_\Lambda\rightarrow \fM^0_\Lambda$ proven in Theorem
\ref{thm-existence}.

Condition \ref{enhanced-open-and-closed-condition} is preserved along twistor deformation 
(Remark \ref{rem-enhanced-condition}). Lemma 
\ref{lem-surjectivity-of-Kodaira-Spencer-implies-reduction-of-conjecture} 
(\ref{lemma-item-vanishing-of-H-1-for-E-x-implies-that-of-E})
shows that
Conjecture \ref{conj-Ext-1-is-constant} follows from Conjecture \ref{conj-Ext-1-is-constant-isolated-singularity-case}
for every hyperholomorphic deformation of $E$ along every generic twistor path. 
The proofs of Theorems 
\ref{thm-existence} and  \ref{main-thm} in our current case 
are thus conditional on either one of the two conjectures above. 

Part \ref{thm-item-Torelli-without-marking})
Condition (a) implies that $E_i$ is slope-stable with respect to some K\"{a}hler class, if and only if
it is slope-stable with respect to every K\"{a}hler class 
(see the proof of Theorem \ref{main-thm} part \ref{thm-item-generic-point-is-separated}).
Hence, we may assume that $E_1$ and $E_2$ are $\omega$-slope-stable with respect to the same 
K\"{a}hler class $\omega$ (which is Condition (b)). 
Let $Mon^2(X)\subset \Aut[H^2(X,\Integers)]$ be the monodromy group. 
$Mon^2(X)$ acts transitively on the set of  
isomorphism classes of marked pair $(X,\eta)$, in a fixed connected component $\fM_\Lambda^0$ 
and with a fixed $X$.  
The action of $g\in Mon^2(X)$
is given by $g(X,\eta)=(X,\eta\circ g^{-1})$. 
There exists a non-trivial character $cov:Mon^2(X)\rightarrow \Integers/2\Integers$ with the following property. 
Given an Azumaya algebra $A$, set $A^{(*^0)}:=A$ and $A^{(*^1)}:=A^*$, the dual Azumaya algebra.
If $(X,A)$ is deformation equivalent to $(S^{[n]},E)$, 
then $(X,\eta,E)$ and $(X,\eta\circ g,A^{(*^{cov(g)})})$ are deformation equivalent, so belong to the same
connected component $\tfM_\Lambda^0$. This statement was proven in
\cite[Theorem 1.2 (4) and Equaton (1.22)]{markman-monodromy-I}, for all elements $g$ of a subgroup 
$W\subset Mon^2(X)$, 
and the subgroup $W$ was later shown to be the full monodromy group $Mon^2(X)$ in
\cite[Theorem 1.2]{markman-integral-constraints}. 

The pairs
$(X,E_1)$ and $(X,E_2)$ are both assumed to be deformation equivalent to $(S^{[n]},E)$. 
This assumption implies
that for every marking $\eta_1$ of $X$, there exists an element $g\in Mon^2(X)$, such that
$(X,\eta_1,E_1)$ and $(X,\eta_1g,E_2)$ belong to the same connected component
$\tfM_\Lambda^0$. Thus, $(X,\eta_1g,E_1^{(*^{cov(g)})})$ and $(X,\eta_1g,E_2)$
belong to the intersection of the same fiber of $\widetilde{P}$ with $\tfM_\Lambda^0$.
Now $E_1^{(*^{cov(g)})}$ and $E_2$ are both $\omega$-slope-stable, by assumption.
Hence, $E_1^{(*^{cov(g)})}$ and $E_2$ are isomorphic, by 
Lemma \ref{lemma-if-stable-with-respect-to-same-Kahler-class-then-separated} and
Theorem \ref{main-thm} part \ref{thm-item-isomorphism}.
\EndProof

\hide{
%
\section{Stability}

Let $S$ be a $K3$ surface with a trivial Picard group. Denote by $L$ the line bundle over $S^{[n]}$, $n\geq 2$, 
such that $c_1(L)$ is half the class of the diagonal divisor. 
Let $L_{a,b}$ be the line bundle $L^a\boxtimes L^b$
over $S^{[n]}\times S^{[n]}$. 
The sheaf $E_S$ over $S^{[n]}\times S^{[n]}$ has determinant $L_{-1,1}$, by 
\cite[Lemma 5.9]{markman-integral-constraints}. 
Let $[\Delta]$ be the class of the diagonal in $H^{4n}(S^{[n]}\times S^{[n]},\Integers)$.
The Chern class $c_{2n-1}(E_S\otimes M)$ vanishes, 
and $c_{2n}(E_S\otimes M)=[1-(2n-1)!][\Delta]$,
for any line-bundle $M$, 
by \cite[Theorem 1 part 3]{markman-diagonal}.
When $n=2$, $E_S$ is $\omega$-slope-stable, with respect to every
K\"{a}hler class $\omega$ on $S^{[2]}\times S^{[2]}$, by the following Lemma.

\begin{lem}
\label{stability-lemma}
Let $E$ be a simple reflexive sheaf of rank $2$ and
determinant $L_{c,d}$ over $S^{[2]}\times S^{[2]}$, 
where $c$ and $d$ are odd integers. Assume that $c_3(E\otimes M)$ vanishes 
and
$c_4(E\otimes M)$ is independent of the line bundle $M$.
Then $E$ does not have any subsheaf of rank $1$.
\end{lem}

\begin{proof}
The proof is by contradiction.
Let $F$ be a rank $1$ subsheaf of $E$. We may assume that $F$ is saturated.
Then $F$  is a reflexive sheaf of rank $1$ and hence locally free.
The Picard group of $S^{[2]}\times S^{[2]}$ has rank $2$, generated by $L_{1,0}$ and $L_{0,1}$.
Let $a$ and $b$ be integers, such that $F$ is isomorphic to $L_{a,b}$. 
Then $E/F$ is isomorphic to $I_Z\otimes L_{c-a,d-b}$, for some subscheme
$Z$ of $S^{[2]}\times S^{[2]}$ of codimension $\geq 2$.
Consider the short exact sequence defining $Q$
\[
0\rightarrow E/F \rightarrow L_{c-a,d-b}\rightarrow Q \rightarrow 0.
\]
We get the long exact sequence
\[
0\rightarrow \Ext^1(Q,L_{a,b}) \rightarrow \Ext^1(L_{c-a,d-b},L_{a,b})\rightarrow \Ext^1(E/F,F)
\rightarrow \Ext^2(Q,L_{a,b}) \rightarrow \cdots
\]
We have also 
the isomorphisms $\Ext^2(Q,L_{a,b})\cong \Ext^{6}(L_{a,b},Q)^*\cong H^{6}(L_{-a,-b}\otimes Q)^*$.

If the co-dimension of $Z$ is larger than $2$, then
$H^{6}(L_{-a,-b}\otimes Q)^*$ vanishes. Hence, 
the extension class of 
\[
0\rightarrow F \rightarrow E \rightarrow E/F \rightarrow 0
\]
is the image of a class $\epsilon$ in $H^1\left(S^{[n]}\times S^{[n]},L_{2a-c,2b-d}\right)$.
The latter is the direct sum
\[
H^0(S^{[n]},L^{2a-c})\otimes H^1(S^{[n]},L^{2b-d})\ \ \oplus \ \ 
H^1(S^{[n]},L^{2a-c})\otimes H^0(S^{[n]},L^{2b-d}).
\]
The class $\epsilon$ must thus vanish, since $H^0(S^{[n]},L^k)$ vanishes, for odd $k$.
This contradicts the assumption that the sheaf $E$ is simple.

It remains to prove that the co-dimension of $Z$ is larger than $2$.
Define $z_i$ as the coefficients of the Chern character
\[
ch(I_Z) \ \ = \ \ 1-z_2t^2+z_3 t^3 -z_4 t^4 +  \cdots 
\]
Let $\delta\in H^2(S^{[2]},\Integers)$ be half the class of the diagonal divisor,
and denote by $\delta_i$ its pullback to $H^2(S^{[2]}\times S^{[2]},\Integers)$ via the $i$-th projection, $i=1,2$.
Set $\alpha:=a\delta_1+b\delta_2$ and $\beta:=(c-a)\delta_1+(d-b)\delta_2$. Then $c(L_{a,b})=1+\alpha t$.
Recall that the  coefficients of degree two, three, and four  of the Chern character are 
$ch_2=\frac{1}{2}c_1^2-c_2$,  $ch_3=\frac{1}{6}c_1^3-\frac{1}{2}c_1c_2+\frac{1}{2}c_3$, and
\[
ch_4 \ \ = \ \ \frac{1}{24}\left(
c_1^4-4c_1c_3+2c_2^2-4c_4
\right).
\]
We get the following equalities.
\begin{eqnarray*}
ch(I_Z\otimes L_{c-a,d-b})&=&1+\beta t + (\beta^2/2-z_2)t^2 +
(\beta^3/6-\beta z_2 +z_3)t^3 
\\
& & + (\frac{1}{24}\beta^4-\frac{1}{2}\beta^2z_2+\beta z_3-z_4) t^4\cdots
\\
c(I_Z\otimes L_{c-a,d-b})&=&1+\beta t + z_2 t^2 + (-\beta z_2+2z_3)t^3 
\\
& & +(\frac{1}{2}z_2^2+4\beta^2 z_2 -8\beta z_3 +6 z_4)t^4 \cdots
\\
c_3(E) & = & (\alpha-\beta)z_2+2z_3.
\\
c_4(E) & = & \alpha(2z_3-\beta z_2) + (\frac{1}{2}z_2^2+4\beta^2 z_2 -8\beta z_3 +6 z_4).
\end{eqnarray*}
The vanishing of $c_3(E)$ yields the equality 
\[
z_3 \ \ = \ \ \frac{1}{2}(\beta-\alpha)z_2.
\]
Eliminating $z_3$, we get the equality
\[
c_4(E) \ \ = \ \
6z_4+\frac{1}{2}z_2^2+z_2(4\alpha\beta-\alpha^2).
\]

Set $\widetilde{E}:=E\otimes L_{x,y}$. Set 
$\tilde{a}:=a+x$, $\tilde{b}=b+y$, $\tilde{c}:=c+2x$, $\tilde{d}=d+2y$,
$\tilde{\alpha}:=\tilde{a}\delta_1+\tilde{b}\delta_2$, 
$\tilde{\beta}:=(\tilde{c}-\tilde{a})\delta_1+(\tilde{d}-\tilde{b})\delta_2$.
We get the equation
\[
c_4(\widetilde{E}) \ \ = \ \ 6z_4+\frac{1}{2}z_2^2+z_2(4\tilde{\alpha}\tilde{\beta}-\tilde{\alpha}^2).
\]
The class $c_4(\widetilde{E})$ is constant with respect to $x$ and $y$.
Equating the coefficients of $x$, $y$, $x^2$, $xy$, and $y^2$ to zero,
we get five equations in $H^8(S^{[2]}\times S^{[2]},\Integers)$. The latter three are equivalent to the following differential equations. 
The partials satisfy $\tilde{\alpha}_x=\tilde{\beta}_x=\delta_1$ and 
$\tilde{\alpha}_y=\tilde{\beta}_y=\delta_2$. We get:
\begin{eqnarray*}
c_4(\widetilde{E})_{xx}&=&6\delta_1^2z_2 = 0.
\\
c_4(\widetilde{E})_{yy}&=&6\delta_2^2z_2=0.
\\
c_4(\widetilde{E})_{xy}&=&6\delta_1\delta_2z_2=0.
\end{eqnarray*}

The cohomology $H^*(S^{[2]},\RationalNumbers)$ is generated by $H^2(S^{[2]},\RationalNumbers)$.
For a generic $S$, the vector space $H^{2,2}(S^{[2]}\times S^{[2]},\RationalNumbers)$ is spanned by 
\[
\delta_1^2\boxtimes 1, \ \ 
1\boxtimes \delta_2^2, \ \ 
\delta_1\boxtimes\delta_2, \ \ 
c_2\boxtimes 1, \ \ 
1\boxtimes c_2.
\]
Let $h$ be the homomorphism from $H^{2,2}(S^{[2]}\times S^{[2]},\RationalNumbers)$ into 
the direct sum of three copies of 
$H^{4,4}(S^{[2]}\times S^{[2]},\RationalNumbers)$, sending a class $\gamma$ to 
$(\delta_1^2\gamma,\delta_1\delta_2\gamma,\delta_2^2\gamma)$. 
Then $h$ is an injective homomorphism. Thus the class $z_2$ vanishes.
\end{proof}

\begin{question}
Can we generalize Lemma \ref{stability-lemma} for $n>2$?
\end{question}
}

\hide{
\begin{lem}
Let $S$ be a generic K\"{a}hler $K3$ surface. Denote by $D\subset S^{[2]}$ the diagonal divisor.
Let $I\subset S^{[2]}\times S^{[2]}$ be the incidence variety, consisting of pairs of
length two subschemes with non-disjoint supports.
\begin{enumerate}
\item
\label{lemma-item-two-dimensional-subvarieties}
Let $Z$ be a reduced and  irreducible two dimensional subscheme of
$S^{[2]}$. Then there exists a point $x\in S$, such that $Z$ is the
locus $\hat{S}_x$ of ideal sheaves of length $2$ subschemes of $S$,
which support contains the point $x$. 
\item
\label{lemma-item-six-dimensional-subvarieties}
Let $Z$ be a reduced and  irreducible six dimensional subscheme of
$S^{[2]}\times S^{[2]}$. Then $Z$ is either $I$, 
$D\times D$, or the product of one of the $S^{[2]}$ factors with a two-dimensional subscheme of the other.
\end{enumerate}
\end{lem}

\begin{proof}
(\ref{lemma-item-two-dimensional-subvarieties}) 
Let $Z$ be a $2$-dimensional subvariety of $S^{[2]}$. 
If $Z$ is contained in $D$, then $\pi:D\rightarrow S$ must map $Z$ onto $S$,
since $S$ does not contain any curves. Let $f:Z\rightarrow S$ be the restriction of $\pi$.
Then $f$ must be \'{e}tale over the complement of finitely many points, for the same reason.
This complement is simply connected and $f$ is proper. Hence, $Z$ must be a section
of $\pi$. But the tangent bundle of $S$ is slope-stable, hence does not have any 
rank $1$ subsheaf. A contradition.

We may thus assume that $Z$ is not contained in $D$. Hence $Z$ is the image of 
a $2$-dimensional subvariety $\widetilde{Z}$ of $S\times S$, which is not the diagonal. 
The projection of $\widetilde{Z}$ to an $S$ factor is either zero-dimensional,
or the whole of $S$, since $S$ does not contain any curve. Hence $\widetilde{Z}$
is of the form $\{x\}\times S$ or $S\times \{x\}$, for some point $x$ of $S$.

(\ref{lemma-item-six-dimensional-subvarieties})
Let $\pi_i:S^{[2]}\times S^{[2]}\rightarrow S^{[2]}$, $i=1,2$, be the projection.
If $\pi_i(Z)$ is two-dimensional, we are done. 
If $\pi_i(Z)$ is a divisor, we are done as well.
We may thus assume that $\pi_i(Z)$ is
the whole of $S^{[2]}$. 
Let $f_i:Z\rightarrow S^{(2)}$ the restriction of $\pi_i$ composed with the 
Hilbert Chow morphism $S^{[2]}\rightarrow S^{(2)}$.
Then the generic fiber of $f_i$ is two dimensional
(not necessarily integral). If the generic fiber is a constant subvariety, we are done.
Otherwise, $Z$ is the image of a subvariety of $S^{[2]}\times [S\times S]$,
and the generic fiber of the  projection to $S\times S$ is integral (use the Stein factorization,
the fact that $S\times S$ is simply connected, and that $S\times S$
does not contain any divisor).
The statement now follows easily from part \ref{lemma-item-two-dimensional-subvarieties}.
\end{proof}
}
\hide{
\section{Isogenies}

The locally free case of Theorems \ref{thm-existence} and \ref{main-thm} admits a natural generalization.
We state it here for the case $d=2$.
Fix lattices $\Lambda_1$, $\Lambda_2$ and a rational isometry 
$\phi:\Lambda_{1,\RationalNumbers}\rightarrow \Lambda_{2,\RationalNumbers}$, where 
$\Lambda_{i,\RationalNumbers}:=\Lambda_i\otimes_\Integers\RationalNumbers$.
Assume that $\Lambda_1\cap\phi^{-1}(\Lambda_2)$ has finite index in $\Lambda_1$. Then
$\Lambda_2\cap\phi(\Lambda_1)$ has finite index in $\Lambda_2$.
Let $\fM_{\Lambda_i}$ be the moduli space of marked pairs,  and 
$P_i:\fM_{\Lambda_i}\rightarrow \Omega_{\Lambda_i}$ the period maps, $i=1,2$.
The isometry $\phi$ induces an isomorphism $\phi:\Omega_{\Lambda_1}\rightarrow \Omega_{\Lambda_2}$.

\begin{condition}
\label{condition-for-quintuples}
Let $(X_i,\eta_i)$ be a marked pair in $\fM_{\Lambda_i}$, $i=1,2$, and $E$ a rank $r$ locally free sheaf 
of Azumaya algebras over $X_1\times X_2$. We require the quintuple $(X_1,\eta_1,X_2,\eta_2,E)$ to
have the following properties.
\begin{enumerate}
\item
\label{cond-item-same-period}
$\phi(P_1(X_1,\eta_1))=P_2(X_2,\eta_2)$.
\item
\label{cond-item-rigidity-of-E-in-quituple}
$H^1(E)=0.$
\item
\label{cond-item-same-kahler-class}
$E$ is $(\omega_1,\omega_2)$-slope-stable with respect to some pair of K\"{a}hler classes $\omega_i$ on $X_i$,
such that $\phi(\eta_1(\omega_1))=\eta_2(\omega_2)$.
\item
\label{condition-item-G-invariance-of-c-2-E}
Set $\varphi:=\eta_2^{-1}\phi\eta_1:H^2(X_1,\RationalNumbers)\rightarrow H^2(X_2,\RationalNumbers)$.
Let $G_\varphi\subset Mon(X_1)\times Mon(X_2)$ be the following subgroup.
\[
G_\varphi=\{
(g_1,g_2) \ : \ \varphi \bar{g}_1=\bar{g}_2\varphi
\},
\]
where $\bar{g}_i$ is the restriction of $g_i$ to an isometry of $H^2(X_i,\Integers)$.
Then $c_2(E)$ is invariant with respect to a finite index subgroup of $G_\varphi$.
\item
The characteristic class of $E$ in $H^2(X_1\times X_2,\mu_r)$ has order $r$.
\end{enumerate}
\end{condition}

Let $(X_1,\eta_1,X_2,\eta_2,E)$ be a quintuple satisfying Condition \ref{condition-for-quintuples}.
Let $\pi_i:\X_i\rightarrow \PP^1_{\omega_i}$ be the twistor families associated to the K\"{a}hler classes
$\omega_i$. The base $\PP^1_{\omega_i}$ naturally embeds in the period domain $\Omega_{\Lambda_i}$
as the intersection of the latter with the plane spanned by the $\eta_i$ image of 
the set $\{\omega_i, \mbox{Re}(\sigma_i), \mbox{Im}(\sigma_i)\}$, where $\sigma_i$ is a non-zero class in
$H^{2,0}(X_i)$.
The isomorphism $\phi:\Omega_{\Lambda_1}\rightarrow \Omega_{\Lambda_2}$ maps 
$\PP^1_{\omega_1}$ isomorphically onto $\PP^1_{\omega_2}$,
by Conditions \ref{condition-for-quintuples} (\ref{cond-item-same-period}) and (\ref{cond-item-same-kahler-class}).
The sheaf $E$ is $(\omega_1,\omega_2)$-hyperholomorphic, and it deforms 
to a holomorphic locally free sheaf of Azumaya algebras over the fiber product
$\X_1\times_{\PP^1_{\omega_2}}\X_2$, where we consider the map 
$\phi\circ \pi_1:\X_1\rightarrow \PP^1_{\omega_2}$ from the first factor.

 Let $\Omega_\phi\subset \Omega_{\Lambda_1}\times\Omega_{\Lambda_1}$ be the graph of the isomorphism
 $\phi$. Let $\fM_\phi\subset \fM_{\Omega_1}\times \fM_{\Omega_1}$ be the subset consisting of
 isomorphism classes of pairs 
 $((X_1,\eta_1),(X_2,\eta_2))$ satisfying Condition \ref{condition-for-quintuples}
 (\ref{cond-item-same-period}) and $\phi(\eta_1(\omega_1))=\eta_2(\omega_2)$ 
 for some K\"{a}hler classes $\omega_i$ on $X_i$. 
 Denote by 
 \[
 P:\fM_\phi\rightarrow \Omega_\phi
 \]
 the restriction of $(P_1,P_2):\fM_{\Lambda_1}\times \fM_{\Lambda_2}\rightarrow 
 \Omega_{\Lambda_1}\times\Omega_{\Lambda_1}$. 
 Note that the above period map $P$ is surjective, by Verbitsky's Global Torelli theorem 
 \cite{huybrechts-torelli,verbitsky-torelli} and the bijection between the set of points in the fiber of $P_i$ 
 and K\"{a}hler type chambers of the positive cone given in
 \cite[Theorem 5.16]{markman-survey}. Furthermore, the fiber of $P$ over a point $p\in \Omega_\phi$
 consists of a single separated point of $\fM_\phi$, whenever $p$ 
 corresponds to periods with trivial Picard number, or with rank one Picard groups generated
 by non-negative classes, by the global Torelli theorem.
 
 \begin{thm}
 \label{thm-quintuples}
 \begin{enumerate}
 \item
 There exists a coarse moduli space $\tfM_\phi$, which is a complex non-Hausdorff manifold,  
 parametrizing isomorphism classes of quintuples
 $(X_1,\eta_1,X_2,\eta_2,E)$ satisfying Condition \ref{condition-for-quintuples}.
 \item
 \label{thm-item-surjectivity-from-quintumples-to-quadruples}
 Let $\tfM^0_\phi$ be a connected component of $\tfM_\phi$ and let 
 $\fM^0_\phi$ the connected component of $\fM_\phi$ containing the image of $\tfM^0_\phi$
 via the forgetful morphism
 $f:\tfM^0_\phi\rightarrow \fM_\phi$, forgetting the sheaf $E$. Then $f$ maps $\tfM^0_\phi$
 onto $\fM^0_\phi$. 
 \item
 The morphism $\widetilde{P}:=P\circ f:\tfM^0_\phi\rightarrow \Omega_\phi$ is surjective. Its 
 fiber over a point $p\in \Omega_\phi$ consists of a single separated point, whenever $p$ corresponds to periods
 of Hodge structures with trivial Neron-Severi groups, or cyclic 
 Neron-Severi groups generated by classes of non-negative self-intersection.
 \item
 If $(X_1,\eta_1,X_2,\eta_2,E)$ and $(X_1,\eta_1,X_2,\eta_2,E')$ are two points in the fiber of $f$
 and both $E$ and $E'$ are $(\omega_1,\omega_2)$-stable
 with respect to the same K\"{a}hler classes in Condition
 \ref{condition-for-quintuples} (\ref{cond-item-same-kahler-class}),
then $E$ is isomorphic to $E'$.
\end{enumerate}
 \end{thm}

Let $O(\Lambda_{1,\RationalNumbers},\Lambda_{2,\RationalNumbers})$ be the set 
of rational isometries $\phi$ as above.
Then $O(\Lambda_1)\times O(\Lambda_2)$ acts naturally 
on the disjoint union 
${\displaystyle \bigcup_{\phi\in O(\Lambda_{1,\RationalNumbers},\Lambda_{2,\RationalNumbers})}\tfM_\phi}$
by 
\[
(g_1,g_2)(X_1,\eta_1,X_2,\eta_2,E)=(X_1,g_1\eta_1,X_2,g_2\eta_2,E),
\]
sending $\tfM_\phi$ to $\tfM_{g_2\phi g_1^{-1}}.$

\begin{example}
\label{example-rational-Hodge-isometries}
Let $M$ be a two dimensional projective moduli space of rank $r$ stable vector bundles over a 
projective $K3$-surface $S$ with a cyclic Picard group.
Let $F$ be a universal  twisted bundle over $S\times M$.
Assume that the Brauer class $\theta$ of $F$ on $M$ has order $r$. Set $E:=\SheafEnd(F)$.
Let $\Lambda_1=\Lambda_2=\Lambda$ be the $K3$-lattice.
Choose marking $\eta_1$ for $S$ and $\eta_2$ for $M$.
Let $\varphi:H^2(S,\RationalNumbers)\rightarrow H^2(M,\RationalNumbers)$
be the isometry induced by the correspondence
\begin{equation}
\label{eq-correspondence-of-E}
\frac{-1}{2r}p_S^*(\sqrt{td_S})c_2(E)p_M^*(\sqrt{td_M}),
\end{equation}
where $p_S$ and $p_M$ are the projections from $S\times M$ onto the corresponding factors. 
Set $\phi:=\eta_2\varphi\eta_1^{-1}:\Lambda\rightarrow\Lambda$.
The monodromy invariance of $c_2(E)$ in Condition
\ref{condition-for-quintuples} (\ref{condition-item-G-invariance-of-c-2-E})
is a tautology, due to the above relationship between $c_2(E)$ and $\phi$.
The universal bundle $F$ is the kernel for a Fourier-Mukai equivalence of the bounded derived categories
of (twisted) coherent sheaves 
$\Phi_F:D^b(S)\rightarrow D^b(M,\theta)$. We get an equivalence 
$1\times \Phi_F:D^b(S\times S)\rightarrow D^b(S\times M,p_M^*\theta)$, and
$F$ is the image of the structure sheaf $\StructureSheaf{\Delta}$ of the diagonal. The rigidity of $E$
in Condition \ref{condition-for-quintuples} (\ref{cond-item-rigidity-of-E-in-quituple}) 
follows from the vanishing of $\Ext^1(\StructureSheaf{\Delta},\StructureSheaf{\Delta})$.
$F$ is slope-stable with respect to every K\"{a}hler class, 
by Proposition \ref{prop-maximally-twisted-sheaf-is-slope-stable}.
Condition \ref{condition-for-quintuples} (\ref{cond-item-same-kahler-class}) is satisfied, since
both $S$ and $M$ are projective with Picard number $1$, and so their K\"{a}hler cones are 
equal to their positive cones. We conclude 
that $(S,\eta_1,M,\eta_2,E)$ is a quintuple in $\tfM_\phi$. The connected component $\tfM^0_\phi$
containing $(S,\eta_1,M,\eta_2,E)$
parametrizes pairs of marked $K3$ surfaces with a rational Hodge isometry, which is induced by a
characteristic class. In particular, this rational Hodge isometry is induced by an algebraic correspondence,
whenever the $K3$ surfaces are projective.
\end{example}

Let $\Lambda$ be the $K3$ lattice and $O(\Lambda_\RationalNumbers)$ the group of rational isometries.
Let
\[
O_{\rm Hdg}(\Lambda_\RationalNumbers)
\]
be the subgroup of $O(\Lambda_\RationalNumbers)$ generated by $O(\Lambda)$ and 
isometries $\phi$ for which the moduli space $\tfM_\phi$ has a non-empty component 
$\tfM^0_\phi$ consisting of quintuples $(X_1,\eta_1,X_2,\eta_2,E)$, whith $X_i$ a $K3$ surface and the sheaf $E$
is such that $\phi$ is equal to the correspondence (\ref{eq-correspondence-of-E}). 

\begin{question}
How big is the subgroup $O_{\rm Hdg}(\Lambda_\RationalNumbers)$ of $O(\Lambda_\RationalNumbers)$?
\end{question}

The surjectivity statement in Theorem \ref{thm-quintuples} (\ref{thm-item-surjectivity-from-quintumples-to-quadruples}) 
means that a proof of an equality $O_{\rm Hdg}(\Lambda_\RationalNumbers)=O(\Lambda_\RationalNumbers)$
would provide a proof of a conjecture of Shafarevich about the algebraicity of rational Hodge isometries 
between projective $K3$ surfaces. Foundational work on this problem was carried out by S. Mukai \cite{mukai-hodge}.
}


\begin{thebibliography}{B-N-R}

\bibitem[Art]{Art} 
Artin, M.: {\em On the solutions of analytic equations.}  
Invent. Math. 5, p. 277--291 (1968).

\bibitem[Ba]{ballico} Ballico, E.:
{\em Rings of holomorphic functions and UFD.\/}
International J. of Pure and Appl. Math.  20, no. 4 (2005), 479--482.

\bibitem[Be]{beauville-varieties-with-zero-c-1}
Beauville, A.: {\em Varietes K\"ahleriennes dont la premiere classe de Chern 
est nulle.}  J. Diff. Geom. 18, p. 755--782 (1983).

\bibitem[BHPV]{BPV} Barth, W., Hulek, K.,  Peters, C., Van de Ven, A.:
{\em Compact complex surfaces.\/} Second edition. 
Ergebnisse der Mathematik und ihrer Grenzgebiete. 3. Folge. 
A Series of Modern Surveys in Mathematics, 4. 
Springer-Verlag, Berlin, 2004.

\bibitem[Bi]{Bi} Bingener, J.:
{\em Darstellbarkeitskriterien f\"{u}r analytische Funktoren,} 
(German) Ann. Sci. ƒcole Norm. Sup. (4) 13 (1980), no. 3, 317--347.

\bibitem[BR]{burns-rapoport}
Burns, D., Rapoport, M.:
{\em On the Torelli problem for k\"{a}hlerian $K3$ surfaces.\/}
Ann. Sci. \'{E}cole Norm. Sup. (4) 8 (1975), no. 2, 235--273. 


\bibitem[BuF]{Buf} Buchweiz, R., Flenner, H.:
{\em A semiregularity map for modules and applications to deformations.\/} 
Compos. Math. 137, 135--210 (2003).


\bibitem[Cal]{caldararu} C\u{a}ld\u{a}raru, A.:
{\em Derived categories of twisted sheaves on
Calabi-Yau manifolds.\/} Thesis, Cornell Univ., May 2000.

\bibitem[Car]{cartan} Cartan, H.: {\em Espace fibr\'{e}s analytiques.\/}
S\'{e}minaire N. Bourbaki, 1956-1958, exp. no 137, p. 7--18.

\bibitem[CM]{charles-markman} Charles, F., Markman, E.:
{\em The Standard Conjectures for holomorphic symplectic varieties deformation equivalent to Hilbert schemes of 
$K3$ surfaces.\/}
Preprint arXiv:1009.0413.

\bibitem[D]{douady} Douady, A.:
{\em Prolongement de faisceaux analytic coh\'{e}rents.\/}
S\'{e}minaire Bourbaki, 1969-1970, exp. no 366, p. 39--54.

\bibitem[F]{friedman} Friedman, R.: 
{\em Algebraic surfaces and holomorphic vector bundles.\/}
Universitext, Springer-Verlag (1998).

\bibitem[FK]{flenner-kosarew}Flenner, H., Kosarew, S.:
{\em On locally trivial deformations.\/}
Publ. RIMS, Kyoto Univ. 23 (1987), 627-665.

\bibitem[Ha0]{AG}
Hartshorne, R.: {\em Algebraic Geometry.\/}
Springer GTM (1977).

\bibitem[Ha1]{defo}
Hartshorne, R.: {\em Deformation Theory.\/}  
Springer GTM (2010).

\bibitem[Ha2]{hartshorne-stable-reflexive} 
Hartshorne, R.: {\em Stable reflexive sheaves.\/}  
Math. Ann. 254, 121--176 (1980).

\bibitem[Ha3]{hartshorne-local-cohomology}
Hartshorne, R.: {\em Local Cohomology.}
Lecture Notes in math. 41 (1967).

\bibitem[HH]{cusp}
Hartmann, H.: 
{\em Cusps of the K\"{a}hler moduli space and stability conditions on K3 surfaces.\/}
Math. Ann. 354 (2012), no. 1, 1--42.

\bibitem[Hi]{hitchin} Hitchin, N.:
{\em Lectures on special lagrangian manifolds.\/}
Winter School on Mirror Symmetry, Vector Bundles and Lagrangian Submanifolds (Cambridge, MA, 1999), 151--182, AMS/IP Stud. Adv. Math., 23, Amer. Math. Soc., 2001. Also available
as arXiv:math/9907034.

\bibitem[HKLR]{HKLR} Hitchin, N. J., Karlhede, A., Lindstr\"{o}m, U., and Ro\v{c}ek, M.:
{\em Hyper-K\"{a}hler metrics and supersymmetry.\/}
Comm. Math. Phys. Volume 108, Number 4 (1987), 535--589. 

\bibitem[HL]{huybrechts-lehn-book}
Huybrechts, D, Lehn, M.: 
{\em The geometry of moduli spaces of sheaves.\/} 
Aspects of Mathematics, E31. Friedr. Vieweg \& Sohn, Braunschweig, 1997.

\bibitem[HSc]{huybrechts-schroer} Huybrechts, D., Schr\"{o}er, S.:
{\em The Brauer group of analytic $K3$ surfaces.\/}
Int. Math. Res. Not. 2003, no. 50, 2687--2698.

\bibitem[HSt]{hilton-stammbach} Hilton, P. J., Stammbach, U.:
{\em A course in Homological Algebra,\/} Graduate Texts in Math. 4, Springer-Verlag 1970.

\bibitem[Hu1]{huybrects-basic-results}
Huybrechts, D.: 
{\em Compact Hyperk\"{a}hler Manifolds: Basic results.\/}
Invent. Math. 135 (1999), no. 1, 63-113 and
Erratum: Invent. Math. 152 (2003), no. 1, 209--212. 

\bibitem[Hu2]{huybrechts-torelli} Huybrechts, D.: {\em 
A global Torelli Theorem for hyperk\"{a}hler manifolds,\/} after M. Verbitsky.
S\'{e}minaire Bourbaki, preprint 2010.

\bibitem[KM]{flips}
Kollar, J., Mori, Sh.:
{\em Classification of  three-dimesnional flips.\/}
Jour. Amer. Math. Soc. 5 (1992), no. 3, 533-703.

\bibitem[KnM]{det} Knudsen, F., Mumford, D.: 
{\em The projectivity of the moduli space of stable curves. I. Preliminaries on``det'' and ``Div''},
Math. Scand. 39 (1976), no. 1, 19--55.

\bibitem[KO]{kosarew-okonek}
Kosarew, S., Okonek, Ch.:
{\em Global moduli spaces and simple holomorphic bundles.\/}
Publ. RIMS, Kyoto Univ. 25 (1989), 1--19.


\bibitem[Lieb]{thesis} Lieblich, M.: {\em Moduli of twisted sheaves and
generalized Azumaya algebras.\/} 
M.I.T. PhD Thesis (2004).




\bibitem[Ma1]{markman-diagonal} Markman, E.:
{\em Generators of the cohomology ring of moduli spaces of sheaves on 
symplectic surfaces.\/} J. Reine Angew. Math. 544 (2002), 61--82. 

\bibitem[Ma2]{markman-monodromy-I} Markman, E.:
{\em On the monodromy of moduli spaces of sheaves on 
K3 surfaces.\/}
J. of Alg. Geom. {\bf 17}  (2008), 29--99. 


\bibitem[Ma3]{markman-integral-generators} Markman, E.:
{\em Integral generators for the cohomology ring of moduli spaces of 
sheaves over Poisson surfaces.\/} 
Adv. in Math. 208 (2007) 622--646.

\bibitem[Ma4]{markman-integral-constraints} Markman, E.:
{\em Integral constraints on 
the monodromy group of the hyperk\"{a}hler 
resolution of a symmetric product of a $K3$ 
surface.\/} 
Internat. J. of Math. Vol. 21, 
Issue: 2(2010) pp. 169--223.  

\bibitem[Ma5]{markman-hodge} Markman, E.:
{\em The Beauville-Bogomolov class as a
characteristic class.\/} Electronic preprint arXiv:1105.3223.

\bibitem[Ma6]{markman-appendix} Markman, E.:
{\em Appendix to ``The Beauville-Bogomolov class as a
characteristic class''.\/} Preprint, May 2010,
http://www.math.umass.edu/$\sim$markman/

\bibitem[Ma7]{markman-survey} Markman, E.:
{\em A survey of Torelli and monodromy results for holomorphic-symplectic varieties.\/}
In  ``Complex and Differential Geometry'', W. Ebeling et. al. (eds.),
Springers Proceedings in Math. 8, (2011), pp 257--323.
Available at arXiv:1101.4606.

\bibitem[Ma8]{markman-stability} Markman, E.:
{\em Stability of an Azumaya algebra over the Cartesian square of the Hilbert scheme of
$n$ points of a generic $K3$ surface.}\/ Preprint 2012. 

\bibitem[MM1]{MM} Markman, E., Mehrotra, S.: 
{\em Integral transforms and deformations of $K3$ surfaces.\/} In preparation.

\bibitem[MM2]{MM-density} Markman, E., Mehrotra, S.: 
{\em Hilbert schemes of K3 surfaces are dense in moduli.\/} 
Electronic preprint arXiv:1201.0031. 

\bibitem[Mat]{matsumura} Matsumura, H.:
{\em Commutative Algebra.\/}
W. A. Benjamin Co., New York (1970).

\bibitem[Mu]{mukai-hodge} Mukai, S.:
{\em On the moduli space of bundles on K3 surfaces I},
Vector bundles on algebraic varieties,
Proc. Bombay Conference, 1984, Tata Institute of Fundamental Research Studies,
no. 11, Oxford University Press, 1987, pp. 341--413.

\bibitem[OTT]{grr-complex}
O'Brian, N. R.; Toledo, D.; Tong, Y. L.:
{\em Grothendieck-Riemann-Roch for complex manifolds.\/}
Bull. Amer. Math. Soc. (N.S.) 5 (1981), no. 2, 182--184.

\bibitem[PS]{PS} Piatetski-Shapiro, I. I., Shafarevich, I. R.:
{\em Torelli's theorem for algebraic surfaces of type $K3$.\/}  (Russian)
Izv. Akad. Nauk SSSR Ser. Mat. 35 1971 530--572. 

\bibitem[Ran]{ran}
Ran, Z.:
{\em Deformations of Maps.\/}
 Algebraic curves and projective geometry (Trento, 1988), 246Ð253, 
 LNM, 1389, Springer, Berlin, 1989.

\bibitem[RRV]{ramis-ruget-verdier}
Ramis, J. P.; Ruget, G.; Verdier, J. L:
{\em Dualit\'e relative en g\'{e}om\'{e}trie analytique complexe.\/}
Inventiones. math. 13 (1971), 261--283.

\bibitem[Sern]{sernesi}
Sernesi, E.:
{\em Deformations of Schemes.\/}
Springer, 2008.

\bibitem[Serre]{serre}
Serre, J. P.:
{\em Prolongement de faisceaux analytique coh\'{e}rents.\/}
Ann. inst. Fourier, tome 16, no 1, (1966), pp. 363-374.


\bibitem[Siu]{siu-local-cohomology} Siu, Y.:
{\em Analytic sheaves of local cohomology.\/}
Trans. Amer. Math. Soc. 148 (1970), pp. 347--366.

\bibitem[T]{trautmann} Trautmann, G.:
{\em Ein Endlichkeitssatz in der analytischen geometrie.\/}
Invent. Math. 8 (1969), 143-174.

\bibitem[Ve1]{verbitsky} Verbitsky, M.: 
{\em Mirror symmetry for hyper-K\"{a}hler manifolds.\/}
Mirror symmetry, III (Montreal, PQ, 1995), 115--156, 
AMS/IP Stud. Adv. Math., 10, Amer. Math. Soc., Providence, RI, 1999. 

\bibitem[Ve2]{verbitsky-trianalytic}  Verbitsky, M.: 
{\em Trianalytic subvarieties of the Hilbert scheme of points on a $K3$ 
surface.  \/} 
Geom. Funct. Anal. 8 (1998), no. 4, 732--782.

\bibitem[Ve3]{verbitsky-cohomology}  Verbitsky, M.: 
{\em Cohomology of compact hyper-K\"{a}hler manifolds and its 
applications.\/} 
Geom. Funct. Anal.  6  (1996),  no. 4, 601--611. 

\bibitem[Ve4]{verbitsky-hyperholomorphic} Verbitsky, M.: 
{\em Hyperholomorphic bundles over a hyper-K\"{a}hler manifold.\/}
J. Alg. Geom. 5 (1996), no. 4, 633--669.

\bibitem[Ve5]{kaledin-verbitski-book} Verbitsky, M.
{\em Hyperholomorphic sheaves and new examples of hyperkaehler manifolds,\/} 
alg-geom/9712012. 
In the book: Hyperk\"{a}hler manifolds, by Kaledin, D. and Verbitsky, M., 
Mathematical Physics (Somerville), 12. International Press, 
Somerville, MA, 1999. 

\bibitem[Ve6]{verbitsky-coherent-sheaves} Verbitsky, M.: 
{\em Coherent sheaves on general $K3$ surfaces and Tori.\/}
Pure Appl. Math. Q. 4 (2008), no. 3, part 2, 651--714.

\bibitem[Ve7]{verbitsky-torelli} Verbitsky, M.:
{\em A global Torelli theorem for hyperk\"{a}hler manifolds.\/} 
Preprint arXiv:0908.4121 v7.

\bibitem[Wa]{Wa} Wahl, J.:
{\em Equisingular deformations of normal surface singularities I.\/}
Ann. of Math. 104 (1976), no. 2, 325--356.


\bibitem[Y]{yoshioka-abelian-surface} Yoshioka, K.:
{\em 
Moduli spaces of stable sheaves on abelian surfaces. \/
}
Math. Ann. 321 (2001), no. 4, 817--884.


\end{thebibliography}
\end{document}